\documentclass[final,notitlepage,8pt]{article}
\usepackage{indentfirst}
\usepackage[normalem]{ulem}
\usepackage{float}

\usepackage{graphicx}
\usepackage{enumitem}[2011/09/28]
\setenumerate{align=left, leftmargin=0pt,labelsep=.5em, labelindent=0\parindent,listparindent=\parindent,itemindent=*}
\usepackage{array}
\usepackage{empheq}
\usepackage{natbib}
\usepackage[all]{xy}

\usepackage{exscale,relsize}
\usepackage{multirow}

\usepackage{caption}[2012/02/19]

\usepackage{longtable}

\usepackage{fouriernc}
\usepackage{fourier}
\usepackage{amssymb,bm,amsmath}
\usepackage{amsfonts}

\usepackage[OT2,T1,T2A]{fontenc}

\usepackage[english,french,german,italian,russian]{babel}
\usepackage{appendix}

\DeclareMathOperator{\D}{d}
\DeclareMathOperator{\I}{Im}
\DeclareMathOperator{\R}{Re}

\bibpunct{[}{]}{,}{n}{}{;}

\def\qed{\hfill$ \blacksquare$}
\def\eor{\hfill$ \square$}

\def\dilog{\mathrm{Li}_2\,}

\def\diTi{\mathrm{Ti}_2\,}

\newtheorem{theorem}{Theorem}[section]

\newtheorem{lemma}[theorem]{Lemma}
\newtheorem{proposition}[theorem]{Proposition}
\newtheorem{corollary}[theorem]{Corollary}

\newenvironment{proof}[1][Proof]{\begin{trivlist}
\item[\hskip \labelsep {\bfseries #1}]}{\end{trivlist}}

\newenvironment{remark}[1][Remark]{\begin{trivlist}
\item[\hskip \labelsep {\bfseries #1}]}{\end{trivlist}}

\begin{document}

\selectlanguage{english}
\title{\textsf{\textbf{Spherical Couplings and Multiple Elliptic Integrals}}}\author{Yajun Zhou
\\\begin{small}\textsc{Program in Applied and Computational Mathematics, Princeton University, Princeton, NJ 08544}\end{small}
}
\date{}
\maketitle

\def\abstractname{}
\begin{abstract}
     \noindent Various integrals over elliptic integrals are evaluated as couplings on spheres, resulting in  some  integral and series representations  for the mathematical constants $ \pi$, $G$ and $\zeta(3)$. \end{abstract}

\tableofcontents
\setcounter{section}{-1}
\section{Introduction}

Multiple elliptic integrals are integrals over expressions containing complete elliptic integrals. There are various techniques for computing multiple elliptic integrals, including the methods of  Bessel moments \cite{Bailey2008,BNSW2011,BSWZ2012}, geometric transformations of Watson type \cite{Zucker2011}, hypergeometric summations \cite{Wan2012} and Legendre function relations \cite{Zhou2013Pnu}.

In this work, we showcase several geometric methods to evaluate multiple elliptic integrals in terms of the circle constant $ \pi:=4\sum_{\ell=0}^{\infty}(-1)^{\ell}(2\ell+1)^{-1}$, Catalan's constant $ G:=\sum_{\ell=0}^{\infty}(-1)^{\ell}(2\ell+1)^{-2}$ \cite{Catalan1865,Catalan1883}, Ap\'ery's constant $ \zeta(3):=\sum_{\ell=1}^\infty\ell^{-3}$ \cite{Apery1978}, among  others. We will make extensive use of  techniques for harmonic analysis on spheres and hyperspheres \cite{SteinWeiss}, in order to provide geometric interpretations for a large assortment of multiple elliptic integrals, new and old, famous or obscure.

To flesh out, we will establish a series of identities by computing certain integrals on   $ S^2\times S^2$ (\S\ref{subsec:S2}) or $ S^3\times S^3$ (\S\ref{subsec:S3}) in two ways:  (1) applying   (ultra)spherical harmonic expansions to the integrands; and (2) introducing  (hyper)spherical coordinates with trigonometric   parametrization.
In \S\ref{sec:Beltrami_revisited}, we follow up with some new perspectives on the Beltrami transformations introduced in our earlier work   \cite{Zhou2013Pnu}, in conjunction with applications of Beltrami transformations to selected multiple elliptic integrals derived in \S\S\ref{subsec:S2}-\ref{subsec:S3}. Some integral formulae in \S\S\ref{subsec:S2}-\ref{sec:Beltrami_revisited} will pave way for our subsequent works that frequently call for multiple elliptic integral representations of the constants  $ \pi$, $G$ and $\zeta(3)$. We wrap up the main text with a collection of   hypergeometric summation formulae (\S\ref{subsec:comb}) that descend from certain integral identities  in  \S\S\ref{subsec:S2}-\ref{sec:Beltrami_revisited}. In the appendix, we tabulate all the integrals and series proved in \S\S\ref{subsec:S2}-\ref{subsec:comb}.

\section{Integrations on $ S^2\times S^2$\label{subsec:S2}}\subsection{Spherical Harmonic Expansions and Geometric Parametrizations\label{subsubsec:Ylm}}We will compute some integrals on the Cartesian products of two unit spheres $ S^2\times S^2$, using  both the spherical harmonic expansions and geometric parametrizations. By evaluating the same integral using different methods, we are able to establish some interesting identities involving Catalan's constant $G$ and   the complete elliptic integral of the first kind $ \mathbf{K}(k):=\int_0^{\pi/2}(1-k^2\sin^2\theta)^{-1/2}\D\theta$.
 \begin{lemma}[Some Integrals on $ S^2\times S^2$]\label{lm:S2}
 \begin{enumerate}[label=\emph{(\alph*)}, ref=(\alph*), widest=a]\item

 Let  $ S^2\subset\mathbb R^3$ be the unit sphere equipped with the standard surface element $ \D\sigma$ such that the surface area is $ m(S^2)=\int_{S^2}\D \sigma=4\pi$. For any unit vector $ \bm n\in S^2$, the following identities hold:\begin{align}\frac{\pi^2}{8}=\sum_{\ell=0}^{\infty}\frac{1}{(2\ell+1)^2}={}&\int_{S^2}\frac{\D \sigma_1}{4\pi}\int_{S^2}\frac{\D \sigma_{2}}{4\pi}\frac{1}{|\bm n-\bm n_1|}\frac{1}{|\bm n_{1}-\bm n_2|}\frac{1}{|\bm n_{2}-\bm n|}\notag\\={}&\int_0^1\D u\int_0^1\D v\int_0^{1}\D w\frac{1}{8\pi\sqrt{\mathstrut u(1-u)(1-u+uw)}\sqrt{\vphantom(  vw(1-v)}}=\int_0^{1}\frac{\mathbf{K}(k)}{1+k}\D k;\label{eq:a}\\G=\sum_{\ell=0}^{\infty}\frac{(-1)^{\ell}}{(2\ell+1)^2}={}&\int_{S^2}\frac{\D \sigma_1}{4\pi}\int_{S^2}\frac{\D \sigma_{2}}{4\pi}\frac{1}{|\bm n-\bm n_1|}\frac{1}{|\bm n_{1}-\bm n_2|}\frac{1}{|\bm n_{2}+\bm n|}\notag\\={}&\int_0^1\D u\int_0^1\D v\int_0^{1}\D w\frac{1}{8\pi\sqrt{\mathstrut  u(1-u)(1-u+uw)}\sqrt{\vphantom(  1-vw}\sqrt{\mathstrut w(1-v)}}=\frac{1}{\pi}\int_0^{1}\frac{\mathbf{K}(k)}{1-k}\log\frac{1}{k}\D k.\label{eq:b}\end{align}\item
 Let $ \beta\in[0,\pi)$ be a given angle, and define the Heaviside theta function as $ \theta_H(x)=(1+\frac{x}{|x|})/2, x\in \mathbb R\smallsetminus\{0\}$, then we have the following  integral formulae:

{\allowdisplaybreaks
\begin{align}\frac{\pi(\pi-\beta)}{8}=\sum _{\ell=0}^{\infty } \frac{1}{ (2 \ell+1)^2}\cos\frac{(2 \ell+1)\beta}{2}
={}&\int_{S^2}\frac{\D \sigma_1}{4\pi}\int_{S^2}\frac{\D \sigma_{2}}{4\pi}\frac{1}{|\bm n-\bm n_1|}\frac{\theta_{H}(\cos\beta-\bm n_1\cdot\bm n_2)}{\sqrt{2(\cos\beta-\bm n_1\cdot\bm n_2)}}\frac{1}{|\bm n_{2}-\bm n|}\notag\\
={}&\int_0^1\D q\int_0^1\D r\int_{0}^1\D s\frac{1}{8\pi\sqrt{\smash[b]{ (1-r)(1-s)rs(1-q)\vphantom{s^2}}}\sqrt{\smash[b]{q +\tan^2(\beta/2)}}}\notag\\={}&\frac{\pi }{8}\int_{0}^{1}\frac{\D q}{\sqrt{\smash[b]{1-q\vphantom{s^2}}}\sqrt{\smash[b]{q +\tan^2(\beta/2)}}};\label{eq:beta}\\
 \mathfrak{L}(\beta):=\sum _{\ell=0}^{\infty } \frac{(-1)^{\ell}}{ (2 \ell+1)^2}\cos\frac{(2 \ell+1)\beta}{2}={}&\int_{S^2}\frac{\D \sigma_1}{4\pi}\int_{S^2}\frac{\D \sigma_{2}}{4\pi}\frac{1}{|\bm n-\bm n_1|}\frac{\theta_{H}(\cos\beta-\bm n_1\cdot\bm n_2)}{\sqrt{2(\cos\beta-\bm n_1\cdot\bm n_2)}}\frac{1}{|\bm n_{2}+\bm n|}\notag\\
={}&\int_0^1\D q\int_0^1\D r\int_{0}^1\D s\frac{1}{8\pi\sqrt{\smash[b]{(1-r)(1-s)rs(1-q)}\vphantom{s^2}}{\sqrt{\smash[b]{\sec^2(\beta/2)-s[qr+\tan^2(\beta/2)]}}}}\notag\\
={}&\frac{\cos(\beta/2)}{2}\int_0^1\mathbf K\left( \sqrt{k^2\cos^2\frac{\beta}{2}+\sin^2\frac{\beta}{2} }\right)\D k=-\left(\cos\frac{\beta}{2}\right)\int_0^1\frac{(1+t^2)\log t}{1+2t^{2}\cos\beta+t^4}\D t\notag\\={}&\frac{\cos(\beta/2)}{2}\int_0^{\pi/2}\frac{ \theta\D \theta}{\sin\theta\sqrt{\smash[b]{1-\cos^2\theta\sin^2(\beta/2)}}}.\label{eq:beta'}\end{align}}
\\Here,   the  function\begin{align*}\mathfrak{L}(\beta)=G+\frac{1}{4}\int_0^\beta\log\tan\frac{\pi-\alpha}{4}\D\alpha\end{align*}has  special values $ \mathfrak{L}(0)=G$, $ \mathfrak{L}(2\pi/3)=2G/3$, and $\mathfrak{L}(2\pi/5)-\mathfrak{L}(4\pi/5)=2G/5$, leading to the identities\begin{align*}G={}&\frac{1}{2}\int_0^1\mathbf K(k)\D k=\frac{3}{8}\int_0^1\mathbf K\left( \frac{\sqrt{k^2+3} }{2}\right)\D k\notag\\={}&\frac{5}{16}\left[ (\sqrt{5}+1)\int_0^1\mathbf K\left( \frac{\sqrt{5}+1}{4}\sqrt{k^2+5-2\sqrt{5}} \right)\D k-(\sqrt{5}-1)\int_0^1\mathbf K\left( \frac{\sqrt{5}-1}{4}\sqrt{k^2+5+2\sqrt{5}} \right)\D k\right].\end{align*}\end{enumerate}\end{lemma}
\begin{proof}\begin{enumerate}[label=(\alph*),widest=a]\item
Using spherical coordinates for points on the unit sphere $ \bm n(\theta,\phi)=(\sin\theta\cos\phi,\sin\theta\sin\phi,\cos\theta)\in S^2$, and the familiar spherical harmonic expansion (with $P_\ell $ being the Legendre polynomials, $ Y_{\ell m}$  the spherical harmonics)\[\frac{1}{4\pi|\bm n-\bm n'|}=\sum_{\ell=0}^\infty \frac{P_\ell(\bm n\cdot\bm n')}{4\pi}=\sum_{\ell=0}^\infty\sum_{m=-\ell}^\ell\frac{\overline{Y_{\ell m}(\theta,\phi)}Y_{\ell m}(\theta',\phi')}{2\ell+1},\quad |\bm n|=|\bm n'|=1,\]we may compute\begin{align*}I_{\ref{eq:a}}:=\int_{S^2}\frac{\D \sigma_1}{4\pi}\int_{S^2}\frac{\D \sigma_{2}}{4\pi}\frac{1}{|\bm n-\bm n_1|}\frac{1}{|\bm n_{1}-\bm n_2|}\frac{1}{|\bm n_{2}-\bm n|}={}&4\pi\sum_{\ell=0}^\infty\sum_{m=-\ell}^\ell\frac{\overline{Y_{\ell m}(\theta,\phi)}Y_{\ell m}(\theta,\phi)}{(2\ell+1)^{3}}=\sum_{\ell=0}^{\infty}\frac{1}{(2\ell+1)^2};\\I_{\ref{eq:b}}:=\int_{S^2}\frac{\D \sigma_1}{4\pi}\int_{S^2}\frac{\D \sigma_{2}}{4\pi}\frac{1}{|\bm n-\bm n_1|}\frac{1}{|\bm n_{1}-\bm n_2|}\frac{1}{|\bm n_{2}+\bm n|}={}&4\pi\sum_{\ell=0}^\infty\sum_{m=-\ell}^\ell\frac{\overline{Y_{\ell m}(\theta,\phi)}Y_{\ell m}(\pi-\theta,\phi+\pi)}{(2\ell+1)^{3}}=\sum_{\ell=0}^{\infty}\frac{(-1)^{\ell}}{(2\ell+1)^2}.\end{align*}Here, we have exploited the orthonormal properties of the spherical harmonics \[ \int_0^\pi\sin\theta_{j}\D\theta_{j}\int_{-\pi}^{\pi}\D\phi_{j} \overline{Y_{\ell 'm'}(\theta_{j},\phi_{j})}Y_{\ell m}(\theta_{j},\phi_{j})=\delta_{\ell'\ell}\delta_{m'm},\quad j=1,2\] to complete the integration with respect to $ \bm n_1\in S^2$ and $ \bm n_2\in S^2$, and have noted that \[4\pi\sum_{m=-\ell}^\ell\frac{\overline{Y_{\ell m}(\theta,\phi)}Y_{\ell m}(\theta',\phi')}{2\ell+1}=P_\ell(\bm n\cdot\bm n'),\quad \forall\bm n,\bm n'\in S^2,\]which equals $1$ if we choose $ \bm n'=\bm n$ for Eq.~\ref{eq:a}, and equals $(-1)^{\ell}$ if we choose $ \bm n'=-\bm n$ for Eq.~\ref{eq:b}.

In lieu of the spherical harmonic expansion,  we may evaluate the surface integrals by variable transformations. We   take, without loss of generality, the north pole $\bm n=(0,0,1) $ on the sphere, and employ spherical coordinates for $ \bm n_1$ and $\bm n_2$, to compute\begin{align*}I_{\ref{eq:a}}={}&\int_{S^2}\frac{\D \sigma_1}{4\pi}\int_{S^2}\frac{\D \sigma_{2}}{4\pi}\frac{1}{|\bm n-\bm n_1|}\frac{1}{|\bm n_{1}-\bm n_2|}\frac{1}{|\bm n_{2}-\bm n|}\notag\\={}&\int_0^\pi\cos\frac{\theta_1}{2}\D\theta_1\int_{-\pi}^{\pi}\frac{\D\phi_1}{4\pi}\int_0^\pi\cos\frac{\theta_2}{2}\D\theta_2\int_{-\pi}^{\pi}\frac{\D\phi_2}{4\pi}\frac{1}{\sqrt{\smash[b]{2[1-\cos\theta_1\cos\theta_2-\sin\theta_1\sin\theta_2\cos(\phi_1-\phi_2)]}}}\notag\\={}&\frac{1}{4}\int_0^{\pi}\frac{\D\varphi}{\pi}\int_0^\pi\D\theta_1\int_0^\pi\D\theta_2\frac{\cos\frac{\theta_1}{2}\cos\frac{\theta_2}{2}}{\sqrt{\smash[b]{2(1-\cos\theta_1\cos\theta_2-\sin\theta_1\sin\theta_2\cos\varphi)}}}.\end{align*}If we make the substitutions $ \theta_1\mapsto\pi-\theta_1$ and $ \theta_2\mapsto\pi-\theta_2$, the denominator of the integrand will remain
unchanged, while the numerator becomes $ \sin\frac{\theta_1}{2}\sin\frac{\theta_2}{2}$ instead. By the addition theorem of cosines, we are led  to the expression\[I_{\ref{eq:a}}=\int_0^{\pi}\frac{\D\varphi}{8\pi}\int_0^\pi\D\theta_1\int_0^\pi\D\theta_2\frac{\cos\frac{\theta_1-\theta_{2}}{2}}{\sqrt{\smash[b]{2(1-\cos\theta_1\cos\theta_2-\sin\theta_1\sin\theta_2\cos\varphi)}}}.\]
Similarly, we have \[I_{\ref{eq:b}}=\int_0^{\pi}\frac{\D\varphi}{8\pi}\int_0^\pi\D\theta_1\int_0^\pi\D\theta_2\frac{\sin\frac{\theta_1+\theta_{2}}{2}}{\sqrt{\smash[b]{2(1-\cos\theta_1\cos\theta_2-\sin\theta_1\sin\theta_2\cos\varphi)}}}.\]

We then introduce the new variables (see Remark~\ref{itm:sphere_to_box} below for geometric motivations)\[u=\cos^2\frac\varphi2, \quad v=\sin^2\dfrac{\theta_1+\theta_2}{2}, \quad w=1-\dfrac{\sin\theta_1\sin\theta_2}{\sin^2\frac{\theta_1+\theta_2}{2}}=\dfrac{\sin^{2}\frac{\theta_1-\theta_2}{2}}{\sin^{2}\frac{\theta_1+\theta_2}{2}},\]with Jacobian determinant\[\left|\det\frac{\partial(u,v,w)}{\partial(\theta_1,\theta_2,\varphi)}\right|=\frac{\sin\varphi\sin|\theta_1-\theta_2|}{2}\cot\frac{\theta_1+\theta_2}{2}=2\sqrt{uw(1-u)(1-v)(1-vw)}.\]
If we swap $ \theta_1$ and $\theta_2$, or perform an equatorial reflection on both  polar angles $ \theta_1\mapsto\pi-\theta_1,\theta_2\mapsto\pi-\theta_2$, then the values of  $(u,v,w) $ will not change.
It is then clear that the region $ 0<\theta_1<\pi,0<\theta_2<\pi,0<\varphi<\pi$ is mapped to the unit box $ 0< u< 1,0< v<1,0< w<1$ exactly four folds  and we have the geometric correspondence\[\cos\frac{\theta_1-\theta_{2}}{2}=\sqrt{1-vw},\quad \sin\frac{\theta_1+\theta_{2}}{2}=\sqrt{v},\quad 1-\cos\theta_1\cos\theta_2-\sin\theta_1\sin\theta_2\cos\varphi=2v(1-u+uw). \]All this allows us to put down\begin{align*}I_{\ref{eq:a}}={}&\int_0^1\D u\int_0^1\D v\int_0^{1}\D w\frac{1}{8\pi\sqrt{\mathstrut u(1-u)(1-u+uw)}\sqrt{\mathstrut vw(1-v)}};\notag\\I_{\ref{eq:b}}={}&\int_0^1\D u\int_0^1\D v\int_0^{1}\D w\frac{1}{8\pi\sqrt{\mathstrut u(1-u)(1-u+uw)}\sqrt{\mathstrut 1-vw}\sqrt{\mathstrut w(1-v)}}.\end{align*}

Now, we note that the  complete elliptic integral of the first kind satisfies\[\mathbf{K}(k):=\int_0^{\pi/2}\frac{\D\theta}{\sqrt{1-k^{2}\sin^{2}\theta}}=\int_0^1\frac{\D u}{2\sqrt{u(1-u)(1-k^{2}u)}},\quad0\leq k<1.\] After integrating in $v$, we may reduce the triple integral representations   into expressions involving $ \mathbf{K}(\sqrt{1-\xi^2})$:\begin{align*}I_{\ref{eq:a}}={}&\frac{1}{4}\int_0^1\D v\int_0^{1}\D \xi\int_{0}^{1}\frac{\D u}{\pi\sqrt{\smash[b]{ u(1-u)(1-u+u\xi^{2})}}\sqrt{\vphantom{s^2}\mathstrut v(1-v)}}=\frac{1}{4}\int_0^{1}\D \xi\int_{0}^{1}\frac{\D u}{\sqrt{u(1-u)(1-u+u\xi^{2})}}=\frac{1}{2}\int_0^{1}\mathbf{K}\left(\sqrt{1-\xi^2}\right)\D \xi;\notag\\I_{\ref{eq:b}}={}&\frac{1}{4}\int_0^1\D v\int_0^{1}\D \xi\int_{0}^{1}\frac{\D u}{\pi\sqrt{\smash[b]{\mathstrut u(1-u)(1-u+u\xi^2)}}\sqrt{\smash[b]{\mathstrut 1-v\xi^2}}\sqrt{\vphantom{s^2} 1-v}}=\frac{1}{2\pi}\int_0^{1}\frac{\mathbf{K}(\sqrt{1-\xi^2})}{\xi}\log\frac{1+\xi}{1-\xi}\D \xi.\end{align*} Here, the complete elliptic integral in question satisfies an identity called Landen's transformation\footnote{\label{fn:Landen1}The simplest way to verify Landen's transformation \begin{align*}\mathbf{K}\left(\sqrt{1-\smash[b]{\xi^2}}\right)=\int_0^{\pi/2}\frac{\D\theta}{\sqrt{\smash[b]{1-(1-\xi^2)\sin^2\theta}}}=\int_0^{\pi/2}\frac{\D\theta}{\sqrt{\smash[b]{\cos^2\theta+\xi^2\sin^2\theta}}}=\int_{-\pi/2}^{\pi/2}\frac{\D\vartheta}{\sqrt{\smash[b]{(1+\xi)^{2}-(1-\xi)^2\sin^2\vartheta}}}=\frac{2}{1+\xi}\mathbf{K}\left( \frac{1-\xi}{1+\xi} \right)\end{align*}is to use the trigonometric substitution \[ \theta=\arctan\frac{\tan\vartheta+\sec\vartheta}{\sqrt\xi}.\]} \cite[Ref.][p.~39]{ByrdFriedman}\[\mathbf{K}\left(\sqrt{1-\smash[b]{\xi^2}}\right)=\frac{2}{1+\xi}\mathbf{K}\left( \frac{1-\xi}{1+\xi} \right),\]which is equivalent to the following observation of Gau{\ss}:\[I(a,b):=\int_0^{\pi/2}\frac{\D\theta}{\sqrt{a^2\cos^2\theta+b^2\sin^2\theta}}=I\left( \frac{a+b}{2},\sqrt{ab} \right).\]Therefore, we may simplify with a M\"obius transformation $ \eta=(1-\xi)/(1+\xi)$:\begin{align*}I_{\ref{eq:a}}=\int_0^{1}\frac{\mathbf{K}(\eta)}{1+\eta}\D\eta,\quad I_{\ref{eq:b}}=\frac{1}{\pi}\int_0^{1}\frac{\mathbf{K}(\eta)}{1-\eta}\log\frac{1}{\eta}\D\eta,\end{align*}which completes the proof of all the claimed identities in part (a).\item To prove the series representations for  Eqs.~\ref{eq:beta} and \ref{eq:beta'}, we need to  avail ourselves with the Mehler-Dirichlet formula (see item 8.927 in \cite{GradshteynRyzhik} or \S4.5.4 in \cite{MOS}): \begin{align}\label{eq:Mehler_Dirichlet}4\pi\sum _{\ell=0}^{\infty } \sum_{m=-\ell}^\ell\frac{\overline{Y_{\ell m}(\theta_{1},\phi_{1})}Y_{\ell m}(\theta_{2},\phi_{2})\cos[(\ell+\frac{1}{2})\beta]}{2\ell+1}=\sum _{\ell=0}^{\infty }P_\ell(\bm n_1\cdot\bm n_2) \cos[(\ell+\tfrac{1}{2})\beta]=\frac{\theta_{H}(\cos\beta-\bm n_1\cdot\bm n_2)}{\sqrt{2(\cos\beta-\bm n_1\cdot\bm n_2)}},\quad |\bm n_{1}|=|\bm n_{2}|=1.\end{align}

To justify the triple integral representations, we introduce the change of variables
(see Remark~\ref{itm:sphere_to_box} below for geometric backgrounds)\[q=\dfrac{\cos\beta-\cos(\theta_1+\theta_2)}{\cos\beta+1}, \quad r=\dfrac{\cos\beta-\cos\theta_1\cos\theta_2-\sin\theta_1\sin\theta_2\cos\varphi}{\cos\beta-\cos(\theta_1+\theta_2)}, \quad s=\frac{ 1-\cos(\theta_1-\theta_2)}{ 1-\cos\theta_1\cos\theta_2-\sin\theta_1\sin\theta_2\cos\varphi}\]with Jacobian determinant
\begin{align*}\left|\det\frac{\partial(q,r,s)}{\partial(\theta_1,\theta_2,\varphi)}\right|={}&\frac{2\sin\theta_1\sin\theta_2\sin(\theta_1+\theta_2)\sin(\theta_1-\theta_2)\sin\varphi}{(\cos\beta+1)[\cos\beta-\cos(\theta_1+\theta_2)](1-\cos\theta_1\cos\theta_2-\sin\theta_1\sin\theta_2\cos\varphi)}\notag\\={}&\frac{2\sqrt{\smash[b]{(1-r)(1-s)s(1-q)}\vphantom{s^2}}{\sqrt{\smash[b]{\sec^2(\beta/2)-s[qr+\tan^2(\beta/2)]}}}\sqrt{\smash[b]{q +\tan^2(\beta/2)}}}{\sqrt{\smash[b]{\mathstrut q}}}.\end{align*}As before, we may verify that the image of the region $ 0<\theta_1<\pi,0<\theta_2<\pi,0<\varphi<\pi$ covers the unit box $ 0< q< 1,0< r<1,0< s<1$ exactly four times, and there are geometric relations
\begin{align*}& \cos\frac{\theta_1-\theta_{2}}{2}=\sqrt{\sec^2\frac{\beta}{2}-s\left( qr+\tan^2\frac{\beta}{2} \right)}\cos\frac{\beta}{2},\quad \sin\frac{\theta_1+\theta_{2}}{2}=\sqrt{q+\tan^2\frac{\beta}{2}}\cos\frac{\beta}{2},\notag\\&\cos\beta-\cos\theta_1\cos\theta_2-\sin\theta_1\sin\theta_2\cos\varphi=2qr\cos^2\frac{\beta}{2}.\end{align*}Consequently, we  infer that\begin{align*}I_{\ref{eq:beta}}:=\frac{\pi(\pi-\beta)}{8}={}&\int_0^{\pi}\frac{\D\varphi}{8\pi}\int_0^\pi\D\theta_1\int_0^\pi\D\theta_2\frac{\theta_{H}(\cos\beta-\cos\theta_1\cos\theta_2-\sin\theta_1\sin\theta_2\cos\varphi)\cos\frac{\theta_1-\theta_{2}}{2}}{\sqrt{\smash[b]{2(\cos\beta-\cos\theta_1\cos\theta_2-\sin\theta_1\sin\theta_2\cos\varphi)}}}\notag\\={}&\int_0^1\D q\int_0^1\D r\int_{0}^1\D s\frac{1}{8\pi\sqrt{\smash[b]{(1-r)(1-s)rs(1-q)}\vphantom{s^2}}\sqrt{\smash[b]{q +\tan^2(\beta/2)}}}\notag\\={}&\frac{\pi }{8}\int_{0}^{1}\frac{\D q}{\sqrt{\smash[b]{1-q}\vphantom{s^2}}\sqrt{\smash[b]{q +\tan^2(\beta/2)}}};\notag\\I_{\ref{eq:beta'}}:=G+\frac{1}{4}\int_0^\beta\log\tan\frac{\pi-\alpha}{4}\D\alpha={}&\int_0^1\D q\int_0^1\D r\int_{0}^1\D s\frac{1}{8\pi\sqrt{\smash[b]{(1-r)(1-s)rs(1-q)}\vphantom{s^2}}{\sqrt{\smash[b]{\sec^2(\beta/2)-s[qr+\tan^2(\beta/2)]}}}}
\notag\\={}&\int_0^1\D q\int_0^1\D r\frac{\cos(\beta/2)}{4\pi\sqrt{\smash[b]{r(1-r)(1-q)}}}\mathbf K\left( \cos\frac{\beta}{2} \sqrt{qr+\tan^2\frac{\beta}{2}}\right).\end{align*}
In the last line,  we may replace $ r$ by $k^2\equiv \rho:=q r\in(0,q)$, to obtain\begin{align*}I_{\ref{eq:beta'}}=\frac{\cos(\beta/2)}{4\pi}\int_0^1\D \rho\int_\rho^1\frac{1}{\sqrt{\smash[b]{\rho(q-\rho)(1-q)}}}\mathbf K\left( \cos\frac{\beta}{2} \sqrt{\rho+\tan^2\frac{\beta}{2}}\right)\D q=\frac{\cos(\beta/2)}{2}\int_0^1\mathbf K\left( \sqrt{k^2\cos^2\frac{\beta}{2}+\sin^2\frac{\beta}{2}} \right)\D k.\end{align*}Now, with the integral representation for the complete elliptic integral of the first kind, we have \begin{align*}I_{\ref{eq:beta'}}={}&\frac{\cos(\beta/2)}{2}\int_0^1\D t\int_0^1\mathbf \D k\frac{2}{\sqrt{\smash[b]{(1+t^2)^{2}-4t^{2}[k^2\cos^2(\beta/2)+\sin^2(\beta/2)]}}}\notag\\={}&\frac{1}{2}\int_0^1\arcsin\frac{2 t \cos (\beta/2)}{\sqrt{\smash[b]{1+2 t^2 \cos \beta +t^4}}}\frac{\D t}{t}=-\left(\cos\frac{\beta}{2}\right)\int_0^1\frac{(1+t^2)\log t}{1+2t^{2}\cos\beta+t^4}\D t,\end{align*}after integration by parts in $t$.
The last line of Eq.~\ref{eq:beta'} is evident once we can verify the integral  formula:\begin{align}\mathbf K\left(\sqrt{k^{2}\cos^2\phi+\sin^2\phi}\right)=\int_0^{\pi/2}\frac{ \D \theta}{\sqrt{\vphantom{\sin^2\theta}1-k^2\cos^2\theta}\sqrt{\smash[b]{1-\sin^2\theta\sin^2\phi}}},\quad 0\leq k<1,0\leq\phi<\pi/2.\label{eq:K_Pyth}\end{align}To show this, we pick  an arbitrary non-negative integer  $n$, and adapt the foregoing proof   as\begin{align*}&\int_0^1\,\mathbf K\left(\sqrt{k^{2}\cos^2{\phi}+\sin^2 {\phi}}\right)k^{n}\D k=\int_0^1\D q\int_0^1\D r\int_{0}^1\D s\frac{(qr)^{n/2}}{4\pi\sqrt{\smash[b]{(1-r)(1-s)rs(1-q)}\vphantom{s^2}}{\sqrt{\smash[b]{1-s( qr\cos^2\phi+\sin^2\phi )}}}}.\end{align*}Such a transformation is valid so long as the integrals   converge on both sides of the equation above.  Then, we take a ``box-to-box transformation'' (cf. Remark~\ref{itm:b2b} below)\begin{align*}\mathcal U=\frac{1-q}{1-qr},\quad \mathcal V=s,\quad \mathcal W=qr\quad\Longleftrightarrow\quad q=1-\mathcal U+\mathcal U\mathcal W ,\quad r=\frac{\mathcal W} {1-\mathcal U+\mathcal U\mathcal W},\quad s=\mathcal V,\end{align*}with Jacobian determinant\begin{align*}\left\vert \det\frac{\partial(q,r,s)}{\partial(\mathcal U,\mathcal V,\mathcal W)}\right\vert=\frac{1-\mathcal W} {1-\mathcal U+\mathcal U\mathcal W},\end{align*}which results in\begin{align*}&\int_0^1\mathbf K\left(\sqrt{k^{2}\cos^2\phi+\sin^2\phi}\right)k^{n}\D k\notag=\int_0^1\D\mathcal U\int_0^1\D\mathcal V\int_0^1\D\mathcal W\, \frac{\mathcal W^{n/2}}{4\pi\sqrt{\smash[b]{\mathcal U\mathcal V\mathcal W(1-\mathcal U)(1-\mathcal V)}   \vphantom{s^2}}\sqrt{\smash[b]{1-\mathcal V( \mathcal W\cos^2\phi+\sin^2\phi )}}}\notag\\={}&\int_0^1\D\mathcal V\int_0^1\D\mathcal W\, \frac{\mathcal W^{n/2}}{4\sqrt{\smash[b]{\mathcal V\mathcal W(1-\mathcal V)}\vphantom{s^2}}\sqrt{\smash[b]{1-\mathcal V\mathcal W-\mathcal V(1- \mathcal W)\sin^2\phi }}}=\int_0^{\pi/2}\D \theta\int_0^1\D k\frac{ k^{n}}{\sqrt{\smash[b]{1-k^2\cos^2\theta}}\sqrt{\smash[b]{1-\sin^2\theta\sin^2\phi}}}.\end{align*}Here,  we have set $\mathcal V=\sin^2\theta/(1-k^2\cos^2\theta),\mathcal W=k^2 $ in the last step. As we run through all the non-negative integer powers   $n=0,1,2,\dots$, we see that Eq.~\ref{eq:K_Pyth} must hold.

The special values taken by the function $\mathfrak{L}(\beta) $ result from the following integral relations for Catalan's constant $G$ \cite{Ramanujan1915,Bradley1999}:\[-2G=2\int_0^{\pi/4}\log\tan\theta\D\theta=3\int_0^{\pi/12}\log\tan\theta\D\theta=5\int^{3\pi/20}_{\pi/20}\log\tan\theta\D\theta,\]which in turn, are consequences of  factorizations for $ \tan3\theta$ and $ \tan5\theta$ into products of tangents.\qed\end{enumerate} \end{proof}
\begin{remark}
\begin{enumerate}[label=(\arabic*)]
\item\label{itm:sphere_to_box}

In the proof of the lemma above, we introduced the following change of variables \begin{align*}\left\{ \begin{array}{lll}
u=\cos^2\dfrac{\varphi}{2}, & v=\sin^2\dfrac{\theta_1+\theta_2}{2}, & w=1-\dfrac{\sin\theta_1\sin\theta_2}{\sin^2\frac{\theta_1+\theta_2}{2}}=\dfrac{\sin^{2}\frac{\theta_1-\theta_2}{2}}{\sin^{2}\frac{\theta_1+\theta_2}{2}}; \vspace{.75em}\\
q=\dfrac{\cos\beta-\cos(\theta_1+\theta_2)}{\cos\beta+1}, & r=\dfrac{\cos\beta-\cos\theta_1\cos\theta_2-\sin\theta_1\sin\theta_2\cos\varphi}{\cos\beta-\cos(\theta_1+\theta_2)}, & s=\dfrac{1-\cos(\theta_1-\theta_2)}{1-\cos\theta_1\cos\theta_2-\sin\theta_1\sin\theta_2\cos\varphi}
\end{array}\right. \end{align*}as a wholesale argument. Now we illustrate the line of thoughts that underlay these geometric transformations.

Taking the integral  \begin{align*}I_{\ref{eq:a}}={}&\int_0^{\pi}\frac{\D\varphi}{8\pi}\int_0^\pi\D\theta_1\int_0^\pi\D\theta_2\frac{\cos\frac{\theta_1-\theta_{2}}{2}}{\sqrt{2[1-\cos\theta_1\cos\theta_2-\sin\theta_1\sin\theta_2\cos\varphi]}}\intertext{as an example, we may naturally reparametrize with $ u=\cos^2\frac{\varphi}{2}$, which brings us the expression}I_{\ref{eq:a}}={}&\frac{1}{16}\int_0^\pi\D\theta_1\int_0^\pi\D\theta_2\int_{0}^{1}\frac{\cos\frac{\theta_1-\theta_{2}}{2}\D u}{\pi\sqrt{u(1-u)(\sin^2\frac{\theta_1+\theta_{2}}{2}-u\sin\theta_1\sin\theta_2)}}.\end{align*}We now introduce the new variables  \[V=\sin\frac{\theta_1+\theta_{2}}{2},\quad W=\sin\theta_1\sin\theta_2,\]with Jacobian determinant \begin{align*}\left|\det\frac{\partial(V,W)}{\partial(\theta_1,\theta_2)}\right|=\left\vert \det\begin{pmatrix}\frac12\cos\frac{\theta_1+\theta_{2}}{2} & \frac12\cos\frac{\theta_1+\theta_{2}}{2}\\
\cos\theta_1\sin\theta_2 & \sin\theta_1\cos\theta_2 \\
\end{pmatrix} \right\vert
=\frac{1}{2}\left\vert \cos\frac{\theta_1+\theta_{2}}{2}\sin(\theta_1-\theta_2) \right\vert=\sqrt{1-V^2}\sqrt{V^{2}-W}\sqrt{1+W-V^2}.\end{align*}Such a map sends the square region $ 0<\theta_1<\pi,0<\theta_2<\pi$ to the parabolic triangle $0< V<1,0< W< V^2 $ exactly four folds, because the values of $V$ and $W$ will remain intact if we trade the point $ (\theta_1,\theta_2)$ for $ (\theta_2,\theta_1)$ or replace the point $ (\theta_1,\theta_2)$ with $ (\pi-\theta_1,\pi-\theta_2)$.
This allows us to further transform  $I_{\ref{eq:a} }$ into\begin{align*}
I_{\ref{eq:a}}
={}&4\times\frac{1}{16}\int_0^1\D V
\int_0^{V^{2}}\D W\int_{0}^{1}
\frac{\D u}{\pi\sqrt{1-V^2}\sqrt{V^{2}-W}
\sqrt{u(1-u)(V^{2}-uW)}}\notag\\
={}&\frac{1}{8}\int_0^1\D v\int_0^{1}\D w\int_{0}^{1}
\frac{\D u}{\pi\sqrt{u(1-u)(1-u+uw)}\sqrt{vw(1-v)}}\end{align*}where we have set $ v= V^2,w=(V^{2}-W)/V^2$ in the last step.

If we apply the foregoing arguments to part (b) of Lemma~\ref{lm:S2}, we may deduce\begin{align*}I_{\ref{eq:beta}}:=\frac{\pi(\pi-\beta)}{8}={}&\frac1{16}\int_0^\pi\D\theta_1\int_0^\pi\D\theta_2\int_{(0,1)\cap(-\infty,\frac{\cos\beta-\cos(\theta_1+\theta_2)}{2\sin\theta_1\sin\theta_2})}\frac{\cos\frac{\theta_1-\theta_{2}}{2}\D u}{\pi\sqrt{u(1-u)(\frac{\cos\beta-1}{2}+\sin^2\frac{\theta_1+\theta_{2}}{2}-u\sin\theta_1\sin\theta_2)}}\notag\\={}&\frac{1}{8}\iiint_{\substack{0< u,v, w<1; v(1-u+uw)>\frac{1-\cos\beta}{2}}}\frac{\D u\D v\D w}{\pi\sqrt{u(1-u)(\frac{\cos\beta-1}{2}+v-uv+uvw)}\sqrt{w(1-v)}}\end{align*}due to the geometric constraint $ \bm n_1\cdot\bm n_2<\cos\beta$.
We next move on to further reduce the triple integrals, starting with two observations: (i) An inequality  $v\equiv\sin^2\frac{\theta_1+\theta_2}{2}>\sin^2\frac{\beta}{2} $  follows from the geometric constraint\[\cos\beta>\cos\theta_1\cos\theta_2+\sin\theta_1\sin\theta_2\cos\varphi\geq \cos(\theta_1+\theta_2);\](ii) The introduction of a new variable\[S:=\frac{v w-\sin^2\frac{\beta}{2}}{ v-\sin^2\frac{\beta}{2}}=\frac{\cos\beta-\cos(\theta_1-\theta_2)}{\cos\beta-\cos(\theta_1+\theta_2)}\in\left[ -\frac{\sin^2\frac{\beta}{2}}{v-\sin^2\frac{\beta}{2}} ,1\right],\quad\text{with}\quad\left\vert\det\frac{\partial(S,u,v)}{\partial(u,v,w)}\right\vert=\frac{v}{ v-\sin^2\frac{\beta}{2}}\]leads to a factorization\[\frac{\cos\beta-1}{2}+v-uv+uvw=[1-u(1-S)]\left( v-\sin^2\frac{\beta}{2} \right).\]It is thus reasonable to introduce another new variable $ r:=1-u(1-S)\in(0,1)$, so that $ S< r$ is a consequence of $ u=(1-r)/(1-S)< 1$.

With the replacements  \[u=\frac{1-r}{1-S},\quad v=q \cos^2\frac{\beta}{2}+\sin^2\frac{\beta}{2},\quad w=S+\frac{1-S}{v}\sin^2\frac{\beta}{2}=\frac{q S+\tan^2\frac{\beta}{2}}{q +\tan^2\frac{\beta}{2}},
\] namely, a   transformation of variables\[q=\frac{v-\sin^2\frac{\beta}{2}}{\cos^2\frac{\beta}{2}}\in(0,1),\quad r=1-u(1-s)=1-\frac{uv (1-w)}{ v-\sin^2\frac{\beta}{2}}\in(0,1),\quad S=\frac{v w-\sin^2\frac{\beta}{2}}{ v-\sin^2\frac{\beta}{2}}\in\left( -\frac{\tan^2\frac{\beta}{2}}{q},r \right)\] with Jacobian determinant\[ \left|\det\frac{\partial(q,r,S)}{\partial(u,v,w)}\right|=\frac{v^{2}(1-w)}{\left( v-\sin^2\frac{\beta}{2} \right)^2\cos^2\frac{\beta}{2}}=\frac{\left( q +\tan^2\frac{\beta}{2} \right)(1-S)}{q\cos^2\frac{\beta}{2}},\] we may convert   the triple integral in $ u,v,w$ into \[I_{\ref{eq:beta}}=\frac{\pi(\pi-\beta)}{8}=\frac{1}{8\pi}\int_0^1\D q\int_0^1\D r\int_{-\frac{\tan^2\frac{\beta}{2}}{q}}^r\D S\frac{\sqrt{q}}{\sqrt{r(1-r)(r-S)(1-q)}\sqrt{q +\tan^2\frac{\beta}{2}}\sqrt{q S+\tan^2\frac{\beta}{2}}}.\]
Finally, it would be natural to rescale $S$ as\[s:=\frac{S+\frac{\tan^2\frac{\beta}{2}}{q}}{r+\frac{\tan^2\frac{\beta}{2}}{q}}=\frac{ w}{ 1-u+uw}\]so that $s$ ranges from $0$ to $1$.
\item

The  ``sphere-to-box transformations''  $ (\theta_1,\theta_2,\varphi)\mapsto(u,v,w)$ and  $(\theta_1,\theta_2,\varphi)\mapsto(q,r,s)$ used for proving $ I_{\ref{eq:a}},I_{\ref{eq:b}},I_{\ref{eq:beta}}$ and $ I_{\ref{eq:beta'}}$ are extensions of  the techniques employed in  Watson's work \cite{Watson1939} on simplifying certain triple integrals of physical significance. The spherical harmonic expansion trick is a continuation of  the author's own previous research~\cite{Zhou2010} on the Hilbert-Schmidt norms of certain operator polynomials in the mathematical model of electromagnetic scattering, which involved an integral on $ S^2\times S^2\times S^2\times S^2$.

During the course of our proof, all the interchanges of  integrations  can be justified without theoretical difficulty: the integrands  are non-negative and the triple integrals are convergent.
\item
The major purpose of the lemma above is to build confidence in a geometric method, rather than producing results that are particularly surprising. Although we have brought a variety of formulae under the  framework of $ S^2\times S^2$,  some of these results can be derived via  alternative methods independent of spherical geometry, which in turn,  can be traced to certain standard references.

 The integral identity\[I_{\ref{eq:a}}=\int_0^{1}\frac{\mathbf{K}(\eta)}{1+\eta}\D\eta=\frac{\pi^{2}}{8}\]is well known (cf.~item 615.09 in \cite{ByrdFriedman} or item 6.144 in \cite{GradshteynRyzhik}) and can be proved by simpler means invoking only double integrals. Our integral representation for Catalan's constant \[ I_{\ref{eq:b}}=\frac{1}{2\pi}\int_0^{1}\frac{\mathbf{K}(\sqrt{1-\xi^2})}{\xi}\log\frac{1+\xi}{1-\xi}\D \xi=G\] is equivalent to formula 4.21.3.13 in \cite{Brychkov2008}, which can be derived using identities satisfied  by hypergeometric functions~\cite{Brychkov}, without going through the integral on spheres.

The identity \[I_{\ref{eq:beta}}=\frac{\pi(\pi-\beta)}{8}=\frac{\pi}{8}\int_0^1\frac{\D q}{\sqrt{\smash[b]{1-q}\vphantom{s^2}}\sqrt{\smash[b]{q +\tan^2(\beta/{2})}}}\]
 is  elementary, while the equality \[I_{\ref{eq:beta'}}=G+\frac{1}{4}\int_0^\beta\log\tan\frac{\pi-\alpha}{4}\D\alpha=-\left(\cos\frac{\beta}{2}\right)\int_0^1\frac{(1+t^2)\log t}{1+2t^{2}\cos\beta+t^4}\D t\]can be double-checked by differentiating both sides in $\beta$. An equivalent form of Eq.~\ref{eq:K_Pyth}, as given by I. Mannari and C. Kawabata in 1964:\[\int_0^{\pi/2}\frac{\D\theta}{\sqrt{(1-C\cos^2\theta)(1-S\sin^2\theta)}}=\mathbf K(\sqrt{C+S-CS}),\]was mentioned in \cite{Bailey2008}. Setting  $ \beta=0$  in\begin{align}
G=\frac{\cos(\beta/2)}{2}\int_0^{\pi/2}\frac{ \theta\D \theta}{\sin\theta\sqrt{\smash[b]{1-\cos^2\theta\sin^2(\beta/{2}})}}-\frac{1}{4}\int_0^\beta\log\tan\frac{\pi-\alpha}{4}\D\alpha,\quad 0\leq \beta<\pi,\label{eq:G_elem}
\end{align}    we obtain\begin{align} G=\frac12\int_0^{\pi/2}\frac{\theta}{\sin\theta}\D\theta,\tag{\ref{eq:G_elem}{$^\circ$}}\label{eq:G_elem'}\end{align} which is a familiar integral representation of Catalan's constant (see 3.747.2 in \cite{GradshteynRyzhik}). \item \label{itm:var_subst}The variable substitutions  in this lemma, either  used in isolation, or in conjunction with hyperspherical geometry (see Proposition~\ref{prop:S3}), can facilitate the integration of many other expressions involving the complete elliptic integral of the first kind $ \mathbf K$:  such integrals can be thus represented by certain mathematical constants, or converted into integrals of elementary functions. As a stand-alone example, we point out that the set of  transformations employed for proving Eq.~\ref{eq:beta'} can also be used to justify the following identity:\begin{align*}\int_\beta^\pi\frac{\sqrt{\smash[b]{2(\cos\beta-\cos\alpha)}}}{2\pi\cos^2(\beta/{2})}\log\cot\frac{\pi-\alpha}{4}\D\alpha={}&\int_{S^2}\frac{\D \sigma_1}{4\pi}\int_{S^2}\frac{\D \sigma_{2}}{4\pi}\frac{1}{|\bm n-\bm n_1|}\frac{\theta_{H}(\cos\beta-\bm n_1\cdot\bm n_2)}{\cos^2(\beta/{2})}\frac{1}{|\bm n_{2}+\bm n|}\notag\\={}&\int_0^1\D q\int_0^1\D r\int_{0}^1\D s\frac{\sqrt{\smash[b]{q}}}{4\pi\sqrt{\vphantom{s^2}\smash[b]{(1-r)(1-s)s(1-q)}}{\sqrt{\smash[b]{1-s[qr\cos^{2}(\beta/{2})+\sin^2(\beta/{2})]}}}}\notag\\={}&\int_0^1\mathbf K\left( \sqrt{k^2\cos^2\frac{\beta}{2}+\sin^2\frac{\beta}{2}} \right)k\D k,\quad 0\leq\beta<\pi.\notag\end{align*} In particular, for $ \beta=0$, the integral on $ S^2\times S^2$ obviously evaluates to unity, thereby bringing us the result\begin{align}\frac{1}{\pi}\int_0^\pi\left(\sin\frac{\alpha }{2}\right)\log\cot\frac{\pi-\alpha}{4}\D\alpha=\int_0^1\mathbf K(k)k\D k=1.\label{eq:unity_int}\end{align}\item\label{itm:b2b} We may write Catalan's constant $G$ as a triple integral using either $ I_{\ref{eq:beta'}}$ in the specific case of $ \beta=0$ or   $ I_{\ref{eq:b}}$. Namely, the number  $ 4\pi^{2} G$, as a member in the ring of  periods \cite{KontsevichZagier},  has two integral representations\begin{align*}4\pi^{2} G={}&\int_0^1\D q\int_0^1\D r\int_0^1\D s\int_0^\infty\D t\frac{1}{t^2+(1-r)(1-s)rs(1-q)(1-qrs)}\notag\\={}&\int_0^\infty\D t\int_0^1\D u\int_0^1\D v\int_0^1\D w\frac{1}{t^{2}+u(1-u)(1-u+uw)(1-vw)w(1-v)}.\end{align*}We can pass from one quadruple integral  to the other by the following algebraic mappings between two unit boxes:\begin{align}q=v,\quad   r=1-u+uw, \quad s=\frac{ w}{ 1-u+uw}\qquad\Longleftrightarrow\qquad
u=\frac{1-r}{1-rs}, \quad v=q, \quad w=rs, \label{eq:qrs_uvw}\end{align}fully oblivious to the geometric interpretations of the variables $ u,v,w$ and $ q,r,s$.

More generally, evaluations of many multiple elliptic integrals amount to the confirmation of relations among certain members in  the ring of periods, which is conceivably accessible by algebraic means \cite{KontsevichZagier}. It is sometimes possible to transition from one multiple elliptic integral to another using a variety of (geometrically motivated) birational mappings. The Beltrami rotations and Ramanujan rotations in  \cite[Ref.][\S3]{Zhou2013Pnu} are examples of such transformations.
 \eor\end{enumerate}\end{remark}

Upon integration, Eqs.~\ref{eq:beta} and \ref{eq:beta'} give rise to some integral formulae involving Ap\'ery's constant  $ \zeta(3)=\sum_{n=1}^{\infty}n^{-3}$ and  complete elliptic integral of the second kind  $ \mathbf E(k)=\int_0^{\pi/2}(1-k^2\sin^2\theta)^{1/2}\D \theta$: {\allowdisplaybreaks\begin{align}\int_0^\pi\frac{\pi(\pi-\beta)}{8}\D\beta=\frac{\pi^{3}}{16}={}&\int_{S^2}\frac{\D \sigma_1}{4\pi}\int_{S^2}\frac{\D \sigma_{2}}{4\pi}\frac{1}{|\bm n-\bm n_1|}\mathbf K\left( \sqrt{\frac{1-\bm n_1\cdot\bm n_2}{2}} \right)\frac{1}{|\bm n_{2}-\bm n|}=\frac{\pi }{8}\int_{0}^{1}\log\frac{1+\sqrt{\smash[b]{1-q}}}{1-\sqrt{\smash[b]{1-q}}}\frac{\D q}{1-q},\label{eq:pipibeta_int}\tag{\ref{eq:beta}$ ^*$}\\\int_0^\pi\frac{\pi(\pi-\beta)}{8}\cos\frac{\beta}{4}\D\beta=(2-\sqrt{2})\pi={}&\int_{S^2}\frac{\D \sigma_1}{4\pi}\int_{S^2}\frac{\D \sigma_{2}}{4\pi}\frac{1}{|\bm n-\bm n_1|} \frac{\mathbf K\left( \sqrt{\frac{2\sqrt{\smash[b]{(1-\bm n_1\cdot\bm n_2)/2}}}{1+\sqrt{\smash[b]{(1-\bm n_1\cdot\bm n_2)/2}}}} \right)}{\sqrt{\smash[b]{1+\sqrt{\smash[b]{(1-\bm n_1\cdot\bm n_2)/2}}}}}\frac{1}{|\bm n_{2}-\bm n|}\notag\\={}&\frac{\sqrt{2}\pi}{6}\int_0^1{_3F_2\left( \left.\begin{array}{c}
\frac{1}{2},1 ,\frac{3}{2}  \\[4pt]
\frac{5}{4},\frac{7}{4}
\end{array}\right| 1-q\right)}\frac{\D q}{\sqrt{\smash[b]{1-q}}},\label{eq:pipibeta_int_cos_quarter_beta}\tag{\ref{eq:beta}$ ^{**}$}\\\int_0^\pi\frac{\pi(\pi-\beta)}{8}\cos\frac{\beta}{2}\D\beta=\frac{\pi}{2}={}&\int_{S^2}\frac{\D \sigma_1}{4\pi}\int_{S^2}\frac{\D \sigma_{2}}{4\pi}\frac{1}{|\bm n-\bm n_1|} \frac{\pi}{2}\frac{1}{|\bm n_{2}-\bm n|}=\frac{\pi }{4}\int_{0}^{1}[\mathbf K(\sqrt{\smash[b]{1-q}})-\mathbf E(\sqrt{\smash[b]{1-q}})]\frac{\D q}{(1-q)^{3/2}},\label{eq:pipibeta_int_cos_half_beta}\tag{\ref{eq:beta}$ ^{***}$}\\
\int_0^\pi\mathfrak L(\beta)\D\beta=\frac{7\zeta(3)}{4}={}&\int_{S^2}\frac{\D \sigma_1}{4\pi}\int_{S^2}\frac{\D \sigma_{2}}{4\pi}\frac{1}{|\bm n-\bm n_1|}\mathbf K\left( \sqrt{\frac{1-\bm n_1\cdot\bm n_2}{2}} \right)\frac{1}{|\bm n_{2}+\bm n|}\notag\\
={}&\int_0^1\D q\int_0^1\D r\int_{0}^1\D s\frac{1}{4\pi s\sqrt{\smash[b]{(1-r)(1-s)r(1-q)(1-qr)}}}\arctan\sqrt{\frac{(1-qr)s}{1-s}}\notag\\
={}&\int_0^1\log\frac{1-t}{1+t}\log t\frac{\D t}{t}=\int_0^{\pi/2}\frac{ \theta}{\sin\theta}\frac{\frac{\pi}{2}-\theta}{\cos\theta}\D \theta,\label{eq:L_beta_int}\tag{\ref{eq:beta'}$ ^*$}\\\int_0^\pi \mathfrak L(\beta)\cos\frac{\beta}{4}\D\beta=2\sqrt{2}\log2={}&\int_{S^2}\frac{\D \sigma_1}{4\pi}\int_{S^2}\frac{\D \sigma_{2}}{4\pi}\frac{1}{|\bm n-\bm n_1|} \frac{\mathbf K\left( \sqrt{\frac{2\sqrt{\smash[b]{(1-\bm n_1\cdot\bm n_2)/2}}}{1+\sqrt{\smash[b]{(1-\bm n_1\cdot\bm n_2)/2}}}} \right)}{\sqrt{\smash[b]{1+\sqrt{\smash[b]{(1-\bm n_1\cdot\bm n_2)/2}}}}}\frac{1}{|\bm n_{2}+\bm n|}\notag\\={}&\frac{\sqrt2}{6}\int_0^1\D q\int_0^1\D r\int_{0}^1\D s\frac{1}{\pi(1- s)\sqrt{\smash[b]{(1-r)rs(1-q)}}} \,{_3F_2\left( \left.\begin{array}{c}
\frac{1}{2},1 ,\frac{3}{2}  \\[4pt]
\frac{5}{4},\frac{7}{4}
\end{array}\right| -\frac{(1-qr)s}{1-s}\right)}\notag\\={}&e^{i\pi/4}\int_0^1\left[ (t+i) \tan ^{-1}\frac{(1+i) \sqrt{t}}{t-i}-(t-i) \tanh ^{-1}\frac{(1+i) \sqrt{t}}{t+i} \right]\frac{\log t\D t}{2 t^{3/2}},\label{eq:L_beta_int_cos_quarter_beta}\tag{\ref{eq:beta'}$ ^{**}$}\\\int_0^\pi \mathfrak L(\beta)\cos\frac{\beta}{2}\D\beta=\frac{\pi}{2}={}&\int_{S^2}\frac{\D \sigma_1}{4\pi}\int_{S^2}\frac{\D \sigma_{2}}{4\pi}\frac{1}{|\bm n-\bm n_1|}\frac{\pi}{2}\frac{1}{|\bm n_{2}+\bm n|}\notag\\
={}&\int_0^1\D q\int_0^1\D r\int_{0}^1\D s\frac{(1-qrs)\mathbf E\left( \sqrt{\frac{(1-qr)s}{1-qrs}} \right)-(1-s)\mathbf K\left( \sqrt{\frac{(1-qr)s}{1-qrs}} \right)}{4\pi (1-qr)s\sqrt{\smash[b]{(1-r)(1-s)rs(1-q)(1-qrs)}\vphantom{s^2}}}\notag\\
={}&-\frac{\pi}{2}\int_0^1\log t\D t=\int_0^{\pi/2}\left[ \frac{\mathbf E(\cos\theta)}{\sin\theta} -\sin\theta\mathbf K(\cos\theta)\right]\frac{ \theta\D \theta}{\cos^2\theta}.\label{eq:L_beta_int_cos_half_beta}\tag{\ref{eq:beta'}$ ^{***}$}
\end{align}}\\In the above, most integrations are straightforward,  except a few that additionally require  Taylor expansions. For example:\begin{align*}\int_0^\pi\frac{\cos(\beta/4)\D\beta}{\sqrt{\smash[b]{q+\tan^2(\beta/2)}}}={}&\int_0^\pi\frac{\cos(\beta/4)\cos(\beta/2)\D\beta}{\sqrt{\smash[b]{1-(1-q)\cos^2(\beta/2)}}}=\sum_{n=0}^\infty(-1)^n(1-q)^{n}\int_0^\pi\frac{\sqrt{\pi}\cos(\beta/4)[\cos(\beta/2)]^{2n+1}\D\beta}{n!\Gamma(\frac{1}{2}-n)}\notag\\={}&\sqrt{2}\pi\sum_{n=0}^\infty(-1)^n(1-q)^{n}\frac{(2n+1)!}{n!\Gamma(\frac{1}{2}-n)\Gamma(n+\frac{5}{2})}=\frac{4\sqrt{2}}{3}\,{_3F_2\left( \left.\begin{array}{c}
\frac{1}{2},1 ,\frac{3}{2}  \\[4pt]
\frac{5}{4},\frac{7}{4}
\end{array}\right| 1-q\right)}.\end{align*}

In \cite[Ref.][Lemma~1.2]{Zhou2013Pnu}, we employed the Hobson coupling formula to establish the following  integral identities for $ \theta_1,\theta_2\in[0,\pi)$:
\begin{align}
\frac{1}{4}\int_0^{2\pi}\mathbf K\left( \sin\frac{\varTheta}{2} \right)\D\phi={}&\begin{cases}\mathbf K\left( \sin\frac{\theta_1}{2} \right)\mathbf K\left( \sin\frac{\theta_2}{2} \right), & \theta_{1}+\theta_2\leq\pi \\[6pt]
\mathbf K\left( \cos\frac{\theta_1}{2} \right)\mathbf K\left( \cos\frac{\theta_2}{2} \right), & \theta_{1}+\theta_2\geq\pi
\end{cases}\label{eq:LegendreP_half}\\\frac{1}{4}\int_0^{2\pi}\mathbf K\left( \sqrt{\frac{2\sin(\varTheta/2)}{1+\sin(\varTheta/2)}} \right)\frac{\D\phi}{\sqrt{1+\sin(\varTheta/2)}}={}&\begin{cases}\frac{1}{\sqrt{\smash[b]{1+\sin(\theta_{1}/2)}}\sqrt{\smash[b]{1+\sin(\theta_{2}/2)}}}\mathbf K\left( \sqrt{\frac{2\sin(\theta_{1}/2)}{1+\sin(\theta_{1}/2)}} \right)\mathbf K\left( \sqrt{\frac{2\sin(\theta_{2}/2)}{1+\sin(\theta_{2}/2)}} \right), & \theta_{1}+\theta_2\leq\pi \\[6pt]
\frac{1}{\sqrt{\smash[b]{1+\cos(\theta_{1}/2)}}\sqrt{\smash[b]{1+\cos(\theta_{2}/2)}}}\mathbf K\left( \sqrt{\frac{2\cos(\theta_{1}/2)}{1+\cos(\theta_{1}/2)}} \right)\mathbf K\left( \sqrt{\frac{2\cos(\theta_{2}/2)}{1+\cos(\theta_{2}/2)}} \right), & \theta_{1}+\theta_2\geq\pi
\end{cases}\label{eq:LegendreP_quarter}
\end{align}
where $ \cos\varTheta=\cos\theta_1\cos\theta_2+\sin\theta_1\sin\theta_2\cos\phi$.  In the light of this, we may  use Eqs.~\ref{eq:LegendreP_half} and \ref{eq:LegendreP_quarter} to ``decouple'' the integrals on $ S^2\times S^2$ in Eqs.~\ref{eq:pipibeta_int}, \ref{eq:pipibeta_int_cos_quarter_beta}, \ref{eq:L_beta_int} and \ref{eq:L_beta_int_cos_quarter_beta}. With some elementary trigonometric identities  that were used in the symmetric reduction of $ I_{\ref{eq:a}}$ and  $ I_{\ref{eq:b}}$ at the initial stages of the proof for Lemma~\ref{lm:S2}, we arrive at some  integrals on two-dimensional simplices:\begin{align}\frac{\pi^3}{16}={}&\frac{1}{2\pi}\iint_{0\leq\theta_1\leq\pi,0\leq\theta_2\leq\pi,0\leq\theta_1+\theta_2\leq\pi}\mathbf K\left( \sin\frac{\theta_{1}}{2} \right)\mathbf K\left( \sin\frac{\theta_{2}}{2} \right)\cos\frac{\theta_1-\theta_2}{2}\D\theta_1\D\theta_2,\tag{\ref{eq:pipibeta_int}$ '$}\label{eq:pipibeta_int'}\\(2-\sqrt{2})\pi={}&\frac{1}{2\pi}\iint_{0\leq\theta_1\leq\pi,0\leq\theta_2\leq\pi,0\leq\theta_1+\theta_2\leq\pi} \frac{\mathbf K\left( \sqrt{\frac{2\sin(\theta_{1}/2)}{1+\sin(\theta_{1}/2)}} \right)\mathbf K\left( \sqrt{\frac{2\sin(\theta_{2}/2)}{1+\sin(\theta_{2}/2)}} \right)}{\sqrt{\smash[b]{1+\sin(\theta_{1}/2)}}\sqrt{\smash[b]{1+\sin(\theta_{2}/2)}}}\cos\frac{\theta_1-\theta_2}{2}\D\theta_1\D\theta_2,\tag{\ref{eq:pipibeta_int_cos_quarter_beta}$'$}\label{eq:pipibeta_int_cos_quarter_beta'}\\\frac{7\zeta(3)}{4}={}&\frac{1}{2\pi}\iint_{0\leq\theta_1\leq\pi,0\leq\theta_2\leq\pi,0\leq\theta_1+\theta_2\leq\pi}\mathbf K\left( \sin\frac{\theta_{1}}{2} \right)\mathbf K\left( \sin\frac{\theta_{2}}{2} \right)\sin\frac{\theta_1+\theta_2}{2}\D\theta_1\D\theta_2,\tag{\ref{eq:L_beta_int}$'$}\label{eq:L_beta_int'}\\2\sqrt{2}\log2={}&\frac{1}{2\pi}\iint_{0\leq\theta_1\leq\pi,0\leq\theta_2\leq\pi,0\leq\theta_1+\theta_2\leq\pi} \frac{\mathbf K\left( \sqrt{\frac{2\sin(\theta_{1}/2)}{1+\sin(\theta_{1}/2)}} \right)\mathbf K\left( \sqrt{\frac{2\sin(\theta_{2}/2)}{1+\sin(\theta_{2}/2)}} \right)}{\sqrt{\smash[b]{1+\sin(\theta_{1}/2)}}\sqrt{\smash[b]{1+\sin(\theta_{2}/2)}}}\sin\frac{\theta_1+\theta_2}{2}\D\theta_1\D\theta_2.\tag{\ref{eq:L_beta_int_cos_quarter_beta}$'$}\label{eq:L_beta_int_cos_quarter_beta'}\end{align}

  The integration of the relation\begin{align*}\mathfrak L(\beta)=G+\frac{1}{4}\int_0^\beta\log\tan\frac{\pi-\alpha}{4}\D\alpha=\frac{\cos(\beta/2)}{2}\int_0^1\mathbf K\left( \sqrt{k^2\cos^2\frac{\beta}{2}+\sin^2\frac{\beta}{2} }\right)\D k\end{align*}was omitted from our list of identities in Eqs.~\ref{eq:pipibeta_int}, \ref{eq:pipibeta_int_cos_half_beta}, \ref{eq:L_beta_int} and \ref{eq:L_beta_int_cos_half_beta}, but will be treated in \S\ref{subsubsec:B_simp}.
\subsection{Integrals for Azimuthally Anisotropic Couplings\label{subsubsec:azimuthal}}

In the physical parlance,  we have  computed in \S\ref{subsubsec:Ylm} some   Coulomb-like interactions on  spheres where the ``force potential'' is modulated with   polar anisotropy. While reducing the triple integrals in  part~(a) of Lemma~\ref{lm:S2}, we integrated over $ u$ before $w$; we then introduced an additional  factor in the integrand (which modified the isotropic Coulomb potential) dependent on $\bm n_1\cdot\bm n_2 $ to  the triple integrals in part~(b).

 In the next proposition, we will redo part~(a) of Lemma~\ref{lm:S2} by completing the integration over $w$ before $u$, to derive identities of a  different flavor, and will explore the effect of  azimuthal dependence of the ``force potential modulation''.
\begin{proposition}[More Integrals on $ S^2\times S^2$]\label{prop:more_S2}\begin{enumerate}[label=\emph{(\alph*)}, ref=(\alph*), widest=a]

 \item We may rewrite Eqs.~\ref{eq:a} and \ref{eq:b} for any unit vector $ \bm n\in S^2$:\begin{align}\frac{\pi^2}{8}={}&\int_{S^2}\frac{\D \sigma_1}{4\pi}\int_{S^2}\frac{\D \sigma_{2}}{4\pi}\frac{1}{|\bm n-\bm n_1|}\frac{1}{|\bm n_{1}-\bm n_2|}\frac{1}{|\bm n_{2}-\bm n|}\notag\\={}&\int_0^1\D u\int_0^1\D v\int_0^{1}\D w\frac{1}{8\pi\sqrt{u(1-u)(1-u+uw)}\sqrt{vw(1-v)}}\notag\\={}&\frac{1}{4}\int_{0}^{\pi/2}\frac{\log\frac{1+\cos\theta}{1-\cos\theta}\D \theta}{ \cos\theta}=-\int_{0}^{\pi/2}\frac{\log\sin\theta\D \theta}{ \cos\theta}=-\int_{0}^{\pi/2}\frac{\log\tan\frac{\theta}{2}\D \theta}{ 2\cos\theta}=\int_{0}^{1}\frac{\log t\D t}{ t^{2}-1};\tag{\ref{eq:a}*}\label{eq:a_star}\\G={}&\int_{S^2}\frac{\D \sigma_1}{4\pi}\int_{S^2}\frac{\D \sigma_{2}}{4\pi}\frac{1}{|\bm n-\bm n_1|}\frac{1}{|\bm n_{1}-\bm n_2|}\frac{1}{|\bm n_{2}+\bm n|}\notag\\={}&\int_0^1\D u\int_0^1\D v\int_0^{1}\D w\frac{1}{8\pi\sqrt{u(1-u)(1-u+uw)}\sqrt{1-vw}\sqrt{w(1-v)}}\notag\\={}&\frac{1}{2\pi}\left\{\int_0^{1}\frac{[\mathbf K(k)]^{2}\D k}{\sqrt{1-k^{2}}}-2\int_{0}^1\mathbf K(k)\mathbf K\left(\sqrt{1-k^2}\right)\D k\right\}+\frac{1}{2\pi}\int_0^{1}\frac{\mathbf K(\sqrt{1-\xi^{2}})}{\xi}\log\frac{1+\xi}{1-\xi}\D \xi.\tag{\ref{eq:b}*}\label{eq:b_star}\end{align} In particular, Eq.~\ref{eq:b_star} implies the identity\begin{align}\int_0^{1}\frac{[\mathbf K(k)]^{2}\D k}{\sqrt{1-k^{2}}}=2\int_{0}^1\mathbf K(k)\mathbf K\left(\sqrt{1-k^2}\right)\D k.\label{eq:KKKK}\end{align}\item  More generally, suppose that the function $ \omega(\cot\frac{|\varphi|}{2})$ is continuously differentiable in $ \varphi$ for $0\leq\varphi\leq\varphi_0$, where $ \varphi_0\in[0,\pi]$, then   we have \begin{align}F_{\varphi_{0}}[\omega]:={}&
\frac{1}{2\pi}\int_0^{1}\left[ \int_{\cos(\varphi_{0}/2)}^{1}\frac{ \omega(k/\sqrt{1-k^{2}})\D k}{\sqrt{(1-k^{2})[1-k^{2}(1-\xi^{2})]}} \right]\log\frac{1+\xi}{1-\xi}\frac{\D \xi}{\xi}\notag\\={}&\int_{S^2}\frac{\D \sigma_1}{4\pi}\int_{S^2}\frac{\D \sigma_{2}}{4\pi}\frac{1}{|\bm n-\bm n_1|}\frac{\omega(\cot\frac{|\phi_1-\phi_2|}{2})\theta_{H}(\varphi_0-|\phi_1-\phi_2|)}{|\bm n_{1}-\bm n_2|}\frac{1}{|\bm n_{2}+\bm n|}\notag\\={}&\frac{1}{2\pi}\int_{\cos(\varphi_{0}/2)}^{1}\frac{[\mathbf K(k)]^{2} \omega(k/\sqrt{1-k^{2}})\D k}{\sqrt{1-k^{2}}}\notag\\&-\frac{1}{\pi}\int_{{\cot(|\varphi_{0}|/2)}}^\infty\left[ \omega(\kappa)+\kappa\frac{\D \omega(\kappa)}{\D\kappa} \right]\D \kappa\int_0^1\D t\int_0^1\D\xi\frac{1}{\xi t\sqrt{\smash[b]{1+\kappa^2t^2}}\sqrt{\smash[b]{1+\kappa^2\xi ^{2}t^2}}}\log\frac{\sqrt{1-\xi ^{2}t^{2}}-\xi\sqrt{1-t^2}}{1-\xi}\notag\\&-\frac{1}{\pi}\omega\left( \cot\frac{|\varphi_0|}{2} \right)\cot\frac{|\varphi_0|}{2}\int_0^1\D t\int_0^1\D\xi\frac{1}{\xi t\sqrt{\smash[b]{1+t^2\cot^{2}(|\varphi_{0}|/2)}}\sqrt{\smash[b]{1+\xi ^{2}t^{2}\cot^{2}(|\varphi_{0}|/2)}}}\log\frac{\sqrt{1-\xi ^{2}t^{2}}-\xi\sqrt{1-t^2}}{1-\xi},\label{eq:omega}
\end{align}where\[|\phi_1-\phi_2|:=\arccos\frac{\bm n_1\cdot\bm n_2-(\bm n_1\cdot\bm n)(\bm n_2\cdot\bm n)}{\sqrt{1-(\bm n_1\cdot\bm n)^{2}}\sqrt{1-(\bm n_2\cdot\bm n)^{2}}}\in[0,\pi].\]Some particular cases of Eq.~\ref{eq:omega} lead to the identities\begin{align}\int_{0}^{1}\left\{[\mathbf K(k)]^{2}-\frac{\pi^2}{4}\right\}\frac{ \D k}{k}={}&\frac{\pi^2}{2}\log2-\frac{7\zeta(3)}{4},\label{eq:K_sqr_minus_pi4}\\\int_{0}^{1}\frac{[\mathbf K(k)]^{2}}{\sqrt{1-k^{2}}}\log\frac{k}{\sqrt{1-k^2}}\D k={}&\frac{\pi}{2}\int_{0}^{1}\frac{[\mathbf K(k)]^{2}k}{\sqrt{1-k^{2}}}\D k.\label{eq:K_sqr_log}\end{align}\end{enumerate}\end{proposition}\begin{proof}\begin{enumerate}[label=(\alph*),widest=a]\item Integrating over $w$ first, we may deduce\begin{align*}\int_0^1\D u\int_0^1\D v\int_0^{1}\D w\frac{1}{8\pi\sqrt{u(1-u)(1-u+uw)}\sqrt{vw(1-v)}}=\int_0^{1}\D u\int_{0}^{1}\frac{\log\frac{1+\sqrt{\vphantom{1}u}}{1-\sqrt{\vphantom{1}u}}\D v}{8\pi u\sqrt{(1-u)v(1-v)}}\overset{u=\cos^2\theta}{=\!\!=\!\!=\!\!=\!\!=\!\!=}\frac{1}{4}\int_{0}^{\pi/2}\frac{\log\frac{1+\cos\theta}{1-\cos\theta}\D \theta}{ \cos\theta}.\end{align*}Meanwhile, it is not hard to verify\[\frac{\pi^2}{8}=I_{\ref{eq:a}}=\int_0^{1}\frac{\mathbf{K}(\eta)}{1+\eta}\D\eta=\int_0^{\pi/2}\left[\int_0^1\frac{\D \eta}{(1+\eta)\sqrt{\smash[b]{1-\eta^{2}\sin^{2}\theta}}}\right]\D\theta=\int_0^{\pi/2}\frac{\log(1+\cos\theta)}{\cos\theta}\D\theta,\]thus Eq.~\ref{eq:a_star} follows from the simple algebra\begin{align*}\frac{\pi^2}{8}=\frac{1}{2}\int_{0}^{\pi/2}\frac{\log\frac{1+\cos\theta}{1-\cos\theta}\D \theta}{ \cos\theta}-\int_0^{\pi/2}\frac{\log(1+\cos\theta)}{\cos\theta}\D\theta=-\int_{0}^{\pi/2}\frac{\log\sin\theta\D \theta}{ \cos\theta},\end{align*}and\[\frac{\pi^2}{8}=-\int_{0}^{\pi/2}\frac{\log\frac{\sin\theta}{1+\cos\theta}\D \theta}{ 2\cos\theta}=-\int_{0}^{\pi/2}\frac{\log\tan\frac{\theta}{2}\D \theta}{ 2\cos\theta}=\int_{0}^{1}\frac{\log t\D t}{ t^{2}-1}.\]Before treating the triple integral for  Eq.~\ref{eq:b_star}, we make the substitution   $ u=\kappa^{2}/(1+\kappa^{2})$ and $ w=t^2$, so that \begin{align*}G={}&\frac{1}{2\pi}\int_0^{\infty}\frac{\D \kappa}{\sqrt{1+\kappa^{2}}}\int_0^{1}\left[ \int_0^{1}\frac{\D t}{\sqrt{1+\kappa^{2}t^{2}}\sqrt{1-vt^2}} \right]\frac{\D v}{\sqrt{1-v}}.\end{align*}We  then move on to manipulate the bracketed integral in $ t$. It is straightforward to check that  the bivariate function\[f(\kappa,v):=\int_0^{1}\frac{\D t}{\sqrt{1+\kappa^{2}t^{2}}\sqrt{1-vt^2}}\]satisfies the following partial differential equation\begin{align}\frac{f(\kappa,v)}{2} +\frac{\kappa}{2}  \frac{\partial f(\kappa,v)}{\partial \kappa}+v \frac{\partial f(\kappa,v)}{\partial v}=\frac{1}{2\sqrt{1-v}\sqrt{1+\kappa^2}},\label{eq:f_PDE}\end{align}with ``boundary conditions'' \begin{align*}&\lim_{v\to0^+}v[f(\kappa,v)]^{2}=0,\quad \lim_{v\to1^-}f(\kappa,v)=\frac{1}{\sqrt{1+\kappa^2}}\mathbf K\left(\frac{\kappa}{\sqrt{1+\kappa^2}}\right)\quad \text{for }0<\kappa<+\infty ;\notag\\& \lim_{\kappa\to0^{+}}\kappa[f(\kappa,v)]^{2}=0,\quad \lim_{\kappa\to+\infty}\kappa[f(\kappa,v)]^{2}=0\quad \text{for }0<v<1.\end{align*} Therefore, we obtain after integration by parts\begin{align*}G={}&\frac{1}{\pi}\int_0^{\infty}\D \kappa\int_0^{1}\D v\,f(\kappa,v)\left[ \frac{f(\kappa,v)}{2} +\frac{\kappa}{2}  \frac{\partial f(\kappa,v)}{\partial \kappa}+v \frac{\partial f(\kappa,v)}{\partial v} \right]\notag\\={}&\frac{1}{2\pi}\int_0^{\infty}\left[\mathbf K\left(\frac{\kappa}{\sqrt{1+\kappa^2}}\right)\right]^2\frac{\D \kappa}{1+\kappa^{2}}-\frac{1}{\pi}\int_0^{\infty}\D \kappa\int_0^{1}\D v\left[\frac{f(\kappa,v)}{2}\right]^2=\frac{1}{2\pi}\int_0^{1}\frac{[\mathbf K(k)]^{2}\mathfrak \D k}{\sqrt{1-k^{2}}}-\frac{1}{4\pi}\int_0^{\infty}\D \kappa\int_0^{1}\D v\,[f(\kappa,v)]^2.\end{align*}  Here, the last double integral can be further simplified, as   we  write \begin{align*}[f(\kappa,v)]^{2}=\left[\int_{0}^{1}\frac{\D t}{\sqrt{\smash[b]{1-vt^2}}\sqrt{\smash[b]{1+\kappa^2t^2}}}\right]^{2}={}&2\int_0^1\D t\int_0^1\D\xi\frac{t}{\sqrt{\smash[b]{1-vt^2}}\sqrt{\smash[b]{1+\kappa^2t^2}}\sqrt{\smash[b]{1-v\xi ^{2}t^2}}\sqrt{\smash[b]{1+\kappa^2\xi ^{2}t^2}}},\end{align*}and integrate in $\kappa$ and $v$  separately:\begin{align*}\int_0^1\left\{\int_0^{\infty}[f(\kappa,v)]^2\D\kappa\right\}\D v=4\int_0^1\D t\int_0^1\D\xi\frac{\mathbf K(\sqrt{1-\xi^2})}{t^2\xi}\log\frac{\sqrt{1-\xi ^{2}t^{2}}-\xi\sqrt{1-t^2}}{1-\xi}.\end{align*}Then we may  integrate by parts with respect to  $t$ to deduce\begin{align*}\int_0^1\frac{\mathbf K(\sqrt{1-\xi^2})}{t^2\xi}\log\frac{\sqrt{1-\xi ^{2}t^{2}}-\xi\sqrt{1-t^2}}{1-\xi}\D t=\left[\mathbf K(\xi)-\frac{1}{2\xi}\log\frac{1+\xi}{1-\xi}\right]\mathbf K\left(\sqrt{1-\xi^2}\right).\end{align*}
Therefore, we can verify Eq.~\ref{eq:b_star} with the computation  \begin{align*}&\frac{1}{4\pi}\int_0^{\infty}\D \kappa\int_0^{1}\D v\,[f(\kappa,v)]^2=\frac{1}{\pi}\int_{0}^1\mathbf K(\xi)\mathbf K\left(\sqrt{1-\xi^2}\right)\D\xi-\int_0^{1}\frac{\mathbf K(\sqrt{1-\xi^{2}})}{2\pi\xi}\log\frac{1+\xi}{1-\xi}\D \xi.\end{align*} Now that the last integral in the equation above evaluates to $G$ (according to Eq.~\ref{eq:b}), we reach the identity claimed in Eq.~\ref{eq:KKKK}.\item We note that the transformation $ \cos^2(\varphi/2)=u=\kappa^2/(1+\kappa^2)$ entails $ \kappa=\cot(|\varphi|/2)$, and \begin{align*}&\int_{S^2}\frac{\D \sigma_1}{4\pi}\int_{S^2}\frac{\D \sigma_{2}}{4\pi}\frac{1}{|\bm n-\bm n_1|}\frac{\omega(\cot\frac{|\phi_1-\phi_2|}{2})\theta_{H}(\varphi_0-|\phi_1-\phi_2|)}{|\bm n_{1}-\bm n_2|}\frac{1}{|\bm n_{2}+\bm n|}\notag\\={}&\frac{1}{2\pi}\int_0^1\D v\int_0^{1}\D t\int_{\cot(|\varphi_{0}|/2)}^\infty\frac{\omega(\kappa)\D \kappa}{\sqrt{(1+\kappa^{2})(1+\kappa^{2}t^{2})}\sqrt{1-vt^{2}}\sqrt{1-v}}\end{align*}can be reduced to a double integral by an integration over $v$, and a subsequent transformation $\kappa=k/\sqrt{1-k^{2}}$:   \begin{align*}F_{\varphi_{0}}[\omega]:=\frac{1}{2\pi}\int_0^{1}\left[ \int_{\cot(|\varphi_{0}|/2)}^{\infty}\frac{\omega(\kappa)\D \kappa}{\sqrt{(1+\kappa^{2})(1+\kappa^{2}t^{2})}} \right]\log\frac{1+t}{1-t}\frac{\D t}{t}=\frac{1}{2\pi}\int_0^{1}\left[ \int_{\cos(\varphi_{0}/2)}^{1}\frac{ \omega(k/\sqrt{1-k^{2}})\D k}{\sqrt{(1-k^{2})[1-k^{2}(1-t^{2})]}} \right]\log\frac{1+t}{1-t}\frac{\D t}{t}.\end{align*}This establishes the first line of  Eq.~\ref{eq:omega}. Using integration by parts as in the proof of part~(a), we may arrive at the expression\begin{align*}F_{\varphi_{0}}[\omega]={}&\frac{1}{2\pi}\int_{\cos(\varphi_{0}/2)}^{1}\frac{[\mathbf K(k)]^{2} \omega(k/\sqrt{1-k^{2}})\D k}{\sqrt{1-k^{2}}}-\frac{1}{4\pi}\int_{\cot(|\varphi_{0}|/2)}^{\infty}\left[\omega(\kappa)+\kappa\frac{\D \omega(\kappa)}{\D\kappa}\right]\D \kappa\int_0^{1}\D v\,[f(\kappa,v)]^2\notag\\{}&-\frac{1}{4\pi}\omega\left( \cot\frac{|\varphi_0|}{2} \right)\cot\frac{|\varphi_0|}{2}\int_{0}^1\left[f\left(\cot\frac{|\varphi_0|}{2},v\right)\right]^2
\D v,\end{align*}where the last term vanishes for $ \varphi_0=\pi$, as in part (a). This is related to  final identity claimed in  Eq.~\ref{eq:omega}, as we note that \begin{align*}\int_{0}^1[f(\kappa,v)]^2
\D v={}&2\int_{0}^1\D v\int_0^1\D t\int_0^1\D\xi\frac{t}{\sqrt{\smash[b]{1-vt^2}}\sqrt{\smash[b]{1+\kappa^2t^2}}\sqrt{\smash[b]{1-v\xi ^{2}t^2}}\sqrt{\smash[b]{1+\kappa^2\xi ^{2}t^2}}}\notag\\={}&4\int_0^1\D t\int_0^1\D\xi\frac{1}{\xi t\sqrt{1+\kappa^2t^2}\sqrt{1+\kappa^2\xi ^{2}t^2}}\log\frac{\sqrt{1-\xi ^{2}t^{2}}-\xi\sqrt{1-t^2}}{1-\xi}.\end{align*}

For $ 0\leq\varphi_0<\pi$, setting  $ \omega(\kappa)=1/\kappa,\cot(\varphi_{0}/2)\leq\kappa<+\infty$ in Eq.~\ref{eq:omega}, we obtain\begin{align*}\int_{\cos(\varphi_{0}/2)}^{1}\frac{[\mathbf K(k)]^{2} \D k}{k}={}&\int_0^{1}\log\frac{1+\sqrt{\smash[b]{\sin^2(\varphi_{0}/2)+\xi^2\cos^2(\varphi_{0}/2)}}}{(1+\xi)\cos(\varphi_{0}/2)}\log\frac{1+\xi}{1-\xi}\frac{\D \xi}{\xi}\notag\\{}&+\int_0^1\D t\int_0^1\D\xi\frac{2}{\xi t\sqrt{\smash[b]{1+t^2\cot^{2}(|\varphi_0|/2)}}\sqrt{\smash[b]{1+\xi ^{2}t^{2}\cot^{2}(|\varphi_0|/2)}}}\log\frac{\sqrt{1-\xi ^{2}t^{2}}-\xi\sqrt{1-t^2}}{1-\xi},\end{align*}which is equivalent to   \begin{align*}\int_{\cos(\varphi_{0}/2)}^{1}\left\{[\mathbf K(k)]^{2}-\frac{\pi^2}{4}\right\}\frac{ \D k}{k}={}&\int_0^{1}\log\frac{1+\sqrt{\smash[b]{\sin^2(\varphi_{0}/2)+\xi^2\cos^2(\varphi_{0}/2)}}}{1+\xi}\log\frac{1+\xi}{1-\xi}\frac{\D \xi}{\xi}\notag\\{}&+\int_0^1\D t\int_0^1\D\xi\frac{2}{\xi t\sqrt{\smash[b]{1+t^2\cot^{2}(|\varphi_0|/2)}}\sqrt{\smash[b]{1+\xi ^{2}t^{2}\cot^{2}(|\varphi_0|/2)}}}\log\frac{\sqrt{1-\xi ^{2}t^{2}}-\xi\sqrt{1-t^2}}{1-\xi},\end{align*}since we have, by Taylor expansion and the monotone convergence theorem\begin{align*}\int_0^{1}\log\frac{1+\xi}{1-\xi}\frac{\D \xi}{\xi}=2\int_0^{1}\sum_{n=0}^\infty\frac{\xi^{2n+1}}{2n+1}\frac{\D \xi}{\xi}=2\sum_{n=0}^\infty\frac{1}{(2n+1)^{2}}=\frac{\pi^2}{4}.\end{align*} In the limit of $ \varphi_0\to\pi$, we obtain\begin{align*}\int_{0}^{1}\left\{[\mathbf K(k)]^{2}-\frac{\pi^2}{4}\right\}\frac{ \D k}{k}={}&\int_0^{1}\log\frac{2}{1+\xi }\log\frac{1+\xi}{1-\xi}\frac{\D \xi}{\xi}+\int_0^1\D t\int_0^1\D\xi\frac{2}{\xi t}\log\frac{\sqrt{1-\xi ^{2}t^{2}}-\xi\sqrt{1-t^2}}{1-\xi}.\end{align*}As we may compute (using Taylor series expansion of the integrand, when necessary)\begin{align*}\int_0^{1}\log\frac{2}{1+\xi }\log\frac{1+\xi}{1-\xi}\frac{\D \xi}{\xi}={}&\frac{\pi^2}{4}\log2-\frac{7\zeta(3)}{8};\notag\\\int_0^1\D t\int_0^1\D\xi\frac{2}{\xi t}\log\frac{\sqrt{1-\xi ^{2}t^{2}}-\xi\sqrt{1-t^2}}{1-\xi}={}&\int_0^1\D t\int_0^1\D\xi\frac{2}{\xi }\log\frac{\sqrt{1-\xi ^{2}t^{2}}-\xi\sqrt{1-t^2}}{1-\xi}\frac{\D\log t}{\D t}\notag\\={}&-\int_0^1\D t\int_0^1\D\xi\frac{2t\log t}{\sqrt{1-t^2}\sqrt{1-\xi^2t^2}}=-\int_0^1\frac{2\arcsin t}{\sqrt{1-t^{2}}}\log t\D t\notag\\={}&\int_0^1\frac{(\arcsin t)^{2}}{t}\D t=\frac{\pi^2}{4}\log2-\frac{7\zeta(3)}{8},\end{align*}we arrive at Eq.~\ref{eq:K_sqr_minus_pi4} as claimed.

With a logarithmic input, we apply a similar limit procedure to  Eq.~\ref{eq:omega} and write down\begin{align*}\int_{0}^{1}\frac{[\mathbf K(k)]^{2}}{\sqrt{1-k^{2}}}\log\frac{k}{\sqrt{1-k^2}}\D k={}&\int_0^{1}\left[ \int_{0}^{\infty}\frac{\log \kappa\D \kappa}{\sqrt{(1+\kappa^{2})(1+\kappa^{2}t^{2})}} \right]\log\frac{1+t}{1-t}\frac{\D t}{t}\notag\\{}&+2\int_{{0}}^\infty(1+ \log \kappa)\D \kappa\int_0^1\D t\int_0^1\D\xi\frac{1}{\xi t\sqrt{\smash[b]{1+\kappa^2t^2}}\sqrt{\smash[b]{1+\kappa^2\xi ^{2}t^2}}}\log\frac{\sqrt{1-\xi ^{2}t^{2}}-\xi\sqrt{1-t^2}}{1-\xi}.\end{align*}Here, the integral involving $ \log\kappa$ can be evaluated in terms of complete elliptic integrals  (see item 800.05 in \cite{ByrdFriedman} or 4.242.1 in \cite{GradshteynRyzhik}), which brings us\begin{align*}\int_{0}^{1}\frac{[\mathbf K(k)]^{2}}{\sqrt{1-k^{2}}}\log\frac{k}{\sqrt{1-k^2}}\D k={}&-\frac12\int_0^{1}\mathbf K\left(\sqrt{1-t^2}\right)\log t\log\frac{1+t}{1-t}\frac{\D t}{t}+2\int_{0}^1\left[\mathbf K(\xi)-\frac{1}{2\xi}\log\frac{1+\xi}{1-\xi}\right]  \mathbf K\left(\sqrt{1-\xi^2}\right)\D\xi\notag\\{}&-\int_0^1\D t\int_0^1\D\xi\frac{\mathbf K(\sqrt{1-\xi^2})\log(\xi t^2)}{\xi t^{2}}\log\frac{\sqrt{1-\xi ^{2}t^{2}}-\xi\sqrt{1-t^2}}{1-\xi}\notag\\={}&\int_{0}^1\left[\mathbf K(\xi)(2-\log\xi)-\frac{1}{\xi}\log\frac{1+\xi}{1-\xi}\right] \mathbf K\left(\sqrt{1-\xi^2}\right)\D\xi\notag\\&-2\int_0^1\D t\int_0^1\D\xi\frac{\mathbf K(\sqrt{1-\xi^2})\log t}{\xi t^{2}}\log\frac{\sqrt{1-\xi ^{2}t^{2}}-\xi\sqrt{1-t^2}}{1-\xi}.\end{align*}The last double integral can be simplified via integration by parts\begin{align*}\int_0^1\D t\int_0^1\D\xi\frac{\mathbf K(\sqrt{1-\xi^2})\log t}{\xi t^{2}}\log\frac{\sqrt{1-\xi ^{2}t^{2}}-\xi\sqrt{1-t^2}}{1-\xi}=\int_0^1\D t\int_0^1\D\xi\, \mathbf K\left(\sqrt{1-\xi^2}\right)\left( \frac{1+\log t}{t}-1 \right)\frac{t}{\sqrt{1-t^{2}}\sqrt{1-\xi ^2t^2}},\end{align*} and  an integration in $ t$ (using item 800.01 in \cite{ByrdFriedman} or 4.242.4 in \cite{GradshteynRyzhik} for the term involving $ \log t$)\begin{align*}\int_0^1\left( \frac{1+\log t}{t}-1 \right)\frac{t\D t}{\sqrt{\smash[b]{1-t^{2}}}\sqrt{\smash[b]{1-\xi ^2t^2}}}=\mathbf K(\xi)-\frac{1}{2}\left[ \frac{\pi}{2}\mathbf K\left( \sqrt{1-\xi^2} \right)+\mathbf K(\xi)\log\xi \right]-\frac{1}{2\xi}\log\frac{1+\xi}{1-\xi}.\end{align*}Thus, we have\begin{align*}\int_{0}^{1}\frac{[\mathbf K(k)]^{2}}{\sqrt{1-k^{2}}}\log\frac{k}{\sqrt{1-k^2}}\D k= \frac{\pi}{2}\int_{0}^1\left[\mathbf K\left( \sqrt{1-\xi^2} \right)\right]^2\D\xi=\frac{\pi}{2}\int_{0}^{1}\frac{[\mathbf K(k)]^{2}k}{\sqrt{1-k^{2}}}\D k,\end{align*} as stated in Eq.~\ref{eq:K_sqr_log}.    \qed\end{enumerate}\end{proof}\begin{remark}
\begin{enumerate}[label=(\arabic*)]
\item
Using Tricomi's formula \cite[Ref.][p.~103]{Tricomi} for the Fourier expansion of $\mathbf K(\sin\theta)$, one may explicitly compute\begin{align*}\int_0^1\frac{[\mathbf K(x)]^2\D x}{\sqrt{1-x^2}}=2\int_0^{1}\mathbf K(x)\mathbf K\left( \sqrt{1-{x^2}} \right)\D x=\frac{\pi}{4}\sum_{n=0}^\infty\frac{[\Gamma(n+\frac{1}{2})]^4}{(n!)^4}=\frac{\pi^{3}}{4}\,{_4F_3\left( \left.\begin{array}{c}
\frac{1}{2},\frac{1}{2} ,\frac{1}{2},\frac{1}{2}  \\
1,1,1
\end{array}\right| 1\right)},\end{align*}as recorded in Eq.~73 of \cite{Bailey2008} and Eq.~35 of \cite{Wan2012}. (Here, $ _4F_3$ is the generalized hypergeometric series.) This backs up our Eq.~\ref{eq:KKKK}, which was derived geometrically.
\item\label{itm:low_deg}
After a preliminary literature search, we have not been able to locate previous reports of the closed-form evaluation given in Eq.~\ref{eq:K_sqr_minus_pi4}, though certain similar-looking integrals  (Eqs.~76, 81 and 82 in \cite{Bailey2008}) had been considered in the framework of Tricomi's expansion, and the expression $ 7\zeta(3)-2\pi^2\log2$ arose in quite a few different contexts \cite{Ayoub1974,NashOConnor1995,Dabrowski1996}. Later in this work, we will give some other integral expressions that evaluate to the same number as  Eq.~\ref{eq:K_sqr_minus_pi4}. Furthermore, a ``lower-degree analog'' of  Eq.~\ref{eq:K_sqr_minus_pi4} that involves Catalan's constant is well known: (cf.~\cite[Ref.][item 615.05]{ByrdFriedman},  \cite[Ref.][item 6.142]{GradshteynRyzhik} or \cite[Ref.][item~2.16.2.3]{PBMVol3})\begin{align}\int_0^1\left[ \mathbf K(k)-\frac{\pi}{2} \right]\frac{\D k}{k}=-2G+\pi\log 2.\label{eq:low_deg}\end{align}
\item
We may  recast  Eq.~\ref{eq:K_sqr_log} into \[\int_{0}^{1}\frac{[\mathbf K(k)]^{2}}{\sqrt{1-k^{2}}}\log\frac{k}{\sqrt{1-k^2}}\D k= \pi\int_{0}^{1}[\mathbf K(k)]^{2}\D k=\frac{\pi^{5}}{32}\,{_7F_6\left( \left.\begin{array}{c}
\frac{1}{2},\frac{1}{2} ,\frac{1}{2},\frac{1}{2}  ,\frac{1}{2},\frac{1}{2},\frac54\\[2pt]\frac14,1,1,
1,1,1
\end{array}\right| 1\right)},\]drawing on a recent result of Borwein \textit{et al.}~\cite{BNSW2011}  mentioned in \cite{Wan2012}, which asserts that  \begin{align*} 2\int_0^1[\mathbf K(k)]^2\D k=\int_0^1\left[\mathbf K\left(\sqrt{1-\kappa^2}\right)\right]^2\D \kappa=\frac{\pi^{4}}{16}\,{_7F_6\left( \left.\begin{array}{c}
\frac{1}{2},\frac{1}{2} ,\frac{1}{2},\frac{1}{2}  ,\frac{1}{2},\frac{1}{2},\frac54\\[2pt]\frac14,1,1,
1,1,1
\end{array}\right| 1\right)},\end{align*} with $ _7F_6$ being the generalized hypergeometric series. As pointed out in  \cite{Wan2012}, the leftmost equality in the formula above is a direct consequence of Landen's transformation. \eor
\end{enumerate} \end{remark}

In \S\ref{subsubsec:B_simp}, we will give a self-contained proof of Eq.~\ref{eq:low_deg}. For the moment, we will accept the truth of  Eq.~\ref{eq:low_deg}, and combine it with Eq.~\ref{eq:K_sqr_minus_pi4} into another integral identity.\begin{corollary}[An Integral Involving $ \pi$, $G$ and $ \zeta(3)$]We have the following formula:\label{cor:KE_pi_G_Apery}\begin{align}\int_0^1\frac{\mathbf K(k)}{k}\left[ 1-\frac{2}{\pi}\mathbf E(k) \right]\D k=\frac{\pi}{4}-2G+\frac{7\zeta(3)}{2\pi}.\label{eq:KE_pi_G_zeta3}\end{align}\end{corollary}\begin{proof}First, by virtue of the ordinary differential equation \begin{align}\frac{\D}{\D t}\left[ t(1-t)\frac{\D \mathbf K(\sqrt{t})}{\D t} \right]=\frac{ \mathbf K(\sqrt{t})}{4},\label{eq:LegendreDiffEq_half}\end{align}we may apply integration by parts to the following integral:\begin{align*} \int_0^1[\mathbf K(\sqrt{t})]^2\log t\D t={}&4\int_0^1\mathbf K(\sqrt{t})\log t\frac{\D}{\D t}\left[t(1-t)\frac{\D \mathbf K(\sqrt{t})}{\D t}\right]\D t=-4\int_0^1t(1-t)\frac{\D\left[\mathbf K(\sqrt{t})\log t\right]}{\D t}\frac{\D \mathbf K(\sqrt{t})}{\D t}\D t\notag\\={}&4\lim_{t\to0^{+}}\left\{t(1-t)\mathbf K(\sqrt{t})\frac{\D\left[\mathbf K(\sqrt{t})\log t\right]}{\D t}\right\}+4\int_0^1\mathbf K(\sqrt{t})\frac{\D }{\D t}\left\{t(1-t)\frac{\D\left[\mathbf K(\sqrt{t})\log t\right]}{\D t}\right\}\D t\notag\\={}&\pi^2+4\int_0^1\frac{\mathbf K(\sqrt{t})}{t}[\mathbf E(\sqrt{t})-\mathbf K(\sqrt{t})]\D t+\int_0^1[\mathbf K(\sqrt{t})]^2\log t\D t,\end{align*}which leads to\begin{align}\label{eq:KEK_pi8}\int_0^1\frac{\mathbf K(k)}{k}\left[\mathbf  E(k) -\mathbf K(k)\right]\D k=\frac{1}{2}\int_0^1\frac{\mathbf K(\sqrt{t})}{t}[\mathbf E(\sqrt{t})-\mathbf K(\sqrt{t})]\D t=-\frac{\pi^{2}}{8}.\end{align}Then, as we add Eq.~\ref{eq:KEK_pi8} to the left-hand side of Eq.~\ref{eq:KE_pi_G_zeta3}, before invoking Eqs.~\ref{eq:low_deg} and \ref{eq:K_sqr_minus_pi4}   subsequently, we may compute\begin{align*}\int_0^1\frac{\mathbf K(k)}{k}\left[ 1-\frac{2}{\pi}\mathbf E(k) \right]\D k={}&\int_0^1\frac{\mathbf K(k)}{k}\left[ 1-\frac{2}{\pi}\mathbf E(k) \right]\D k+\frac{\pi}{4}+\frac{2}{\pi}\int_0^1\frac{\mathbf K(k)}{k}\left[\mathbf  E(k) -\mathbf K(k)\right]\D k\notag\\={}&\frac{\pi}{4}+\int_0^1\frac{\mathbf K(k)}{k}\left[ 1-\frac{2}{\pi}\mathbf K(k) \right]\D k=\frac{\pi}{4}+\int_0^1\frac{\mathbf K(k)+\frac{\pi}{2}}{k}\left[ 1-\frac{2}{\pi}\mathbf K(k) \right]\D k-2G+\pi\log2\notag\\={}&\frac{\pi}{4}-2G+\pi\log2-\frac{2}{\pi}\int_{0}^{1}\left\{[\mathbf K(k)]^{2}-\frac{\pi^2}{4}\right\}\frac{ \D k}{k}=\frac{\pi}{4}-2G+\frac{7\zeta(3)}{2\pi},\end{align*}as claimed.\qed\end{proof}

The  reduction of a logarithmic factor in the integrand of  Eq.~\ref{eq:K_sqr_log} is not purely accidental. It can be viewed as a consequence of  Parseval-type identities for the Tricomi transform (see \cite[Ref.][\S4.3, Eqs.~2 and 4]{TricomiInt} or \cite[Ref.][Eqs.~11.237 and 11.241]{KingVol1}):\begin{align}\int_{-1}^1 f(x)(\widehat{\mathcal T}g)(x)\D x+\int_{-1}^1 g(x)(\widehat{\mathcal T}f)(x)\D x={}&0,\label{eq:Tricomi_Parseval1}\\\int_{-1}^1 f(x)g(x)\D x-\int_{-1}^1 (\widehat{\mathcal T}f)(x)(\widehat{\mathcal T}g)(x)\D x={}&\frac{1}{\pi}\int_{-1}^1\left[ f(x)(\widehat{\mathcal T}g)(x) +g(x)(\widehat{\mathcal T}f)(x)\right]\log\frac{1+x}{1-x}\D x,\label{eq:Tricomi_Parseval2}\end{align}where $ f\in L^p(-1,1),p>1$; $ g\in L^q(-1,1),q>1$ and $ \frac1p+\frac1q<1$, where  the Tricomi transform is defined as a Cauchy principal value\begin{align*}(\widehat{\mathcal T} f)(x):=\mathcal P\int_{-1}^1\frac{f(\xi)\D \xi}{\pi(x-\xi)},\quad\text{a.e. } x\in(-1,1),\end{align*}and induces a bounded  linear operator $ \widehat {\mathcal T}:L^p(-1,1)\longrightarrow L^p(-1,1)$ for $ 1<p<+\infty$ \cite[Ref.][p.~188]{SteinWeiss}. In the next corollary, we  explain this ``Tricomi pairing'' by putting   Eq.~\ref{eq:K_sqr_log} in a slightly more general context. \begin{corollary}[Some Applications of Tricomi Pairing]The following chain of identities hold:\begin{align}
&\int_0^1\left[\mathbf K\left(\sqrt{1-\kappa^2}\right)\right]^2\D \kappa-\int_0^1[\mathbf K(k)]^2\D k=\frac{2}{\pi}\int_0^1\mathbf K(\xi)\mathbf K\left( \sqrt{1-\xi^2} \right)\log\frac{1+\xi}{1-\xi}\D\xi=\frac{2}{\pi}\int_0^1\mathbf K(k)\mathbf K\left( \sqrt{1-k^2} \right)\log\frac{1}{k}\D k\notag\\={}&\frac{1}{\pi}\int_0^1\frac{[\mathbf K(k)]^2-[\mathbf K(\sqrt{1-k^2})]^{2}}{\sqrt{1-k^{2}}}\log k\D k=\frac{1}{\pi}\int_0^1\frac{\mathbf K(k)\mathbf K(\sqrt{1-k^{2}})}{\sqrt{1-k^{2}}}\arccos(1-2k^2)\D k,\label{eq:K_sqr_log_chain}
\end{align}where $ \arccos x\in[0,\pi]$ for $ x\in[-1,1]$.\end{corollary}\begin{proof}In \cite[Ref.][Corollary 3.2(b)]{Zhou2013Pnu}, we showed that \begin{align}\label{eq:BT0}\mathbf K(k)=\mathcal P\int_{-1}^1\frac{\mathbf K(\sqrt{1-\kappa^{2}})\D\kappa}{\pi(k-\kappa)}=\frac{k}{\pi}\mathcal P\int_{-1}^1\frac{\mathbf K(\sqrt{1-\kappa^{2}})\D\kappa}{k^{2}-\kappa^{2}},\quad 0<k<1,\end{align}which entails a Tricomi transform formula:\begin{align}\mathcal P\int_{-1}^1\frac{\mathbf K(\sqrt{1-\kappa^{2}})\D\kappa}{\pi(k-\kappa)}=\frac{k}{|k|}\mathbf K(k),\quad k\in(-1,0)\cup(0,1).\label{eq:BT}\end{align}

Now, setting $ f(x)=g(x)=\mathbf K(\sqrt{1-x^2})$ in Eq.~\ref{eq:Tricomi_Parseval2}, and reading off $ (\widehat{\mathcal T} f)(x)=(\widehat{\mathcal T} g)(x)=\frac{x}{|x|}\mathbf K(x)$ from Eq.~\ref{eq:BT}, one obtains the first equality of Eq.~\ref{eq:K_sqr_log_chain}:
\begin{align*}\int_0^1\left[\mathbf K\left(\sqrt{1-\kappa^2}\right)\right]^2\D \kappa-\int_0^1[\mathbf K(k)]^2\D k=\frac{2}{\pi}\int_0^1\mathbf K(\xi)\mathbf K\left( \sqrt{1-\xi^2} \right)\log\frac{1+\xi}{1-\xi}\D\xi.\end{align*} Using the variable substitution $ \xi=(1-k)/(1+k)$ along with Landen's transformations for $ 2\mathbf K((1-k)/(1+k))=(1+k)\mathbf K(\sqrt{1-k^2})$ and $ \mathbf K(2\sqrt{k}/(1+k))=(1+k)\mathbf K(k)$, we may convert the rightmost term in the equation above into\begin{align*}\frac{2}{\pi}\int_0^1\mathbf K(\xi)\mathbf K\left( \sqrt{1-\xi^2} \right)\log\frac{1+\xi}{1-\xi}\D\xi=\frac{2}{\pi}\int_0^1\mathbf K(k)\mathbf K\left( \sqrt{1-k^2} \right)\log\frac{1}{k}\D k,\end{align*}thereby completing the first line in Eq.~\ref{eq:K_sqr_log_chain}.

To move on to the next term of Eq.~\ref{eq:K_sqr_log_chain}, we fall back on the following formula (see \cite[Ref.][Eq.~26]{Wan2012} or \cite[Ref.][Eq.~41$^*$]{Zhou2010})\begin{align}\label{eq:K_sqr_star}
[\mathbf K(\sqrt{t})]^2={}&\frac{2}{\pi}\int_0^1\frac{\mathbf K(\sqrt{\mathstrut\mu})\mathbf K(\sqrt{\mathstrut1-\mu})\D \mu}{1-\mu t},\quad 0\leq t<1
\end{align} to compute \begin{align*}&\frac{1}{\pi}\int_0^1\frac{[\mathbf K(k)]^2-[\mathbf K(\sqrt{1-k^2})]^{2}}{\sqrt{1-k^{2}}}\log k\D k=\frac{1}{\pi}\int_0^1\frac{[\mathbf K(k)]^2}{\sqrt{1-k^{2}}}\log \frac{k}{\sqrt{1-k^{2}}}\D k=\frac{1}{\pi}\int_0^{\pi/2}[\mathbf K(\sin\theta)]^2\log \tan\theta\D \theta\notag\\={}&\frac{2}{\pi^2}\int_0^{\pi/2}\left[\int_0^1\frac{\mathbf  K(\sqrt{\mathstrut\mu})\mathbf K(\sqrt{\mathstrut1-\mu})\D\mu}{1-\mu\sin^2\theta}\right]\log \tan\theta\D \theta=-\frac{1}{2\pi}\int_0^1\frac{\mathbf  K(\sqrt{\mathstrut\mu})\mathbf K(\sqrt{\mathstrut1-\mu})\log(1-\mu)\D\mu}{\sqrt{1-\mu}},\end{align*}where the  integration over $ \theta\in[0,\pi/2]$ can be rendered in closed form. A substitution $ k=\sqrt{1-\mu}$ in the last step will bring the equation above to the last term in the first line of Eq.~\ref{eq:K_sqr_log_chain}. In particular, this result also hearkens back to Eq.~\ref{eq:K_sqr_log}.

Now, we note that the following Tricomi transform formula  (see \cite[Ref.][Corollary~3.4]{Zhou2013Pnu}) \begin{align*}\mathcal P\int_{-1}^1\mathbf K\left( \sqrt{\frac{1+\xi}{2}} \right)\mathbf K\left( \sqrt{\frac{1-\xi}{2}} \right)\frac{ 2\D \xi}{\pi (x-\xi)}={}&\left[\mathbf K \left(\sqrt{\frac{\vphantom{\xi}1+x}{2}}\right) \right]^2-\left[\mathbf K \left(\sqrt{\frac{\vphantom{\xi}1-x}{2}}\right) \right]^2,\quad -1<x<1\end{align*}allows us to specialize  Eq.~\ref{eq:Tricomi_Parseval1} into\begin{align}\int_{-1}^1f(x)
\left\{ \left[\mathbf K \left(\sqrt{\frac{\vphantom{\xi}1+x}{2}}\right) \right]^2-\left[\mathbf K \left(\sqrt{\frac{\vphantom{\xi}1-x}{2}}\right) \right]^2 \right\}\D x={}&-2\int_{-1}^1(\widehat {\mathcal T}f)(x)
\mathbf K\left( \sqrt{\frac{\vphantom{\xi}1+x}{2}} \right)\mathbf K\left( \sqrt{\frac{\vphantom{\xi}1-x}{2}} \right)\D x\label{eq:TP3}\end{align}for arbitrary $ f\in L^p(-1,1),p>1$.
From \cite[Ref.][Eqs.~12A.5 and 12A.12]{KingVol2}, we know that $ \sqrt{1-x^2}f(x)=\log\frac{1+x}{2}$ entails  $ \sqrt{1-x^2}(\widehat {\mathcal T}f)(x)=-\arccos x$, so Eq.~\ref{eq:TP3} further specializes to \begin{align*}&\frac{1}{\pi}\int_0^1\frac{[\mathbf K(k)]^2-[\mathbf K(\sqrt{1-k^2})]^{2}}{\sqrt{1-k^{2}}}\log k\D k=\frac{1}{4\pi}\int_{-1}^1\frac{\log\frac{1+x}{2}}{ \sqrt{1-x^2}}
\left\{ \left[\mathbf K \left(\sqrt{\frac{\vphantom{\xi}1+x}{2}}\right) \right]^2-\left[\mathbf K \left(\sqrt{\frac{\vphantom{\xi}1-x}{2}}\right) \right]^2 \right\}\D x\notag\\={}&\frac{1}{2\pi}\int_{-1}^1\frac{\arccos x}{ \sqrt{1-x^2}}
\mathbf K\left( \sqrt{\frac{\vphantom{\xi}1+x}{2}} \right)\mathbf K\left( \sqrt{\frac{\vphantom{\xi}1-x}{2}} \right)\D x=\frac{1}{\pi}\int_0^1\frac{\mathbf K(k)\mathbf K(\sqrt{1-k^{2}})}{\sqrt{1-k^{2}}}\arccos(1-2k^2)\D k,\end{align*}which completes the second line of  Eq.~\ref{eq:K_sqr_log_chain}. \qed\end{proof}
 \section{Integrations on $ S^3\times S^3$\label{subsec:S3}}\subsection{Ultraspherical Harmonic Expansions and Geometric Parametrizations\label{subsubsec:S3xS3}}
Now, we shall extend the geometric approach to integrations on the Cartesian product of hyperspheres $ S^n\times S^n$, where $n>2$. The usual spherical coordinates have a natural extension to the unit hypersphere $ S^n:=\{\bm x=(x_1,\dots,x_{n+1})\in\mathbb R^{n+1}|\sum_{j=1}^{n+1}x^2_j=1\}$ as \[\begin{cases}x_{1}=\cos\theta_1   \\
x_2=\sin\theta_1\cos\theta_2  \\
x_3=\sin\theta_1\sin\theta_2\cos\theta_3  \\\;\;\quad
\vdots  \\x_{n}=\sin\theta_1\cdots\sin\theta_{n-1}\cos\theta_n\\
x_{n+1}=\sin\theta_1\cdots\sin\theta_{n-1}\sin\theta_{n} \\
\end{cases}(0\leq\theta_j\leq\pi,j=1,2,\dots,n-1;0\leq\theta_n<2\pi)\]so that the standard surface element (induced Lebesgue measure) can be put  into\[\D\sigma=\prod_{j=1}^n\sin^{n-j}\theta_j\D\theta_j.\]In place of the spherical harmonics $ Y_{\ell m}$, we need the Gegenbauer polynomials $ C^{(\lambda)}_\ell(\xi)$ defined by the generating function\[\frac1{(1-2t\xi+t^2)^\lambda}=\sum_{\ell=0}^\infty t^\ell C_\ell^{(\lambda)}(\xi),\quad 0\leq t<1,-1\leq\xi\leq 1.\]\begin{lemma}[Convolution on Hyperspheres]\label{lm:Sn_int}Let $ S^n\subset\mathbb R^{n+1}$ $(n>2)$ be the unit hypersphere equipped with the standard surface element $ \D\sigma$ and surface area $ m(S^{n}):=\int_{S^n}\D\sigma=2\pi^{(n+1)/2}/\Gamma(\frac{n+1}{2})$,  then  the product of the following three functions defined for $ \bm n,\bm n_1,\bm n_2,\bm n'\in S^n$\[f(\bm n\cdot\bm n_{1})=f_{0}+\sum_{\ell=1}^\infty \frac{f^{\phantom{\frac12}\!\!\!}_\ell C_{\ell}^{(\frac{n-1}{2})}(\bm n\cdot\bm n_{1})}{n-1},\quad g(\bm n_{1}\cdot\bm n_{2})=g_{0}+\sum_{\ell=1}^\infty \frac{g^{\phantom{\frac12}\!\!\!}_\ell C_{\ell}^{(\frac{n-1}{2})}(\bm n_{1}\cdot\bm n_{2})}{n-1},\quad h(\bm n_{2}\cdot\bm n')=h_{0}+\sum_{\ell=1}^\infty \frac{h^{\phantom{\frac12}\!\!\!}_\ell C_{\ell}^{(\frac{n-1}{2})}(\bm n_{2}\cdot\bm n')}{n-1}\]will integrate to\begin{align*}\int_{S^n}\frac{\D \sigma_1}{m(S^{n})}\int_{S^n}\frac{\D \sigma_2}{m(S^{n})}f(\bm n\cdot\bm n_{1})g(\bm n_{1}\cdot\bm n_{2})h(\bm n_{2}\cdot\bm n')=f_{0}g_0h_0+\sum_{\ell=1}^\infty f_\ell g_\ell h_\ell  \frac{C_{\ell}^{(\frac{n-1}{2})}(\bm n\cdot\bm n')}{(n-1)(2\ell+n-1)^2},\end{align*}so long as both sides converge. Furthermore, the above integral formula remains valid for $ n=2$. \end{lemma}\begin{proof}To complete the integration with respect to $ \bm n_1\in S^n$ $(n>2)$, we exploit the addition theorem of Gegenbauer polynomials (see 8.934.3 in \cite{GradshteynRyzhik} or \S5.3.2 in \cite{MOS}) in the form of\begin{align*}C_{\ell}^{(\frac{n-1}{2})}(\bm n_{1}\cdot\bm n_{2})={}&\frac{(n-3)!}{[\Gamma(\frac{n-1}{2})]^{2}}\sum_{j=0}^\ell\frac{2^{2j}(\ell-j)![\Gamma(\frac{n-1}{2}+j)]^2}{(n-2+\ell+j)!}(n-2+2j)\left[ 1- ( \bm n_{1}\cdot\bm n)^{2}\right]^{j/2}\left[ 1- ( \bm n\cdot\bm n_{2})^2\right]^{j/2}\times\notag\\&\times C_{\ell-j}^{(\frac{n-1}{2}+j)}( \bm n_{1}\cdot\bm n)C_{\ell-j}^{(\frac{n-1}{2}+j)}( \bm n\cdot\bm n_{2}) C_{j}^{(\frac{n-2}{2})}\left(\frac{\bm n_{1}-\bm n(\bm n_{1}\cdot\bm n)}{\left\vert \bm n_{1}-\bm n(\bm n_{1}\cdot\bm n)\right\vert}\cdot\frac{\bm n_{2}-\bm n(\bm n_{2}\cdot\bm n)}{\left\vert \bm n_{2}-\bm n(\bm n_{2}\cdot\bm n)\right\vert}\right).\end{align*}  We assert that only the $ j=0$ term in the finite sum on the right-hand side of the equation above will eventually contribute to the integral over $ \bm n_1\in S^n$. This  is because we may rotate $ \bm n_1$ about the axis  $\bm n$, without altering the value of $ \bm n\cdot\bm n_1$, amounting to a latitude-longitude breakdown of the integral on hypersphere in the following manner:\begin{align*}&\int_{S^{n}}{{f(\bm n\cdot\bm n_{1}) \left[ 1- ( \bm n_{1}\cdot\bm n)^{2}\right]^{j/2} C_{\ell-j}^{(\frac{n-1}{2}+j)}( \bm n_{1}\cdot\bm n)}} C_{j}^{(\frac{n-2}{2})}\left(\frac{\bm n_{1}-\bm n(\bm n_{1}\cdot\bm n)}{\left\vert \bm n_{1}-\bm n(\bm n_{1}\cdot\bm n)\right\vert}\cdot\frac{\bm n_{2}-\bm n(\bm n_{2}\cdot\bm n)}{\left\vert \bm n_{2}-\bm n(\bm n_{2}\cdot\bm n)\right\vert}\right)\frac{\D \sigma_1}{m(S^{n})} \notag\\={}&\int_{-1}^{ 1} f(\xi) ( 1- \xi^{2})^{(n-2+j)/2} C_{\ell-j}^{(\frac{n-1}{2}+j)}(\xi)\left[\int_{|\bm s|=1,\bm s\cdot\bm n=0}C_{j}^{(\frac{n-2}{2})}\left(\bm s\cdot\frac{\bm n_{2}-\bm n(\bm n_{2}\cdot\bm n)}{\left\vert \bm n_{2}-\bm n(\bm n_{2}\cdot\bm n)\right\vert}\right)\D\sigma_{\bm s}\right]\frac{\D \xi}{m(S^{n})}=0,\quad j=1,2,\dots,\ell.\end{align*}Here in the last line, $ \D\sigma_{\bm s}$ stands for the surface element of $ S^{n-1}$, which forms the ``equator'' with respect to the ``north pole'' $\bm n$; the  integral on $ \D\sigma_{\bm s}$ vanishes due to the orthogonality relation of Gegenbauer polynomials (see 8.904 in \cite{GradshteynRyzhik}):\[\int_0^{\pi}C_{j}^{(\frac{n-2}{2})}(\cos\varphi)\sin^{n-2}\varphi\D\varphi=\int_0^{\pi}C_{j}^{(\frac{n-2}{2})}(\cos\varphi)C_{0}^{(\frac{n-2}{2})}(\cos\varphi)\sin^{n-2}\varphi\D\varphi=0,\quad n>2,\;j=1,2,\dots,\ell.\] Now, to compute the surface integral\begin{align*}\int_{S^n}\frac{\D \sigma_1}{m(S^{n})}f(\bm n\cdot\bm n_{1})g(\bm n_{1}\cdot\bm n_{2})=\int_{S^n}\frac{\D \sigma_1}{m(S^{n})}f(\bm n\cdot\bm n_{1})\left[g_0+\sum_{\ell=1}^\infty  \frac{(n-2)!\ell!}{(n-2+\ell)!}\frac{g^{\phantom{\frac12}\!\!\!}_\ell C_{\ell}^{(\frac{n-1}{2})}(\bm n_{1}\cdot\bm n)C_{\ell}^{(\frac{n-1}{2})}(\bm n\cdot\bm n_{2})}{n-1}\right],\end{align*}we employ the orthogonality relation for the Gegenbauer polynomials (see 8.904 in \cite{GradshteynRyzhik})\[\int_{-1}^1C_{\ell'}^{(\frac{n-1}{2})}(\xi)C_\ell^{(\frac{n-1}{2})}(\xi)(1-\xi^{2})^{\frac{n-2}{2}}\D \xi=\frac{\pi(n-2+\ell)!}{2^{n-2}\ell!(\frac{n-1}{2}+\ell)[\Gamma(\frac{n-1}{2})]^2}\delta_{\ell'\ell}.\]Especially, with the latitude-longitude decomposition, we have the relation\begin{align*}\int_{S^n}\frac{\D \sigma_1}{m(S^{n})}C_{\ell'}^{(\frac{n-1}{2})}(\bm n\cdot\bm n_1)C_{\ell}^{(\frac{n-1}{2})}(\bm n_{1}\cdot\bm n)=\frac{m(S^{n-1})}{m(S^{n})}\int_{-1}^1C_{\ell'}^{(\frac{n-1}{2})}(\xi)C_\ell^{(\frac{n-1}{2})}(\xi)(1-\xi^{2})^{\frac{n-2}{2}}\D \xi=\frac{(n-2+\ell)!}{\ell!(n-2)!}\frac{n-1}{2\ell+n-1}\delta_{\ell'\ell},\end{align*}which proves the identity\[\int_{S^n}\frac{\D \sigma_1}{m(S^{n})}f(\bm n\cdot\bm n_{1})g(\bm n_{1}\cdot\bm n_{2})=f_{0}g_0+\sum_{\ell=1}^\infty f_\ell g_\ell\frac{C_{\ell}^{(\frac{n-1}{2})}(\bm n\cdot\bm n_{2})}{(n-1)(2\ell+n-1)},\] so long as both sides exist in the distributional sense. Yet another application of the addition theorem for Gegenbauer polynomials will  lead us to the formula\[\int_{S^n}\frac{\D \sigma_1}{m(S^{n})}\int_{S^n}\frac{\D \sigma_2}{m(S^{n})}f(\bm n\cdot\bm n_{1})g(\bm n_{1}\cdot\bm n_{2})h(\bm n_{2}\cdot\bm n')=f_{0}g_0h_0+\sum_{\ell=1}^\infty f_\ell g_\ell h_\ell  \frac{C_{\ell}^{(\frac{n-1}{2})}(\bm n\cdot\bm n')}{(n-1)(2\ell+n-1)^2},\] as claimed. It is worth noting that for $ n=2$, we have precisely\begin{align*}C_{\ell}^{(\frac{1}{2})}(\bm n\cdot\bm n')=P_{\ell}(\bm n\cdot\bm n')=4\pi\sum_{m=-\ell}^\ell\frac{\overline{Y_{\ell m}(\theta,\phi)}Y_{\ell m}(\theta',\phi')}{2\ell+1}\end{align*}for  $ \bm n(\theta,\phi)=(\sin\theta\cos\phi,\sin\theta\sin\phi,\cos\theta)\in S^2$,   $ \bm n'(\theta',\phi')=(\sin\theta'\cos\phi',\sin\theta'\sin\phi',\cos\theta')\in S^2$, so the current lemma is indeed a generalization for the technique of spherical harmonic expansion employed in Lemma~\ref{lm:S2}. \qed\end{proof}\begin{remark}The integral formula in this lemma remains true in the ``$ n\to1^+$ limit''. After we recall that, for every positive integer $\ell$, the Gegenbauer polynomial has limit behavior\[\lim_{n\to1^+}\frac{C_{\ell}^{(\frac{n-1}{2})}(\cos\phi)}{n-1}=\frac{\cos\ell\phi}{\ell},\]we may  immediately formulate an analogous integration formula on  $S^1\times S^1 $: the product of the  three functions defined for $ \bm n=(\cos\phi,\sin\phi),\bm n_1=(\cos\phi_{1},\sin\phi_{1}),\bm n_2=(\cos\phi_{2},\sin\phi_{2}),\bm n'=(\cos\phi',\sin\phi')\in S^1$\[f(\bm n\cdot\bm n_{1})=f_{0}+\sum_{\ell=1}^\infty f_\ell\frac{\cos\ell(\phi-\phi_{1})}{\ell},\quad g(\bm n_{1}\cdot\bm n_{2})=g_{0}+\sum_{\ell=1}^\infty g_\ell\frac{\cos\ell(\phi_{1}-\phi_{2})}{\ell},\quad h(\bm n_{2}\cdot\bm n')=h_{0}+\sum_{\ell=1}^\infty h_\ell\frac{\cos\ell(\phi_{2}-\phi')}{\ell}\]integrates to\begin{align*}\int_{S^1}\frac{\D \phi_1}{2\pi}\int_{S^1}\frac{\D \phi_2}{2\pi}f(\bm n\cdot\bm n_{1})g(\bm n_{1}\cdot\bm n_{2})h(\bm n_{2}\cdot\bm n')=f_{0}g_0h_0+\sum_{\ell=1}^\infty f_\ell g_\ell h_\ell  \frac{\cos\ell(\phi-\phi')}{4\ell^{3}}.\end{align*}Obviously, the truth of such a statement can be established by the elementary addition theorem of cosines and the orthogonality relation for Fourier series, which are just degenerate cases of the addition theorem and orthogonality relation for Gegenbauer polynomials.\eor\end{remark}

In the next proposition, we  narrow down to the case of $ S^3\times S^3$, where the unit hypersphere $ S^3$ is customarily  parametrized by $ \bm n(\psi,\theta,\phi)=(\sin\psi\sin\theta\cos\phi,\sin\psi\sin\theta\sin\phi,\sin\psi\cos\theta,\cos\psi)$, along with the standard surface element $ \D\sigma=(\sin^2\psi\D\psi)(\sin\theta\D\theta\D\phi)
$ and total area $ \int_{S^3}\D\sigma=\int_{0}^\pi\sin^2\psi\D\psi\int_{0}^\pi\sin\theta\D\theta\int_{0}^{2\pi}\D\phi=2\pi^2$. For $ n=3$, the Gegenbauer polynomial $ C_{\ell}^{(1)}(\xi)$ is identical to the Tchebychev polynomial of the second kind $ U_\ell^{\phantom{1}}(\xi)$, satisfying\[C_{\ell}^{(1)}(\cos\otherTheta)=U_\ell^{\phantom{1}}(\cos\otherTheta)=\begin{cases}\ell+1, & \otherTheta=0; \\[4pt]
\dfrac{\sin(\ell+1)\otherTheta}{\sin\otherTheta}, & 0<\otherTheta<\pi; \\[8pt]
(-1)^{\ell} (\ell+1),& \otherTheta=\pi. \\
\end{cases}\]  \begin{proposition}[Some Integrals on $ S^3\times S^3$]\label{prop:S3}\begin{enumerate}[label=\emph{(\alph*)}, ref=(\alph*), widest=a]\item Let $ S^3\subset\mathbb R^{4}$  be the standard unit hypersphere,
  on which two unit vectors $ \bm n\in S^3$ and $ \bm n_\bot\in S^3$ are orthogonal to each other $($\emph{i.e.}~$ \bm n\cdot\bm n_\bot=0)$, then we have an integral representation of Catalan's constant\begin{align}G={}&\int_{S^3}\frac{\D \sigma_1}{2\pi^{2}}\int_{S^3}\frac{\D \sigma_2}{2\pi^{2}}\frac{1}{|\bm n-\bm n_1|^2}\frac{1}{|\bm n_{1}-\bm n_2|^{2}}\frac{1}{|\bm n_{2}-\bm n_\bot|^2}=\frac{1}{4\pi}\int_{0}^\pi(\pi-\psi)\log\frac{1+\sin\psi}{1-\sin\psi}\D\psi.\label{eq:G_S3}\end{align}
  Generally, if two unit vectors $ \bm n\in S^3$ and $ \bm n_\otherTheta\in S^3$ are separated by an angle $ \otherTheta$ $($\emph{i.e.}~$ \bm n\cdot\bm n_\otherTheta=\cos\otherTheta)$ where $ 0<\otherTheta<\pi$, we have the identity  \begin{align}\frac{\I[\dilog (ke^{i\otherTheta})]}{k\sin\otherTheta}=\sum_{\ell=1}^\infty\frac{k^{\ell-1}\sin\ell\otherTheta}{\ell^2\sin\otherTheta}={}&\int_{S^3}\frac{\D \sigma_1}{2\pi^{2}}\int_{S^3}\frac{\D \sigma_2}{2\pi^{2}}\frac{1}{|\bm n-\bm n_1|^2}\frac{1}{|\bm n_{1}-k\bm n_2|^{2}}\frac{1}{|\bm n_{2}-\bm n_\otherTheta|^2}\notag\\={}&\frac{1}{2\pi k\sin\otherTheta}\int_{0}^\pi\left[\arctan\frac{k\sin\psi}{1-k\cos\psi}\right]
  \left[\log\frac{1-\cos(\psi+\otherTheta)}
  {1-\cos(\psi-\otherTheta)}\right]\D\psi,\quad 0<k\leq 1,\label{eq:ImLi2}\end{align}
  where \[\dilog(z):=\sum_{n=1}^\infty\frac{z^n}{n^2},\quad |z|\leq 1\]defines  the dilogarithm function for complex valued $ z\in\mathbb C$.
\item
We have the following integral formulae:\begin{align}\int_0^1\frac{ z\mathbf K( \sqrt{1-t} )}{\sqrt{(1-z^{2})^{2}+4z^{2}t}}\D t={}&\dilog (z)-\dilog (-z),\quad |z|\leq1\label{eq:Li2int}\\\int_0^1\frac{ (1-z^{4} )\mathbf K( \sqrt{1-t} )}{[(1-z^{2})^{2}+4z^{2}t]\sqrt{(1-z^{2})^{2}+4z^{2}t}}\D t={}&\frac{1}{z}\log\frac{1+z}{1-z},\quad |z|<1\label{eq:Li2int'}\\\frac{1}{2}\int_0^1 \mathbf K( \sqrt{1-t} )\log \frac{\sqrt{(1-z^{2})^{2}+4z^{2}t}+2 t+z^2-1}{2 t}\D t={}&z \dilog(z)-z \dilog(-z)-(1+z) \log (1+z)-(1-z) \log (1-z),\quad|z|\leq1\label{eq:Li2int_S}\end{align}where the univalent branches of the square root and logarithm functions are chosen conventionally. As a consequence, Catalan's constant can be represented as\begin{align}G={}&\frac{\pi}{4}-\frac{1}{2}\log2-\frac{1}{4}\int_0^1 \mathbf K( \sqrt{1-t} )\log \frac{\sqrt{1-t}(1-\sqrt{1-t})}{ t}\D t,\label{eq:G_Klog}\\G={}&\frac{1}{\pi}\int_0^{\pi/2}\D\theta\int_0^{\pi/2}\D\phi\frac{\sin\theta(1-\cos^4\theta)\sin\phi\cos\phi\mathbf K(\sin\theta)\mathbf K(\sin\phi)}{(\sin^4\theta+4\cos^2\theta\cos^2\phi)^{3/2}},\label{eq:G_KK}\\\text{and }\quad G={}&\frac{1}{2}\int_0^{\pi/2}\D\theta\int_0^{\pi/2}\D\phi\frac{\cos\theta(1-\cos^4\theta)\sin\phi\cos\phi\mathbf K(\sin\phi)}{(\sin^4\theta+4\cos^2\theta\cos^2\phi)^{3/2}}.\label{eq:G_K}\end{align}\end{enumerate}\end{proposition}\begin{proof}
\begin{enumerate}[label=(\alph*)]\item Specializing Lemma~\ref{lm:Sn_int} to the case $ n=3$ with
 \[f_0=g_0=h_0=1;\quad f_\ell=g_\ell=h_\ell=2,\;\forall\ell\in\mathbb Z\cap[1,+\infty),\]and using the special value of the Gegenbauer polynomial
$ C_\ell^{(1)}(0)=U_{\ell}^{\phantom{1}}(0)=\sin\frac{(\ell+1)\pi}{2}$, we may readily confirm the left half of  Eq.~\ref{eq:G_S3}:
\begin{align*}\int_{S^3}\frac{\D \sigma_1}{2\pi^{2}}
\int_{S^3}\frac{\D \sigma_2}{2\pi^{2}}\frac{1}{|\bm n-\bm n_1|^2}
\frac{1}{|\bm n_{1}-\bm n_2|^{2}}
\frac{1}{|\bm n_{2}-\bm n_\bot|^2}=\sum_{\ell=0}^\infty\frac{4\sin\frac{(\ell+1)\pi}{2}}{(2\ell+2)^2}=\sum_{\ell=0}^\infty\frac{(-1)^\ell}{(2\ell+1)^2}=G,\end{align*}where  $ \bm n\cdot\bm n_\bot=0$.
More generally, if we start from $ g_\ell=2k^\ell$ instead of $ g_\ell=2$ for positive integers $\ell$, we may derive \[\int_{S^3}\frac{\D \sigma_1}{2\pi^{2}}\int_{S^3}\frac{\D \sigma_2}{2\pi^{2}}\frac{1}{|\bm n-\bm n_1|^2}\frac{1}{|\bm n_{1}-k\bm n_2|^{2}}\frac{1}{|\bm n_{2}-\bm n_\otherTheta|^2}=\sum_{\ell=1}^\infty\frac{k^{\ell-1}\sin\ell\otherTheta}{\ell^2\sin\otherTheta},\quad 0<k\leq 1,\]for $ \bm n\cdot\bm n_\otherTheta=\cos\otherTheta $, which is the first line of Eq.~\ref{eq:ImLi2}.

 Alternatively, we can evaluate the integral on $ S^3\times S^3$ with an admixture of ultraspherical harmonic expansion and geometric  parametrization. Following the proof of Lemma~\ref{lm:Sn_int}, we compute \[\int_{S^3}\frac{\D \sigma_1}{2\pi^{2}}\frac{1}{|\bm n-\bm n_1|^2}\frac{1}{|\bm n_{1}-k\bm n_2|^{2}}=\sum_{\ell=1}^\infty\frac{k^{\ell-1}\sin\ell\otherPsi}{\ell\sin\otherPsi}=\frac{1}{k\sin\otherPsi}\arctan\frac{k\sin\otherPsi}{1-k\cos\otherPsi},\]where $ \bm n\cdot\bm n_2=\cos\otherPsi$. We then complete the integration over $ \bm n_2$ using the coordinates $ \bm n=(0,0,0,1),\bm n_{\otherTheta}=(0,0,\sin\otherTheta,\cos\otherTheta)$ and $ \bm n_2=(\sin\psi_{2}\sin\theta_{2}\cos\phi_{2},\sin\psi_{2}\sin\theta_{2}\sin\phi_{2},\sin\psi_{2}\cos\theta_{2},\cos\psi_{2})$. This amounts to a verification of the second line in Eq.~\ref{eq:ImLi2}:\begin{align*}\frac{\I[\dilog (ke^{i\otherTheta})]}{k\sin\otherTheta}={}&\int_{S^3}\frac{\D \sigma_2}{2\pi^{2}}\frac{1}{|\bm n_{2}-\bm n_\otherTheta|^2k\sin\psi_{2}}\arctan\frac{k\sin\psi_{2}}{1-k\cos\psi_{2}}
\notag\\={}&\frac{1}{4\pi^2}\int_{0}^\pi\sin^2\psi_{2}\D\psi_{2}\int_{0}^\pi\sin\theta_{2}\D\theta_{2}\int_{0}^{2\pi}\D\phi_{2}\frac{\arctan\frac{k\sin\psi_{2}}{1-k\cos\psi_{2}}}{(1-\cos\psi_{2}\cos\otherTheta-\sin\psi_{2}\cos\theta_{2}\sin\otherTheta)k\sin\psi_{2}}\notag\\={}&\frac{1}{2\pi}\int_{0}^\pi\sin^2\psi_{2}\D\psi_{2}\int_{0}^\pi\sin\theta_{2}\D\theta_{2}\frac{\arctan\frac{k\sin\psi_{2}}{1-k\cos\psi_{2}}}{(1-\cos\psi_{2}\cos\otherTheta-\sin\psi_{2}\cos\theta_{2}\sin\otherTheta)k\sin\psi_{2}}\notag\\={}&\frac{1}{2\pi k\sin\otherTheta}\int_{0}^\pi\left[\arctan\frac{k\sin\psi}{1-k\cos\psi}\right]\left[\log\frac{1-\cos(\psi+\otherTheta)}{1-\cos(\psi-\otherTheta)}\right]\D\psi.\end{align*}The particular case for $ k=1,\otherTheta=\pi/2$ also brings us the integral representation of $G$, as claimed in Eq.~\ref{eq:G_S3}. Noting that $ \sin\psi=\sin(\pi-\psi)$, we may reduce the integral formula  into\[G=\frac{1}{4\pi }\int_{0}^\pi(\pi-\psi)\log\frac{1+\sin\psi}{1-\sin\psi}\D\psi=\frac{1}{4 }\int_{0}^{\pi/2}\log\frac{1+\sin\psi}{1-\sin\psi}\D\psi=-\frac{1}{2}\int_{0}^{\pi/2}\log\tan\frac{\pi-2\psi}{4}\D\psi=-\frac{1}{4}\int_{0}^{\pi}\log\tan\frac{\pi-\alpha}{4}\D\alpha,\]which is equivalent to the statement $\mathfrak  L(\pi)=0$ for Eq.~\ref{eq:beta'}.
\item
We first verify Eq.~\ref{eq:Li2int} for $ z=ik$ where $ 0<k<1$, namely,\begin{align}\I[\dilog (ik)]=\int_0^1\frac{\xi k\mathbf K( \sqrt{1-\xi^{2}} )}{\sqrt{(1+k^{2})^{2}-4k^{2}\xi^{2}}}\D\xi,\quad 0<k<1.
\tag{\ref{eq:Li2int}*}\label{eq:Li2int_star}
\end{align}To show this, we consider the specific scenario of  Eq.~\ref{eq:ImLi2} with $ 0<k<1$ and $\otherTheta=\pi/2$, which amounts to\begin{align*}\I[\dilog (ik)]=\frac{1}{2\pi }\int_{0}^\pi\left[\arctan\frac{k\sin\psi}{1-k\cos\psi}\right]\log\frac{1+\sin\psi}{1-\sin\psi}\D\psi=\frac{1}{2\pi }\int_{0}^{\pi/2}\left[\arctan\frac{2k\sin\psi}{1-k^{2}}\right]\log\frac{1+\sin\psi}{1-\sin\psi}\D\psi.\end{align*}We now integrate by parts in $ \psi$, with the knowledge of  Eq.~\ref{eq:beta'} in Lemma~\ref{lm:S2}: \begin{align*}\I[\dilog (ik)]={}&\frac{2}{\pi }\int_{0}^{\pi/2}\frac{k(1-k^2)\cos\psi}{1-2k^{2}\cos2\psi+k^4}\left[\cos\psi \int_0^1\mathbf K\left( \sqrt{1-(1-\kappa^2)\cos^2\psi} \right)\D \kappa \right]\D\psi.\end{align*}Then, we set $ \xi=\sqrt{1-\kappa^2}\cos\psi$, to convert the integral  above into\begin{align*}\I[\dilog (ik)]={}&\frac{2}{\pi }\int_{0}^{\pi/2}\frac{k(1-k^2)\cos\psi}{1-2k^{2}\cos2\psi+k^4}\left[ \int_0^{\cos\psi}\frac{\xi\mathbf K( \sqrt{1-\xi^{2}} )}{\sqrt{\cos^{2}\psi-\xi^{2}}}\D \xi\ \right]\D\psi\notag\\={}&\frac{2}{\pi }\int_{0}^{1}\frac{k(1-k^2)}{1-2k^{2}(1-2\eta^{2})+k^4}\left[ \int_0^{\sqrt{1-\eta^2}}\frac{\xi\mathbf K( \sqrt{1-\xi^{2}} )}{\sqrt{1-\eta^{2}-\xi^{2}}}\D \xi\ \right]\D\eta=\int_0^1\frac{\xi k\mathbf K( \sqrt{1-\xi^{2}} )}{\sqrt{(1+k^{2})^{2}-4k^{2}\xi^{2}}}\D\xi,\end{align*}where we have completed the integration in $ \eta\in(0,\sqrt{1-\xi^2})$  to arrive at the last expression.

We then note that both sides of  Eq.~\ref{eq:Li2int} are complex-analytic functions of $z$ in the simply-connected domain $|z|<1$,  because   the roots of the quadratic equation $ (1-Z)^{2}+4Zt=0$ are given by $ Z_{\pm}=1-2t\pm2 i\sqrt{t(1-t)}$, satisfying $ |Z_\pm|=1$.    By the principle of analytic continuation, we may extend the identity from the top-half vertical diameter ($z=ik$ where $ 0<k<1$) to the open unit disk ($ |z|<1$). With the Vitali convergence theorem, we may then check that both sides of   Eq.~\ref{eq:Li2int} are continuous up to the boundary $ |z|=1$. In particular,  the special cases $ z=1$ and $z=i$ for  Eq.~\ref{eq:Li2int}  recover a pair of identities we showed previously  in Lemma~\ref{lm:S2}: \begin{align*}\frac{\pi^{2}}{4}=\dilog (1)-\dilog (-1)={}&\int_0^1\frac{ 2\xi\mathbf K( \sqrt{1-\xi^{2}} )}{\sqrt{4\xi^{2}}}\D \xi=\int_{0}^1\mathbf K\left( \sqrt{1-\xi^{2}} \right)\D\xi=2I_{\ref{eq:a}},\notag\\2iG=\dilog (i)-\dilog (-i)={}&\int_0^1\frac{2i\xi \mathbf K( \sqrt{1-\xi^{2}} )}{\sqrt{4-4\xi^{2}}}\D\xi=i\int_0^1\mathbf K(k)\D k=2iI_{\ref{eq:beta}}.\end{align*} In the $ z\to0$ limit,  we have \begin{align*}2=\lim_{z\to0}\frac{\dilog (z)-\dilog (-z)}{z}=\int_0^1 \mathbf K( \sqrt{1-t} )\D t,\end{align*}which hearkens back to the conclusion from Remark~\ref{itm:var_subst} to Lemma~\ref{lm:S2}.

To confirm Eq.~\ref{eq:Li2int'}, simply differentiate both sides of   Eq.~\ref{eq:Li2int}. To prove Eq.~\ref{eq:Li2int_S}, integrate  Eq.~\ref{eq:Li2int} from $0$ to $z$;  then Eq.~\ref{eq:G_Klog} emerges as the specific case $ z=i$ for  Eq.~\ref{eq:Li2int_S}. To verify Eqs.~\ref{eq:G_KK} and \ref{eq:G_K}, we combine the previously established results  \[G=\frac{1}{2\pi}\int_0^{1}\frac{\mathbf{K}(\sqrt{1-\xi^2})}{\xi}\log\frac{1+\xi}{1-\xi}\D \xi=\frac{1}{4 }\int_{0}^{\pi/2}\log\frac{1+\sin\psi}{1-\sin\psi}\D\psi\]    with the cases $ z=\xi$ in and $ z=\sin\psi$ Eq.~\ref{eq:Li2int'}, before carrying out trigonometric substitutions.
\qed\end{enumerate}\end{proof}\begin{remark}\begin{enumerate}[label=(\arabic*)]\item
 Our geometric derivation of the integral formulae in this proposition is partly inspired by Coxeter's investigation of some functions relevant to non-Euclidean geometries  \cite{Coxeter1935}, in which context Catalan's constant  $G$ also appeared as the volume of an octahedron in hyperbolic space.

An integral formula like Eq.~\ref{eq:Li2int} is a simple output of such a geometric approach. It is still possible to verify  Eq.~\ref{eq:Li2int} combinatorially, using the  moment expansion  of elliptic integrals. This will be discussed in \S\ref{subsec:comb}.\item In \cite[Ref.][Proposition~3.3]{Zhou2013Pnu}, we used Ramanujan rotations to prove the following identity\begin{align*}[\mathbf K(\sqrt{x})]^2=\int_0^{\pi/2}\frac{ \mathbf K(\sin\theta)\sin\theta\D \theta}{\sqrt{1-x+\frac{x^2}{4}\cos^{2}\theta}},\quad 0\leq x<1.\end{align*}In Eq.~\ref{eq:Li2int}, if we set $ t=\cos^{2}\theta$ and $ z=\sqrt{1-x}>0$ for $ x\in[0,1)$, then we obtain the identity\begin{align*}\int_0^1\frac{ \mathbf K(\sin\theta)\sin\theta\D \theta}{\sqrt{1-x+\frac{x^{2}}{4\cos^{2}\theta}}}={}&\frac{\dilog (\sqrt{1-x})-\dilog (-\sqrt{1-x})}{\sqrt{1-x}},\end{align*}which bears close resemblance to the integral representation of $ [\mathbf K(\sqrt{\mathstrut x})]^2$. As a by-product, we have a strict inequality $\dilog (\sqrt{1-x})-\dilog (-\sqrt{1-x}) <\sqrt{1-x}[\mathbf K(\sqrt{\mathstrut x})]^2$ for $ 0<x<1$.
\eor\end{enumerate}

\end{remark}
\subsection{Dilogarithm Relations, Mehler-Dirichlet Decomposition, and Abel Transform\label{subsubsec:DR_MD_AT} }

The connection between the complete elliptic integral of the first kind $ \mathbf K$ and the dilogarithm function $ \dilog$ (Proposition~\ref{prop:S3}) provides a wealth of integral identities of more or less geometric flavor.
\begin{corollary}[Dilogarithm Relations]\label{cor:chi2}We have the integral identities:{\allowdisplaybreaks\begin{align}&&\int_0^1
\frac{ (1-z^{2} )\mathbf K( \sqrt{1-t} )}{2\sqrt{4z^{2}+t(1-z^{2})^{2}}}\D t+\int_0^1\frac{ z\mathbf K( \sqrt{1-t} )}{\sqrt{(1-z^{2})^{2}+4z^{2}t}}\D t={}&\frac{\pi^{2}}{4}+\log z\log\frac{1+z}{1-z},&& |z|<1,&&\label{eq:chi2_1}\\&&\int_0^1
\frac{ \mathbf K( \sqrt{1-t} )}{\sqrt{t-\sin^2\theta}}\D t+i\int_0^1\frac{ \mathbf K( \sqrt{1-t} )\sin\theta}{\sqrt{1-t\sin^2\theta}}\D t={}&\frac{\pi^2}{2}+2i\theta\log\left( i \tan\frac{\theta}{2}\right),&&0\leq\theta\leq\frac{\pi}{2},&&\label{eq:chi2_2}\\&&-\int_0^{\sin^2\theta}
\frac{ \mathbf K( \sqrt{1-t} )}{\sqrt{\sin^2\theta-t}}\D t+\int_0^1\frac{ \mathbf K( \sqrt{1-t} )\sin\theta}{\sqrt{1-t\sin^2\theta}}\D t={}&2\theta\log\tan\frac{\theta}{2},&& 0\leq\theta\leq\frac{\pi}{2},&&\label{eq:Im_ext}\\&&\int_0^{x^{2}}
\frac{ \mathbf K( \sqrt{t} )}{\sqrt{x^{2}-t}}\D t={}&\pi\arcsin x,&& 0\leq x\leq1,&&\label{eq:K_arcsin}\\&&\int_0^{x^{2}}
\frac{ 1}{\sqrt{ 1+t}\sqrt{x^{2}-t}}\mathbf K\left( \sqrt{\frac{t}{1+t}}\right )\D t={}&\pi\sinh^{-1} x=\pi \log(x+\sqrt{1+x^2}),&& 0\leq x<+\infty,&&\label{eq:K_arsinh}
\end{align}}the closed-form evaluations:{\allowdisplaybreaks\begin{align}
\int_0^1\frac{\mathbf K(\sqrt{1-t})}{2\sqrt{1+t}}\D t={}&\frac{\pi ^2}{8}-\frac{1}{2} \log ^2(\sqrt{2}-1),\label{eq:K_tan_pi8}\\\int_0^1\frac{\mathbf K(\sqrt{1-t})}{\sqrt{1+4t}}\D t={}&\frac{\pi ^2}{6}-\frac{3}{2} \log ^2\left(\frac{\sqrt{5}-1}{2}\right),\label{eq:K_gr}\\\int_0^1\frac{\mathbf K(\sqrt{1-t})}{2\sqrt{4+t}}\D t={}&\frac{\pi ^2}{12}-\frac{1}{6} \log ^2(\sqrt{5}-2),\label{eq:K_5_2}\\
\int_0^1[\mathbf K(\sqrt{1-t})]^2\D t={}&\int_0^1[\mathbf K(\sqrt{t})]^2\D t=\frac{7\zeta(3)}{2},\label{eq:Kt2_7Zeta3}\\\int_0^1\mathbf K(\sqrt{t})\mathbf K(\sqrt{1-t})\log t\D t={}&-\pi\left[\frac{\pi^{2}}{2}\log2-\frac{7\zeta(3)}{4} \right]=-\pi\int_{0}^{1}\left\{[\mathbf K(k)]^{2}-\frac{\pi^2}{4}\right\}\frac{ \D k}{k},\label{eq:KKlogt_zeta3}
\end{align}}together with more integral representations of Catalan's constant $G$:{\allowdisplaybreaks\begin{align}G={}&\frac{7\zeta(3)}{2\pi}-\frac{1}{\pi}\int_0^1\frac{\mathbf K (\eta)}{1+\eta}\log\frac{1}{\eta}\D \eta,\label{eq:G_K_mix}\\
G={}&\frac{\pi}{6}\log2+\frac{1}{3}\int_0^1\mathbf K(\sqrt{1-t})\left( \frac{2}{\sqrt{25-16 t}}+\frac{3}{4 \sqrt{25-9 t}}+\frac{3}{\sqrt{625-576 t}} \right)\D t,\label{eq:G_Ti_a}\\G= {}&\frac{\pi}{8}\log(2+\sqrt{3})+\frac{3}{8}\int_0^1\frac{\mathbf K(\sqrt{1-t})}{\sqrt{4-t}}\D t,\label{eq:G_Ti_b}\\G={}&\frac{\pi}{8}\log\frac{10+\sqrt{50-22\sqrt{5}}}{10-\sqrt{50-22\sqrt{5}}}+\frac{5}{8}\int_0^1\mathbf K(\sqrt{1-t})\left( \frac{1}{ \sqrt{6-2 \sqrt{5}-t}}-\frac{1}{\sqrt{6+2 \sqrt{5}-t}} \right)\D t.\label{eq:G_Ti_c}
\end{align}} \end{corollary}\begin{proof}We note that the right-hand side of Eq.~\ref{eq:Li2int} is $ \dilog(z)-\dilog(-z)=2\chi_2(z)$, where Legendre's $ \chi_2$-function satisfies (see Eq.~1.67 in \cite{Lewin1981})\[\chi_2\left( \frac{1-z}{1+z} \right)+\chi_2(z)=\frac{\pi^{2}}{8}+\frac{\log z}{2}\log\frac{1+z}{1-z}.\]This translates into Eq.~\ref{eq:chi2_1}. Setting $ z=i\tan\frac{\theta}{2}$ in  Eq.~\ref{eq:chi2_1}, we obtain  Eq.~\ref{eq:chi2_2}. In particular, when $ \theta=\pi/2$, both sides of  Eq.~\ref{eq:chi2_2} vanish in an obvious manner.
The imaginary part of Eq.~\ref{eq:chi2_2} gives rise to Eq.~\ref{eq:Im_ext}. Meanwhile, as we read off the real part of   Eq.~\ref{eq:chi2_2}, we arrive at Eq.~\ref{eq:K_arcsin} after the variable substitutions $ t\mapsto1-t$, $\sin^2\theta\mapsto1-x$. (See also 4.21.1.2 in \cite{Brychkov2008} for an equivalent form of  Eq.~\ref{eq:K_arcsin}.) By analytically continuing Eq.~\ref{eq:K_arcsin} to purely imaginary $x$, and noting that \begin{align*}\mathbf K(\sqrt{-t})=\int_0^{\pi/2}\frac{\D\theta}{\sqrt{1+t\cos^2\theta}}=\int_0^{\pi/2}\frac{\D\theta}{\sqrt{1+t-t\sin^2\theta}}=\frac{1}{\sqrt{1+t}}\mathbf K\left(\sqrt{\frac{t}{1+t}}\right),\end{align*} we may confirm Eq.~\ref{eq:K_arsinh}.

The precise values of $ \chi_2(\sqrt{2}-1)$, $ \chi_2((\sqrt{5}-1)/2)$ and $\chi_2(\sqrt{5}-2) $ (see Eqs.~1.68, 1.69 and 1.70 in \cite{Lewin1981}) had been known to Landen, which correspond to the evaluations given in Eqs.~\ref{eq:K_tan_pi8}, \ref{eq:K_gr} and \ref{eq:K_5_2}, respectively. To prove Eq.~\ref{eq:Kt2_7Zeta3}, we integrate  Eq.~\ref{eq:K_arcsin} with the aid of  the following formula (item 804.01 in \cite{ByrdFriedman})\[\int_{\arcsin(b/a)}^{\pi/2}\log\frac{1+\sin\theta}{1-\sin\theta}\frac{\D\theta }{\sqrt{a^{2}\sin^2\theta-b^{2}}}=\frac{\pi}{a}\mathbf K(b/a),\quad a>b>0.\]This leads us to \begin{align*}\pi\int_0^1[\mathbf K(\sqrt{t})]^2\D t=\int_{0}^{\pi/2}\log\frac{1+\sin\theta}{1-\sin\theta}\left[\int_0^{\sin^2\theta}
\frac{ \mathbf K( \sqrt{t} )}{\sqrt{\sin^2\theta-t}}\D t\right]\D\theta=\pi\int_{0}^{\pi/2}\theta\log\frac{1+\sin\theta}{1-\sin\theta}\D\theta=\frac{7\pi\zeta(3)}{2},\end{align*}where the last integral can be evaluated with Fourier series expansion.   (An equivalent form of  Eq.~\ref{eq:Kt2_7Zeta3}
 has appeared in Remark~1 to Theorem~2 in \cite{Wan2012}.)
Before proving Eq.~\ref{eq:KKlogt_zeta3}, we reformulate item 800.01 in  \cite{ByrdFriedman} as\[\int_0^{\arcsin \sqrt{t}}\frac{\log\sin\theta}{\sqrt{t-\sin^2\theta}}\D\theta=\frac{\mathbf K(\sqrt{t})\log t-\pi\mathbf K(\sqrt{1-t})}{4},\]and proceed with the calculation\begin{align*}\int_0^1\frac{\mathbf K(\sqrt{t})\log t-\pi\mathbf K(\sqrt{1-t})}{4}\mathbf K(\sqrt{1-t})\D t={}&\int_{0}^{\pi/2}\left[\R\int_0^1
\frac{ \mathbf K( \sqrt{1-t} )}{\sqrt{t-\sin^2\theta}}\D t\right]\log\sin\theta\D\theta=\pi\int_{0}^{\pi/2}\left( \frac{\pi}{2}-\theta \right)\log\sin\theta\D \theta
\notag\\={}&-\frac{\pi[2\pi^2\log2+7\zeta(3)]}{16}.\end{align*}Combined with   Eqs.~\ref{eq:K_sqr_minus_pi4} and \ref{eq:Kt2_7Zeta3}, the formula above is equivalent to  Eq.~\ref{eq:KKlogt_zeta3}.

The equivalence among the three quantities in Eq.~\ref{eq:KKlogt_zeta3} is not accidental, as the two extreme ends of Eq.~\ref{eq:KKlogt_zeta3} are actually connected by  Ramanujan's rotation formulae \cite[Ref.][Eqs.~41 and 44]{Zhou2013Pnu}:\begin{align}
[\mathbf K(\sqrt{t})]^2={}&\int_{S^2}\frac{2(2-t)\mathbf K(\sqrt{X^{2}+Y^2})}{(2-t)^{2}-t^{2}X^{2}}\frac{\D\sigma}{4\pi},\quad 0\leq t<1;\label{eq:K_sqr}\\\int_{S^2}\mathbf K\left(\sqrt{X^{2}+Y^2}\right)f(|X|){\frac{\D\sigma}{4\pi}}={}&\frac{2}{\pi}\int_0^1f(k)\mathbf K\left( \sqrt{\frac{1+k}{2}} \right)\mathbf K\left( \sqrt{\frac{1-k}{2}} \right)\D k,\label{eq:S2_int_KK}\end{align}where $ (X,Y,Z)$ are Cartesian coordinates on the unit sphere $ S^2:X^2+Y^2+Z^2=1$.
Concretely speaking, we may use  Eqs.~\ref{eq:K_sqr} and \ref{eq:S2_int_KK} to convert Eq.~\ref{eq:K_sqr_minus_pi4} into
\begin{align*}\frac{\pi^{2}}{2}\log2-\frac{7\zeta(3)}{4}={}&\frac{1}{2}\int_0^1\left\{ [\mathbf K(\sqrt{t})]^2-[\mathbf K(0)]^2 \right\}\frac{\D t}{t}=\frac{1}{2}\int_0^1\left\{\int_{S^2} \left[\frac{2(2-t)}{(2-t)^{2}-t^{2}X^{2}}  -1\right]\mathbf K\left(\sqrt{X^{2}+Y^2}\right)\frac{\D\sigma}{4\pi}\right\}\frac{\D t}{t}\notag\\={}&-\frac{1}{4}\int_{S^2}\mathbf K\left(\sqrt{X^{2}+Y^2}\right)\log\frac{1-X^2}{4}\frac{\D\sigma}{4\pi}=-\frac{1}{2\pi}\int_0^1\mathbf K\left( \sqrt{\frac{1+k}{2}} \right)\mathbf K\left( \sqrt{\frac{1-k}{2}} \right)\log\frac{1-k^2}{4}\D k.\end{align*}This clearly supplies a geometric link between
both ends of Eq.~\ref{eq:KKlogt_zeta3}.

To prove Eq.~\ref{eq:G_K_mix}, we integrate both sides of Eq.~\ref{eq:Im_ext} over the range $ 0\leq\theta\leq\pi/2$. The left-hand side becomes\begin{align*}&\int_0^{\pi/2}\left[-\int_0^{\sin^2\theta}
\frac{ \mathbf K( \sqrt{1-t} )}{\sqrt{\sin^2\theta-t}}\D t+\int_0^1\frac{ \mathbf K( \sqrt{1-t} )\sin\theta}{\sqrt{1-t\sin^2\theta}}\D t\right]\D\theta\notag\\={}&-\int_{0}^1\left[\int_{\arcsin\sqrt{t}}^{\pi/2}
\frac{\D\theta}{\sqrt{\sin^2\theta-t}}\right] \mathbf K( \sqrt{1-t} )\D t+\frac{1}{2}\int_0^1\frac{ \mathbf K( \sqrt{1-t} )}{\sqrt{t}}\log\frac{1+\sqrt{t}}{1-\sqrt{t}}\D t\notag\\={}&-\int_{0}^1\left[\int^{\arcsin\sqrt{1-t}}_{0}
\frac{\D\theta}{\sqrt{1-t-\sin^2\theta}}\right] \mathbf K( \sqrt{1-t} )\D t+\int_0^1\mathbf K \left(\sqrt{1-\xi^{2}} \right)\log\frac{1+\xi}{1-\xi}\D \xi\notag\\={}&-\int_{0}^1\left[\int^{1}_{0}
\frac{\D u}{\sqrt{1-(1-t)u^2}\sqrt{1-u^{2}}}\right] \mathbf K( \sqrt{1-t} )\D t+2\int_0^1\frac{\mathbf K (\eta)}{1+\eta}\log\frac{1}{\eta}\D \eta=-\int_{0}^1[ \mathbf K( \sqrt{1-t} )]^{2}\D t+2\int_0^1\frac{\mathbf K (\eta)}{1+\eta}\log\frac{1}{\eta}\D \eta \end{align*}after trigonometric substitutions and Landen's transformation. Meanwhile, the right-hand side evaluates to (see item 1.442.2 in \cite{GradshteynRyzhik} for the Fourier expansion)\begin{align}\int_0^{{\pi }/{2}} 2 \theta  \log \tan \frac{\theta }{2} \, \D\theta=-4\int_0^{{\pi }/{2}}  \theta  \left[ \sum_{\ell=0}^\infty\frac{\cos(2\ell+1)\theta}{2\ell+1} \right] \, \D\theta=2 \sum_{\ell=0}^\infty\frac{2-(-1)^{\ell}(2\ell+1)\pi}{(2\ell+1)^{3}}=-2\pi G+\frac{7\zeta(3)}{2}.\label{eq:logtan_G_zeta3}\end{align}Quoting the result from Eq.~\ref{eq:Kt2_7Zeta3}, we thus have\[-2\pi G+7\zeta(3)=2\int_0^1\frac{\mathbf K (\eta)}{1+\eta}\log\frac{1}{\eta}\D \eta,\] which implies Eq.~\ref{eq:G_K_mix}.

The inverse tangent integral $ \diTi(z):=-i\chi_2(iz)$ has been  treated in detail by Chapter 2 of \cite{Lewin1981}, where   three functional relations $ 3\diTi(1)-2\diTi(\tfrac12)-\diTi(\tfrac13)-\tfrac12\diTi(\tfrac34)=\tfrac\pi2\log2$ \cite[Ref.][Eq.~2.28]{Lewin1981}, $ 3\diTi(2-\sqrt{3})=2\diTi(1)+\tfrac\pi4\log(2-\sqrt{3})$ \cite[Ref.][Eq.~2.29]{Lewin1981} and $\diTi(\tan\tfrac{\pi}{20})-\diTi(\tan\tfrac{3\pi}{20})+\tfrac25\diTi(1)=\tfrac{\pi}{20}\log\tan\tfrac{\pi}{20}-\tfrac{3\pi}{20}\log\tan\tfrac{3\pi}{20} $ \cite[Ref.][Eq.~2.51]{Lewin1981} paraphrase into our Eqs.~\ref{eq:G_Ti_a}, \ref{eq:G_Ti_b} and \ref{eq:G_Ti_c}. \qed\end{proof}

We may produce more integral formulae out of Eqs.~\ref{eq:K_arcsin} and \ref{eq:K_arsinh}, as explained separately in the following two corollaries.
\begin{corollary}[Mehler-Dirichlet Decomposition]\label{cor:MD}Let $ n\in\{1,3,5,\dots\}$ be an odd natural number, then we have the following integral formulae:\begin{align}\int_0^{\pi}\mathbf K\left(\sin\frac{\beta}{2}\right) \sin\beta\cos\frac{n\beta}{2}\D \beta=\begin{cases}\dfrac{\pi^{2}}{4}, & n=1 \\[8pt]
0, & n\equiv1\bmod 4 \quad \text{and }\quad n>1\\[2pt]
-\dfrac{\pi^{2}n}{2^{(n+1)/2}(n-1)^2}\left[ \dfrac{(\frac{n-1}{2})!!}{(\frac{n+1}{4})!} \right]^2, & n\equiv3\bmod 4 \\
\end{cases}\label{eq:MD_proj}\end{align}which implies the identity\begin{align}\frac{1}{1+r}\int_{0}^{1}\frac{\mathbf K(k)rk\sqrt{1-k^2}}{1-2r(1-2k^2)+r^2}\D k=\frac{\pi\mathbf K(|r|)}{8}-\frac{\pi^2}{16(1+r)},\quad -1<r<1\label{eq:K_reprod}\end{align}as well as two more  integral representations for Catalan's constant $G $:\begin{align}
G={}&\frac{\pi}{2}\log2+\frac{1}{\pi}\int_0^1\mathbf K\left(\sqrt{1-\kappa^2}\right)\log\sqrt{1-\kappa^2}\D\kappa,\label{eq:G_log2_log}\\G={}&\frac{7\zeta(3)}{2\pi}-\frac{1}{2\pi^2}\int_0^{\pi/2}\mathbf K(\sin\theta)\left[ \frac{\dilog(e^{2i\theta})-\dilog(e^{-2i\theta})}{i\cos\theta} \right]\D\theta.\label{eq:7Apery_Li2}
\end{align}\end{corollary}\begin{proof}We first set $t=\sin^2\frac{\beta}{2},x=\sin\frac{\theta}{2}$ in  Eq.~\ref{eq:K_arcsin}, and recall the Mehler-Dirichlet formula (Eq.~\ref{eq:Mehler_Dirichlet}):\begin{align*}\int_0^{\pi}\left[\sum_{\ell=0}^\infty P_\ell(\cos\theta)\cos\frac{(2\ell+1)\beta}{2}\right]\mathbf K\left(\sin\frac{\beta}{2}\right) \sin\beta\D \beta=\int_0^{\theta}
\frac{ \mathbf K(\sin\frac{\beta}{2} )\sin\beta}{\sqrt{2(\cos\beta-\cos\theta)}}\D \beta={}&\frac{\pi\theta}{2},\quad 0\leq\theta \leq\pi.\end{align*}Meanwhile, the rightmost term in the equation above can also be developed into a series of Legendre polynomials (see 8.925.1 of \cite{GradshteynRyzhik}) \[\frac{\pi\theta}{2}=\frac{\pi^2}{4}\left\{ 1-\sum_{m=1}^\infty \frac{4m-1}{(2m-1)^2}\left[ \frac{(2m-1)!!}{2^{m}m!} \right]^2P_{2m-1}(\cos\theta)\right\}.\]As we compare the coefficients of $ P_\ell(\cos\theta)$, we may conclude that for  positive integers $ m=1,2,\dots$, there are the following integral formulae\begin{align*}\int_0^{\pi}\mathbf K\left(\sin\frac{\beta}{2}\right) \sin\beta\cos\frac{(2\ell+1)\beta}{2}\D \beta=\begin{cases}\frac{\pi^2}{4}, & \ell=0 \\
0, & \ell=2m \\
-\frac{\pi^2}{4}\frac{4m-1}{(2m-1)^2}\left[ \frac{(2m-1)!!}{2^{m}m!} \right]^2, & \ell=2m-1 \\
\end{cases}\end{align*}which are equivalent to the statement in Eq.~\ref{eq:MD_proj}.

Considering the familiar Fourier expansions \begin{align*}\sum_{m=1}^\infty r^{m}\cos\frac{(2m-1)\beta}{2}=\frac{r(1-r)\cos\frac{\beta}{2}}{1-2r\cos\beta+r^2},\quad -1<r<1.\end{align*}we may compute\begin{align*}\int_0^{\pi}\mathbf K\left(\sin\frac{\beta}{2}\right) \sin\beta\frac{r(1-r)\cos\frac{\beta}{2}}{1-2r\cos\beta+r^2}\D \beta={}&4\int_{0}^{1}\mathbf K(k)k\sqrt{1-k^2}\frac{r(1-r)}{1-2r(1-2k^2)+r^2}\D k\notag\\={}&\frac{\pi^2}{4}\left[ r-\sum_{m=1}^\infty \frac{4m-1}{(2m-1)^2}\left[ \frac{(2m-1)!!}{2^{m}m!} \right]^2 r^{2m} \right]=\frac{\pi(1-r^2)\mathbf K(|r|)}{2}-\frac{\pi^2(1-r)}{4},\end{align*}which entails Eq.~\ref{eq:K_reprod}.

As we integrate Eq.~\ref{eq:K_reprod} over the range $0\leq r\leq1 $ and exploit the fact that $ 2G=\int_0^1\mathbf K(r)\D r$, we may verify the equality\begin{align*}\frac{1}{4}\int_0^1\mathbf K(k)\left( \arccos k+\frac{k\log k}{\sqrt{1-k^2}} \right)\D k=\frac{\pi G}{4}-\frac{\pi^2}{16}\log2.\end{align*} Next, we calculate (see also 4.21.4.8 in \cite{Brychkov2008})\begin{align*}\int_0^1\mathbf K(k)\arccos k\D k=\int_0^{1}\left[ \int_0^{\pi/2}\frac{\D\theta}{\sqrt{1-k^{2}\sin^2\theta}} \right] \arccos k\D k=\int^{\pi/2}_0\frac{\dilog(\sin\theta)-\dilog(-\sin\theta)}{2\sin\theta} \D\theta=\frac{\pi^{2}}{4}\log2,\end{align*} so that \[\frac{1}{4}\int_0^1\mathbf K(k)\frac{k\log k}{\sqrt{1-k^2}}\D k=\frac{\pi G}{4}-\frac{\pi^2}{8}\log2\]rearranges into Eq.~\ref{eq:G_log2_log} after the variable substitution $ k\mapsto\sqrt{1-\kappa^2}$. Recalling Eq.~\ref{eq:G_K_mix}, we multiply both sides of  Eq.~\ref{eq:K_reprod} by $ \frac{1}{1+r}\log\frac{1}{r}$, then integrate over $ r\in[0,1]$, to deduce\begin{align*}\int_0^{\pi/2}\mathbf K(\sin\theta)\left[ \frac{\dilog(e^{2i\theta})-\dilog(e^{-2i\theta})}{16i\cos\theta} -\frac{\log2}{4}\sin\theta\right]\D\theta=\frac{\pi^2}{8}\left[ \frac{7\zeta(3)}{2\pi} -G\right]-\frac{\pi^2}{16}\log2.\end{align*}This leads to Eq.~\ref{eq:7Apery_Li2} after we recall from Eq.~\ref{eq:a} that $ \int_0^{\pi/2}\mathbf K(\sin\theta)\sin\theta\D \theta=\pi^2/4$.   \qed\end{proof}   \begin{remark}It is not hard to recognize that    we may deduce Eq.~\ref{eq:MD_proj} from Tricomi's Fourier expansion\begin{align}\label{eq:Tricomi_Fourier}\mathbf K(|\sin\theta|)=\sum_{n=0}^\infty\left[\frac{\Gamma(n+\frac12)}{n!}\right]^2\sin(4n+1)\theta,\quad 0<\theta<\pi\end{align} and \textit{vice versa}.\eor\end{remark}
\begin{corollary}[Abel Transform]\label{cor:Abel}\begin{enumerate}[label=\emph{(\alph*)}, ref=(\alph*), widest=a]\item Define the Abel transform as\[(\widehat Af)(x):=\int_x^{\infty}
\frac{ 2f(r)r}{\sqrt{r^{2}-x^{2}}}\D r,\] then we have the following integral identity\begin{align}\int_0^\infty \frac{x(\widehat Af)(x)}{\sqrt{1+x^{2}}}\mathbf K\left(\frac{x}{\sqrt{1+x^2}}\right )\D x=\pi\int_{0}^\infty f(r)r\log(r+\sqrt{1+r^2})\D r\label{eq:Abel_Tr}\end{align}so long as the  function $ f(r)$ guarantees convergence of the integrals on both sides.
Consequently, we have the integral formulae:\begin{align}\int_0^1\frac{\mathbf K(k)}{1-k^2}\left[ 1-\frac{k}{\sqrt{k^{2}+a^2(1-k^2)}} \right]\D k={}&\begin{cases}\dfrac{\arcsin a}{2}\left( \pi-\arcsin a \right), & 0\leq a\leq1 \\[8pt]\dfrac{\pi ^2}{8}+
\dfrac{\log ^2(\sqrt{a^2-1}+a)}{2} , & a>1 \\
\end{cases},\label{eq:Abel1}\\\int_0^{\frac{a}{\sqrt{1+a^2}}}\frac{k\mathbf K(k)}{(1-k^2)^{2}}\sqrt{a^{2}-(1+a^2)k^2}\D k={}&\frac{\pi}{8}\left[ (1+2a^2)\log(a+\sqrt{1+a^2}) -a\sqrt{1+a^{2}}\right],\quad a\geq0,\label{eq:Abel2}\\\int_0^{\frac{a}{\sqrt{1+a^2}}}\frac{k\mathbf K(k)}{(1-k^2)^{3}}[{a^{2}-(1+a^2)k^2}]^{3/2}\D k={}&\frac{3\pi}{128}   \left[(8 a^4+8 a^2+3) \log (a+\sqrt{1+a^2})-3 a(1+2a^2)\sqrt{1+a^{2}}\right],\quad a\geq0,\label{eq:Abel3}\end{align}and\begin{align}&\int_0^{\frac{a}{\sqrt{1+a^2}}}\frac{k\mathbf K(k)}{(1-k^2)^{2}}\left[a  \sqrt{{a^{2}-(1+a^2)k^2}}-\frac{k^2}{\sqrt{1-k^2}} \log \frac{\sqrt{{a^{2}-(1+a^2)k^2}}+a \sqrt{1-k^2}}{k}\right]\D k\notag\\={}&\pi \left[\frac{2}{9}-\frac{1}{36} \sqrt{1+a^2} (8+5 a^2)+\frac{a (3+2 a^2)}{12}  \log (a+\sqrt{1+a^2})\right],\quad a\geq0.\label{eq:Abel4}\end{align}Furthermore, we have integral representations of  $G$ and  $ \pi^2/8$:\begin{align}
G={}&\frac{1}{\pi}\int_0^1\frac{k\tanh^{-1}\sqrt{1-k^2}}{1-k^{2}}\mathbf K(k)\D k,\label{eq:G_tanh}
\\\frac{\pi^2}{8}={}&\frac{1}{\pi}\int_0^1\frac{\arccos k}{1-k^{2}}\mathbf K(k)\D k,\label{eq:pi8_arccos}
\end{align}along with two formulae valid for $ 0\leq\theta<\pi/2$:\begin{align}&\frac1\pi\int_{0}^1\frac{k\tanh^{-1}\sqrt{\frac{1-k^2}{1-k^2\sin^2\theta}}}{(1-k^{2})\sqrt{k^{2}+(1-k^2)\sec^2\theta}}\mathbf K(k)\D k=\frac{\theta}{2}\left(\frac{i\pi}{2}-\log\tan\frac{\pi-2\theta}{4}\right)-\frac{i[\dilog(ie^{-i\theta})-\dilog(-ie^{-i\theta})]}{2},\label{eq:Abel_Li2}\\&
\frac{1}{\pi}\int_0^1\left[ \frac{\tanh ^{-1}\sqrt{1-k^2}}{k}-\frac{ \log \cos \theta}{k\sqrt{1-k^{2}}}-\frac{1}{k\sqrt{1-k^2 \sin ^2\theta }} \tanh ^{-1}\sqrt{\frac{1-k^2}{1-k^2 \sin ^2\theta }} \right]\mathbf K(k)\D k=\frac{\theta^2}{4}.\label{eq:Abel_Li2_int}
\end{align}\item We have a pair of   integral formulae:\begin{align}
2[\mathbf E(\sqrt{u})-(1-u)\mathbf K(\sqrt{u})]={}&\int_0^u\mathbf K
(\sqrt{t})\D t=\int_0^{u}\frac{\arcsin\sqrt t}{\sqrt{u-t}}\D t,\quad 0<u<1,\label{eq:arcsin_Abel_eq}\\2\sqrt{1+\xi}\left[ \mathbf K\left( \sqrt{\frac{\xi}{1+\xi}} \right)-\mathbf E\left( \sqrt{\frac{\xi}{1+\xi}} \right) \right]={}&\int_0^\xi\frac{1}{\sqrt{1+t}}\mathbf K
\left(\sqrt{\frac{t}{1+t}}\right)\D t=\int_{0}^\xi\frac{\log(\sqrt{t}+\sqrt{1+t})}{\sqrt{\xi-t}}\D t,\quad \xi>0,\label{eq:arsinh_Abel_eq}\end{align}which lead to the relations\begin{align}
\int_0^u\frac{\mathbf E(\sqrt{t})-(1-t)\mathbf K(\sqrt{t})}{\sqrt{u-t}}\D t={}&\frac{\pi}{4}[\sqrt{\smash[b]{u(1-u)}}-(1-2u) \arcsin\sqrt{u}],\quad 0<u<1,
\label{eq:E_arcsin}\\\int_0^\xi\frac{\sqrt{1+t}}{\sqrt{\smash[b]{\xi-t}}}\left[ \mathbf K\left( \sqrt{\frac{t}{1+t}} \right)-\mathbf E\left( \sqrt{\frac{t}{1+t}} \right) \right]\D t={}&\frac{\pi}{4}[(1+2\xi)\log(\sqrt{\smash[b]{\xi}}+\sqrt{\smash[b]{1+\xi}})-\sqrt{\smash[b]{\xi(1+\xi)}}],\quad \xi>0.\label{eq:E_arsinh}\end{align}\end{enumerate}\end{corollary}\begin{proof}\begin{enumerate}[label=(\alph*),widest=a]\item We first rewrite  Eq.~\ref{eq:K_arsinh} as\begin{align*}\pi \log(r+\sqrt{1+r^2})=\int_0^{r}
\frac{ 2x}{\sqrt{ 1+x^{2}}\sqrt{r^{2}-x^{2}}}\mathbf K\left(\frac{x}{\sqrt{1+x^2}}\right )\D x,\quad 0\leq r<+\infty,\end{align*}then multiply both sides by $ f(r)r$ and integrate over the positive real axis, as follows:\begin{align*}\pi\int_{0}^\infty f(r)r\log(r+\sqrt{1+r^2})\D r={}&\int_0^\infty f(r)r\left[ \int_0^{r}
\frac{ 2x}{\sqrt{ 1+x^{2}}\sqrt{r^{2}-x^{2}}}\mathbf K\left(\frac{x}{\sqrt{1+x^2}}\right )\D x\right]\D r\notag\\={}&\int_0^\infty \frac{x}{\sqrt{1+x^{2}}}\mathbf K\left(\frac{x}{\sqrt{1+x^2}}\right )\left[ \int_x^{\infty}
\frac{ 2f(r)r}{\sqrt{r^{2}-x^{2}}}\D r\right]\D x=\int_0^\infty \frac{x(\widehat Af)(x)}{\sqrt{1+x^{2}}}\mathbf K\left(\frac{x}{\sqrt{1+x^2}}\right )\D x,\end{align*}hence the formula in Eq.~\ref{eq:Abel_Tr}.

Taking advantage of the table of Abel transforms (Table 13.9 in \cite{Bracewell}) we may put down the following equalities involving a non-negative parameter $ a\geq0$:{\allowdisplaybreaks\begin{align*}\pi\int_0^\infty \frac{x}{\sqrt{1+x^{2}}}\left( \frac{1}{x}-\frac{1}{\sqrt{a^2+x^2}} \right)\mathbf K\left(\frac{x}{\sqrt{1+x^2}}\right )\D x={}&\pi\int_{0}^\infty \left(\frac{1}{r^{2}}-\frac{1}{a^{2}+r^2}\right)r\log(r+\sqrt{1+r^2})\D r,\notag\\\int_0^a \frac{2x\sqrt{a^{2}-x^2}}{\sqrt{1+x^{2}}}\mathbf K\left(\frac{x}{\sqrt{1+x^2}}\right )\D x={}&\pi\int_{0}^a r\log(r+\sqrt{1+r^2})\D r,\notag\\\int_0^a \frac{4x({a^{2}-x^2})^{3/2}}{3\sqrt{1+x^{2}}}\mathbf K\left(\frac{x}{\sqrt{1+x^2}}\right )\D x={}&\pi\int_{0}^a (a^{2}-r^{2})r\log(r+\sqrt{1+r^2})\D r,\notag\\\int_0^a \frac{x}{\sqrt{1+x^{2}}}\left(a\sqrt{a^{2}-x^2}-x^2\cosh^{-1}\frac{a}{x}\right)\mathbf K\left(\frac{x}{\sqrt{1+x^2}}\right )\D x={}&\pi\int_{0}^a (a-r)r\log(r+\sqrt{1+r^2})\D r.\end{align*} }As we perform the variable substitution $ x\mapsto k/\sqrt{1-k^2}$ on the left,  and complete the integration on the right, we obtain the integral formulae listed in Eqs.~\ref{eq:Abel1}-\ref{eq:Abel4}. One might note that the $ a=1$ specific case in Eq.~\ref{eq:Abel1} corresponds to Eq.~\ref{eq:a}.

We note that Catalan's constant can be cast into an integral in the following manner (cf.~3.521.2 in \cite{GradshteynRyzhik}):\begin{align}G=\sum_{n=0}^\infty\frac{(-1)^n}{(2n+1)^2}=\sum_{n=0}^\infty(-1)^n\int_0^\infty ye^{-(2n+1)y}\D y=\int_0^\infty\frac{ye^{-y}\D y}{1+e^{-2y}}=\frac12\int_0^{1}\frac{y\D y}{\cosh y}=\frac{1}2\int_0^\infty\frac{\log(r+\sqrt{1+r^2})\D r}{1+r^{2}},\label{eq:G_cosh}\end{align}where the last step involves a substitution $y=\log(r+\sqrt{1+r^2}) $. Meanwhile, the Abel transform reveals\begin{align*}\int_x^\infty\frac{1}{1+r^2}\frac{\D t}{\sqrt{r^2-x^2}}=\frac{1}{\sqrt{1+x^{2}}}\tanh^{-1}\frac{1}{\sqrt{1+x^2}},\end{align*}which brings us to \begin{align*}\int_0^\infty\frac{x\tanh^{-1}\frac{1}{\sqrt{1+x^{2}}}}{1+x^2}\mathbf K\left(\frac{x}{\sqrt{1+x^2}}\right )\D x=\int_0^\infty\frac{x\sinh^{-1}\frac{1}{x}}{1+x^2}\mathbf K\left(\frac{x}{\sqrt{1+x^2}}\right )\D x=\pi G.\end{align*}After replacing $x$ with $ k/\sqrt{1-k^2}$, we obtain Eq.~\ref{eq:G_tanh} as claimed. (It is easy to recognize that Eq.~\ref{eq:G_tanh} becomes Eq.~\ref{eq:b} after the substitution $ \kappa=\sqrt{1-k^{2}}$ and Landen's transformation.) The proof of Eq.~\ref{eq:pi8_arccos} is entirely similar, and is based on the facts that (cf.~3.521.1 in \cite{GradshteynRyzhik})\begin{align}\frac{\pi^2}{8}=\frac12\int_0^{1}\frac{y\D y}{\sinh y}=\frac{1}2\int_0^\infty\frac{\log(r+\sqrt{1+r^2})\D r}{r\sqrt{1+r^{2}}}\label{eq:pi8_y_sinhy}\end{align}and \begin{align}\int_x^\infty\frac{1}{r\sqrt{1+r^2}}\frac{\D t}{\sqrt{r^2-x^2}}=\frac{1}{x}\arctan\frac{1}{x},\quad\int_0^\infty\frac{\arctan\frac{1}{x}}{\sqrt{1+x^2}}\mathbf K\left(\frac{x}{\sqrt{1+x^2}}\right )\D x=\pi\frac{\pi^{2}}{8}.\label{eq:pi8_arctan_inf_int}\end{align}

The foregoing approach naturally extends to the scenario where\begin{align*}\frac12\int_0^\infty\frac{\log(r+\sqrt{1+r^2})\D r}{{\sec^{2}\theta}+r^{2}}=\frac{\theta\cos\theta}{2}\left(\frac{i\pi}{2}-\log\tan\frac{\pi-2\theta}{4}\right)-\frac{i\cos\theta[\dilog(ie^{-i\theta})-\dilog(-ie^{-i\theta})]}{2},\quad0\leq\theta<\frac{\pi}{2},\end{align*}and\[\int_x^\infty\frac{1}{{\sec^{2}\theta}+r^2}\frac{\D t}{\sqrt{r^2-x^2}}=\frac{\cos\theta\sinh^{-1}\frac{1}{x\cos\theta}}{\sqrt{{\sec^{2}\theta}+x^{2}}}=\frac{\cos\theta\tanh^{-1}\frac{1}{\sqrt{x^{2}\cos^{2}\theta+1}}}{\sqrt{{\sec^{2}\theta}+x^{2}}}.\]This  leads to the formula in Eq.~\ref{eq:Abel_Li2}. (In contrast, there does not seem to be an ``elementary'' generalization of   Eq.~\ref{eq:pi8_arccos}: if we have $ (a+r)\sqrt{1+r^2}$ or $ r\sqrt{a^2+r^2}$ in place of $ r\sqrt{1+r^2}$ in the denominator, the resulting formula will involve the Meijer $G$-function, instead of the dilogarithm.)

Multiplying both sides of  Eq.~\ref{eq:Abel_Li2} by $ \sin\theta$ before integrating in $ \theta$, we obtain\begin{align*}&\frac{1}{\pi}\int_0^1\left[ \frac{\tanh ^{-1}\sqrt{1-k^2}}{k(1-k^{2})}-\frac{ \log \cos \theta}{k\sqrt{1-k^{2}}}-\frac{\sqrt{1-k^2 \sin ^2\theta }}{k(1-k^{2})} \tanh ^{-1}\sqrt{\frac{1-k^2}{1-k^2 \sin ^2\theta }} \right]\mathbf K(k)\D k\notag\\={}&G+\frac{\theta^2}{4}-\cos\theta\left\{\frac{\theta}{2}\left(\frac{i\pi}{2}-\log\tan\frac{\pi-2\theta}{4}\right)-\frac{i[\dilog(ie^{-i\theta})-\dilog(-ie^{-i\theta})]}{2}\right\}.\end{align*}As we multiply  both sides of  Eq.~\ref{eq:Abel_Li2} by $ \cos\theta$ and add to the equation above, we have\begin{align*}\frac{1}{\pi}\int_0^1\left[ \frac{\tanh ^{-1}\sqrt{1-k^2}}{k(1-k^{2})}-\frac{ \log \cos \theta}{k\sqrt{1-k^{2}}}-\frac{1}{k\sqrt{1-k^2 \sin ^2\theta }} \tanh ^{-1}\sqrt{\frac{1-k^2}{1-k^2 \sin ^2\theta }} \right]\mathbf K(k)\D k={}&G+\frac{\theta^2}{4}.\end{align*}
We will arrive at   Eq.~\ref{eq:Abel_Li2_int} after eliminating $G$ from the last formula using Eq.~\ref{eq:G_tanh}. \item The leftmost equalities in Eqs.~\ref{eq:arcsin_Abel_eq} and \ref{eq:arsinh_Abel_eq} can be verified by differentiation in $ u$ and $\xi$, respectively. Meanwhile, we note that the solution to the Abel integral equation (which is morally equivalent to the Abel transform)\begin{align}\int_0^x\frac{\phi(y)}{\sqrt{\smash[b]{x-y}}}\D y=f(x)\label{eq:Abel_int_eq}\end{align} satisfies \cite[Ref.][p.~39]{TricomiInt}\begin{align}\pi\int_0^x\phi(y)\D y=\int_0^x\frac{f(y)}{\sqrt{\smash[b]{x-y}}}\D y.\label{eq:Abel_int_eq'}\tag{\ref{eq:Abel_int_eq}$'$}\end{align}Specializing this to Eqs.~\ref{eq:K_arcsin} and \ref{eq:K_arsinh}, we obtain the rightmost equalities in  Eqs.~\ref{eq:arcsin_Abel_eq} and \ref{eq:arsinh_Abel_eq}.

Clearly, the extreme ends of   Eqs.~\ref{eq:arcsin_Abel_eq} and \ref{eq:arsinh_Abel_eq} represent another pair of Abel integral equations. Their respective solutions bring us \begin{align*}\int_0^u\frac{\mathbf E(\sqrt{t})-(1-t)\mathbf K(\sqrt{t})}{\sqrt{u-t}}\D t={}&\frac{\pi}{2}\int_0^u\arcsin\sqrt{t}\D t=\frac{\pi}{4}[\sqrt{\smash[b]{u(1-u)}}-(1-2u) \arcsin\sqrt{u}],\quad 0<u<1,\notag\\\int_0^\xi\frac{\sqrt{1+t}}{\sqrt{\smash[b]{\xi-t}}}\left[ \mathbf K\left( \sqrt{\frac{t}{1+t}} \right)-\mathbf E\left( \sqrt{\frac{t}{1+t}} \right) \right]\D t={}&\frac{\pi}{2}\int_0^\xi\log(\sqrt{t}+\sqrt{1+t})\D t=\frac{\pi}{4}[(1+2\xi)\log(\sqrt{\smash[b]{\xi}}+\sqrt{\smash[b]{1+\xi}})-\sqrt{\smash[b]{\xi(1+\xi)}}],\quad \xi>0,\end{align*}  as stated in  Eqs.~\ref{eq:E_arcsin} and   \ref{eq:E_arsinh}.
   \qed\end{enumerate}\end{proof}\begin{remark}

The closed-form evaluations in Eqs.~\ref{eq:Abel1}-\ref{eq:Abel4} are not unexpected, as all of them can be regarded as appropriate integrations of some previously known identities.
As one differentiates in the parameter $a$, one can reduce Eq.~\ref{eq:Abel1}  to Eq.~\ref{eq:Li2int'}, Eq.~\ref{eq:Abel2} to Eq.~\ref{eq:K_arsinh}, while both Eq.~\ref{eq:Abel3} and \ref{eq:Abel4} become a certain multiple of Eq.~\ref{eq:Abel2}.

Differentiating   Eq.~\ref{eq:Abel_Li2} in $ \theta$ and multiplying by $ 1/\sin\theta$,  we obtain\begin{align}\frac1\pi\int_0^1\frac{\mathbf K(\sqrt{1-\kappa^2})}{\cos^2\theta+\kappa^2\sin^2\theta}\left[ 1-\frac{\kappa}{\sqrt{\cos^2\theta+\kappa^2\sin^2\theta}} \tanh^{-1}\frac{\kappa}{\sqrt{\cos^2\theta+\kappa^2\sin^2\theta}}\right]\D\kappa=
\frac{\theta}{\sin2\theta},\quad0<\theta<\frac{\pi}{2},\tag{\ref{eq:Abel_Li2}$'$}
\end{align}a result that will be derived again in \S\ref{subsubsec:B_simp}  using a different method. Taking derivative of  Eq.~\ref{eq:Abel_Li2_int} with respect to  $ \theta$ and dividing the result by  $ \sin2\theta$, we get the same relation as in Eq.~\ref{eq:Abel_Li2}$'$.\eor\end{remark}

\section{Beltrami Rotations Revisited\label{sec:Beltrami_revisited}}
\subsection{Beltrami Transformations and Some Simple Consequences\label{subsubsec:B_simp}}
In \cite[Ref.][\S3.1]{Zhou2013Pnu}, we introduced the Beltrami rotations, which amount to the following integral identities:\begin{align}\mathbf K(k)={}&\frac{2}{\pi}\int_0^1\frac{\mathbf K(\sqrt{1-\kappa^2})\D \kappa}{1-k^2\kappa^2},& 0<k<1,\label{eq:Beltrami}\\\mathbf K(\xi)={}&\frac{2}{\pi}\int_0^1\frac{\sqrt{1-\xi^{2}}\mathbf K(\sqrt{1-\kappa^2})\D \kappa}{1-\xi^2(1-\kappa^2)},& 0<\xi<1,\label{eq:iB}\tag{\ref{eq:Beltrami}$'$}\\\mathbf K(r)={}&\frac{2}{\pi}\int_0^1\frac{(1+r)\mathbf K(\sqrt{1-\kappa^2})\D \kappa}{(1+r)^{2}-4r\kappa^2},& 0<r<1,\label{eq:LB}\tag{\ref{eq:Beltrami}$_L$}\\\mathbf K(\eta)={}&\frac{2}{\pi}\int_0^1\frac{(1-\eta)\mathbf K(\sqrt{1-\kappa^2})\D \kappa}{(1-\eta)^{2}+4\eta\kappa^2},& 0<\eta<1.\label{eq:iLB}\tag{\ref{eq:Beltrami}$'_L$}\end{align}As pointed out in \cite{Zhou2013Pnu}, the four identities above are related to each other
by standard transformations for complete elliptic integrals of the first kind: Eq.~\ref{eq:iB} descends from Eq.~\ref{eq:Beltrami} and the imaginary modulus transformation; Eq.~\ref{eq:LB} results from  Eq.~\ref{eq:Beltrami} and  Landen's transformation; Eq.~\ref{eq:iLB} is attributed to   Eq.~\ref{eq:Beltrami} along with a combination of the imaginary modulus and Landen's transformations.

We have proved  Eq.~\ref{eq:Beltrami}  in \cite{Zhou2013Pnu} using rotations on a unit sphere, which is inspired  by Beltrami's work \cite{Beltrami1880}. In the next two short paragraphs, we offer alternative derivations of the Beltrami transformations using the Mehler-Dirichlet decomposition and Abel transforms.

For  $0<r<1$, we may use the variable transformation $ k=\sqrt{1-\kappa^2}$ to rewrite Eq.~\ref{eq:K_reprod} (a consequence of the Mehler-Dirichlet decomposition) as\begin{align*}\mathbf K(r)=\frac{8}{\pi}\frac{1}{1+r}\int_0^1\frac{\mathbf K(\sqrt{1-\kappa^2})r\kappa^{2}\D\kappa}{1-2r(2\kappa^{2}-1)+r^{2}}+\frac{8}{\pi}\frac{1}{4(1+r)}\int_{0}^1\mathbf K\left( \sqrt{1-\kappa^2} \right)\D\kappa=\frac{2}{\pi}\int_0^1\frac{(1+r)\mathbf K(\sqrt{1-\kappa^2})\D\kappa}{(1+r)^{2}-4r\kappa^{2}}\end{align*}  thereby demonstrating    Eq.~\ref{eq:LB}. (If we start with $ -1<r<0$ instead, then  Eq.~\ref{eq:K_reprod} leads to a verification of    Eq.~\ref{eq:iLB}.)

 As one  combines the following identities\begin{align*}\int_0^{\infty}\frac{r\log(r+\sqrt{1+r^2})\D r}{\left(\frac{1}{1-\xi^2}+r^{2}\right)^{3/2}}=\sqrt{1-\xi^2}\mathbf K(\xi)\quad \text{and}\quad \int_x^{\infty}\frac{r\D r}{\left(\frac{1}{1-\xi^2}+r^{2}\right)^{3/2}\sqrt{r^2-x^2}}=\frac{1}{\frac{1}{1-\xi^2}+x^{2}},\qquad 0\leq\xi<1\end{align*} into the formula\begin{align}\mathbf K(\xi)=\frac{2}{\pi\sqrt{1-\xi^{2}}}\int_0^\infty\frac{\mathbf K\left(\frac{x}{\sqrt{1+x^{2}}}\right)}{\frac{1}{1-\xi^2}+x^{2}}\frac{x\D x}{\sqrt{1+x^2}}=\frac{2}{\pi\sqrt{1-\xi^{2}}}\int_0^1\frac{\mathbf K(\sqrt{{1-\kappa^{2}}})}{1+\frac{\xi^{2}\kappa^{2}}{1-\xi^2}}\D \kappa,\label{eq:im_sphere}\end{align}one succeeds in proving Eq.~\ref{eq:iB} using the Abel transform (Eq.~\ref{eq:Abel_Tr}).

At this moment, we may immediately appreciate the equivalence between two previously derived integral representations of Catalan's constant \[G=\frac{1}{2}\int_{0}^1\mathbf K(k)\D k=\frac{1}{2\pi}\int_0^1\frac{\mathbf K(\sqrt{1-\kappa^2})}{\kappa}\log\frac{1+\kappa}{1-\kappa}\D\kappa,\]using the  rational integral transformation given by  Eq.~\ref{eq:Beltrami}. In other words, the Beltrami transformation (Eq.~\ref{eq:Beltrami}) is a formal generalization of the ``box-to-box transformation'' (Eq.~\ref{eq:qrs_uvw}) introduced  in Remark~\ref{itm:b2b} of Lemma~\ref{lm:S2}.
Then,
we may also deduce alternative forms of Eqs.~\ref{eq:a} and \ref{eq:b} as follows:\begin{align}
\frac{\pi^2}{8}={}&\int_0^1\frac{\mathbf K(k)}{1+k}\D k=\int_{0}^1\frac{\mathbf K(\sqrt{1-\kappa^2}) }{\pi  (\kappa ^2-1)}\left( \log \frac{1-\kappa ^2}{4}+\kappa  \log\frac{1+\kappa}{1-\kappa} \right)\D\kappa,\label{eq:a_B}\tag{\ref{eq:a}$_\mathrm{B}$}\\G={}&\int_0^1\frac{\mathbf K(k)}{1-k}\log\frac{1}{k}\D k=\int_{0}^1\frac{\mathbf K(\sqrt{1-\kappa^2})}{3 \pi ^2   (1-\kappa ^2)}\left[\pi^{2}-3 (1+\kappa ) \text{Li}_2(\kappa )-3 (1-\kappa) \text{Li}_2(-\kappa )\right]\D\kappa.\label{eq:b_B_R}\tag{\ref{eq:b}$_\mathrm{B}$}
\end{align} Hereafter, an equation numbered  $ x_{\mathrm B}$ is derived from Eq.~$x$ and a Beltrami transformation (Eq.~\ref{eq:Beltrami}). Descendants of  Eq.~$x$ after transformations in Eqs.~\ref{eq:iB}, \ref{eq:LB} and \ref{eq:iLB} will be labeled accordingly as $ x_{\mathrm {iB}}$, $ x_{\mathrm{LB}}$ and $ x_{\mathrm{iLB}}$. We will also refer to the four avatars of Beltrami transformations as types ``B'', ``iB'', ``LB'' and ``iLB'' for short.

Applying the identity \begin{align*}\mathbf K\left( \sqrt{k^2\cos^2\frac{\beta}{2}+\sin^2\frac{\beta}{2} }\right)=\frac{2}{\pi}\int_0^1\frac{\mathbf K(\sqrt{1-\kappa^2})}{1-[k^2\cos^2(\beta/2)+\sin^2(\beta/2)]\kappa^{2}}\end{align*}to Eq.~\ref{eq:beta'}, we  arrive at the  expression below:\begin{align}G+\frac{1}{4}\int_0^\beta\log\tan\frac{\pi-\alpha}{4}\D\alpha={}&\frac{1}{\pi}\int_{0}^1\frac{\mathbf K(\sqrt{1-\kappa^2})}{\kappa\sqrt{\smash[b]{1-\kappa^{2}\sin^2(\beta/2)}}}\tanh^{-1}\frac{\kappa\cos(\beta/2)}{\sqrt{\smash[b]{1-\kappa^{2}\sin^2(\beta/2)}}}\D\kappa,\quad 0\leq\beta<\pi.\label{eq:beta_B}
\tag{\ref{eq:beta'}$_\mathrm{B}$}
\end{align} After we differentiate both sides of the equation above in $ \beta$, the integral formula\begin{align}\frac14\log\tan\frac{\pi-\beta}{4}=-\frac{\sin(\beta/2)}{2\pi}\int_{0}^1\frac{\mathbf K(\sqrt{1-\kappa^2})}{1-\kappa^{2}\sin^2(\beta/2)}\left[1-\frac{\kappa\cos(\beta/2)}{\sqrt{\smash[b]{1-\kappa^{2}\sin^2(\beta/2)}}}\tanh^{-1}\frac{\kappa\cos(\beta/2)}{\sqrt{\smash[b]{1-\kappa^{2}\sin^2(\beta/2)}}}\right]\D\kappa,\quad 0\leq\beta<\pi\label{eq:beta'_B}
\tag{\ref{eq:beta'}$'_\mathrm{B}$}
\end{align} arises as a result.\footnote{Of course, one can rewrite the right-hand side of Eq.~\ref{eq:beta'_B}  by the substitution \[ \mathbf K\left(\sin\frac{\beta}{2}\right)=\frac{2}{\pi}\int_{0}^1\frac{\mathbf K(\sqrt{1-\kappa^2})\D\kappa}{1-\kappa^{2}\sin^2(\beta/2)},\quad 0\leq\beta<\pi.\]A similar procedure can be performed on a few formulae in the rest of \S\ref{subsubsec:B_simp} as well, except that some subtle changes in the domain of validity may occur. The details about  paraphrasing these integral identities will be omitted in  \S\ref{subsubsec:B_simp}, but will appear in the collation of formulae in Appendix~\ref{app:ind_formulae}.   } Here, in Eq.~\ref{eq:beta'_B}, the bracketed term is an even function of $ \kappa\cos\frac{\beta}{2}({1-\kappa^2\sin^2\frac{\beta}{2}})^{-1/2}$, so its value is not affected by the choice of the univalent branches of the square root. By analytic continuation, Eq.~\ref{eq:beta'_B} extends to the following identity valid in the open unit disk:\begin{align}
\frac{1}{4}\log\frac{1-z}{1+z}=\frac{z(1+z^{2})}{\pi}\int_{0}^1\frac{\mathbf K(\sqrt{1-\kappa^2})}{(1+z^{2})^{2}-4z^{2}\kappa^{2}}\left[-1+\frac{\kappa(1-z^{2})}{\sqrt{(1+z^{2})^{2}-4z^{2}\kappa^{2}}}\tanh^{-1}\frac{\kappa(1-z^{2})}{\sqrt{(1+z^{2})^{2}-4z^{2}\kappa^{2}}}\right]\D\kappa,\quad |z|<1.\label{eq:beta'_star_B}\tag{\ref{eq:beta'}$'^*_\mathrm{B} $}
\end{align}If we set $ z=i\tan({\theta}/{2}),0\leq\theta<\pi/2$ in the equation above, we recover Eq.~\ref{eq:Abel_Li2}$'$.

In Remark~\ref{itm:low_deg} to Proposition~\ref{prop:more_S2}, we mentioned Eq.~\ref{eq:low_deg} without a proof. Now, we can demonstrate Eq.~\ref{eq:low_deg} by connecting it to Eqs.~\ref{eq:a} and \ref{eq:G_log2_log} via the type-B transformation (Eq.~\ref{eq:Beltrami}):\begin{align}\int_0^1\left[ \mathbf K(k)-\frac{\pi}{2}\right]\frac{\D k}{k}={}&\int_0^1\left[\int_0^1 \frac{2\mathbf K(\sqrt{1-\kappa^2})}{\pi}\left( \frac{1}{1-k^{2}\kappa^2}-1 \right)\D\kappa\right]\frac{\D k}{k}=-\frac{2}{\pi}\int_0^1 \mathbf K\left(\sqrt{1-\kappa^2}\right)\log\sqrt{1-\kappa^2}\D\kappa=-2G+\pi\log2.\label{eq:low_deg'}
\tag{\ref{eq:low_deg}$_\mathrm{B}$}
\end{align}

Setting $ 0<iz<1$ in   Eq.~\ref{eq:Li2int}, writing $ k$ for $ \sqrt{1-t}$, and applying the rational integral transform in Eq.~\ref{eq:Beltrami}, we obtain an equivalent form:\begin{align}\int_0^1\frac{4\mathbf  K(\sqrt{1-\kappa ^2}) }{\pi  \kappa  \sqrt{(1+z ^2)^2 \kappa ^2-4 z ^2}}\left[ \log\frac{1}{\sqrt{1-\kappa^{2}}} +\log\frac{\sqrt{(1+z^{2})^2 \kappa ^2-4 z ^2}+(1-z ^2) \kappa}{\sqrt{(1+z^{2})^2 \kappa ^2-4 z ^2}+(1+z ^2) \kappa}\right]\D\kappa=\frac{\dilog (z)-\dilog (-z)}{ z},\tag{\ref{eq:Li2int}$_\mathrm{B}$}\label{eq:Li2int_B}\end{align}which extends to be valid for all the points satisfying $ |z|\leq1$, by analytic continuation and the Vitali convergence theorem. In the $ z\to0$ limit, we obtain the identity \begin{align}\int_0^1\frac{\mathbf K(\sqrt{1-\kappa^2})}{\kappa^2}\log\frac{1}{1-\kappa^2}\D\kappa=\pi.\label{eq:Li2_limit_pi}\end{align}Differentiating Eq.~\ref{eq:Li2int_B} in $ z$, we  arrive at an equivalent form of Eq.~\ref{eq:Li2int'} for $|z|<1,\I z\neq0$:\begin{align}&\int_0^1\frac{4\kappa (1-z ^4 ) \mathbf  K(\sqrt{1-\kappa ^2}) }{\pi[(1+z^{2})^2 \kappa ^2-4 z ^2]\sqrt{(1+z^{2})^2 \kappa ^2-4 z ^2}  }\left[ \log\frac{1}{\sqrt{1-\kappa^{2}}} +\log\frac{\sqrt{(1+z^{2})^2 \kappa ^2-4 z ^2}+(1-z ^2) \kappa}{\sqrt{(1+z^{2})^2 \kappa ^2-4 z ^2}+(1+z ^2) \kappa}\right]\D\kappa\notag\\&+\int_0^1\frac{8z^{2} \mathbf  K(\sqrt{1-\kappa ^2}) }{\pi[(1+z^{2})^2 \kappa ^2-4 z ^2]  }\D\kappa=\frac{1}{z}\log\frac{1+z}{1-z}.\label{eq:Li2int'_B}
\tag{\ref{eq:Li2int'}$_\mathrm{B}$}
\end{align} Similarly, Eq.~\ref{eq:Li2int_S} has a cousin\begin{align}&\int_0^1\frac{\mathbf  K(\sqrt{1-\kappa ^2}) }{4 \pi  \kappa ^2}\left\{\log ^2(1-\kappa ^2)-4\left[ \log\frac{1}{\sqrt{1-\kappa^{2}}} +\log\frac{\sqrt{(1+z^{2})^2 \kappa ^2-4 z ^2}+(1-z ^2) \kappa}{\sqrt{(1+z^{2})^2 \kappa ^2-4 z ^2}+(1+z ^2) \kappa} \right]^{2}\right\}\D\kappa\notag\\={}&z \dilog(z)-z \dilog(-z)-(1+z) \log (1+z)-(1-z) \log (1-z),\quad|z|\leq1.\label{eq:Li2int_S_B}\tag{\ref{eq:Li2int_S}$_\mathrm{B}$}\end{align}In particular, the special case of  $ z=i$  for the equation above leads to an analog of Eq.~\ref{eq:G_Klog}:\begin{align}
G={}&\frac{\pi}{4}-\frac{1}{2}\log2-\frac{1}{2\pi}\int_0^1 \frac{\mathbf K( \sqrt{1-\kappa^{2}})}{\kappa^2} \log (1-\kappa)\log(1+\kappa)\D \kappa.\label{eq:G_Klog'}\tag{\ref{eq:G_Klog}$_\mathrm{B}$}\end{align}

The following identities also follow naturally from the Beltrami transformation:\begin{align}\pi\arcsin x={}&\frac4\pi\int_0^1 \frac{\mathbf K( \sqrt{1-\kappa^{2}})\arcsin(\kappa x)}{\kappa  \sqrt{1- \kappa ^2x^2}}\D\kappa,&x\in {}&[0, 1],\tag{\ref{eq:K_arcsin}$_\mathrm{B}$}\label{eq:K_arcsin'}\\\pi\log( x+\sqrt{1+x^2})={}&\frac4\pi\int_0^1 \frac{\mathbf K( \sqrt{1-\kappa^{2}})}{1-\kappa^{2}}\left[\arcsin\left(\frac{x}{\sqrt{1+x^2}}\right)-\frac{\kappa}{\sqrt{1+(1-\kappa ^2 )x^2}}   \arcsin\left(\frac{\kappa  x}{\sqrt{1+x^2}}\right)\right]\D\kappa,& x\in{}& [0,+\infty).\label{eq:K_arsinh'}\tag{\ref{eq:K_arsinh}$_\mathrm{B}$}\\\intertext{(Here, an equivalent form of  Eq.~\ref{eq:K_arcsin'} appeared as 4.21.4.5 in \cite{Brychkov2008}.) If we analytically continue Eq.~\ref{eq:K_arcsin'} from real-valued $x$ to  purely imaginary $x$, we obtain yet another integral formula (see also 4.21.3.10 in \cite{Brychkov2008}):}\pi\log( x+\sqrt{1+x^2})={}&\frac4\pi\int_0^1 \frac{\mathbf K( \sqrt{1-\kappa^{2}})\log( \kappa x+\sqrt{1+\kappa ^{2}x^2})}{\kappa  \sqrt{1+ \kappa ^2x^2}}\D\kappa,&x\in  {}&[0,+ \infty).\label{eq:K_arcsin''}\tag{\ref{eq:K_arcsin}$^*_\mathrm{B}$}\intertext{When a similar service is performed on  Eq.~\ref{eq:K_arsinh'}, we may deduce}\pi\arcsin x={}&\frac4\pi\int_0^1 \frac{\mathbf K( \sqrt{1-\kappa^{2}})}{1-\kappa^{2}}\left[\sinh^{-1}\left(\frac{x}{\sqrt{1-x^2}}\right)-\frac{\kappa}{\sqrt{1-(1-\kappa ^2 )x^2}}   \sinh^{-1}\left(\frac{\kappa  x}{\sqrt{1-x^2}}\right)\right]\D\kappa,& x\in  {}&[0, 1)\label{eq:K_arsinh''}\tag{\ref{eq:K_arsinh}$^*_\mathrm{B}$}\end{align}instead.

With the  Beltrami transformation in Eq.~\ref{eq:Beltrami}, one may  readily verify the  following  integrals of $ \mathbf K(\sqrt{1-\kappa^2})$ times an elementary function of $ \kappa$, which may otherwise appear formidable:\begin{align}
\int_{0}^{1}\frac{2\mathbf K(\sqrt{1-\kappa^2})[ \kappa(1+r)  \sqrt{r}  \tanh ^{-1}\sqrt{r}-r \sqrt{1-\kappa ^2}  \arcsin \kappa ]}{\pi  (1+r) \kappa  [4 r+(1-r)^2 \kappa ^2]}\D\kappa=\frac{\pi\mathbf K(|r|)}{8}-\frac{\pi^2}{16(1+r)},\label{eq:K_reprod'}\tag{\ref{eq:K_reprod}$_\mathrm{B}$}
\end{align}\begin{align}&
\frac{2}{\pi}\int_{0}^{1}\frac{\mathbf K(\sqrt{1-\kappa^2})}{1-\kappa^2}\left[ \log(1+a)-\kappa\tanh^{-1}\kappa-\frac{\kappa}{\sqrt{1-(1-\kappa^{2})a^{2}}}\log\frac{\sqrt{1-\kappa^2}[a \kappa +\sqrt{1-(1-\kappa^{2})a^{2}}]}{ \kappa +\sqrt{1-(1-\kappa^{2})a^{2}}}\right]\D\kappa \notag\\={}&\begin{cases}\dfrac{\arcsin a}{2}\left( \pi-\arcsin a \right), & 0\leq a\leq1 \\[8pt]\dfrac{\pi ^2}{8}+
\dfrac{\log ^2(\sqrt{a^2-1}+a)}{2} , & a>1 \\
\end{cases}\label{eq:Abel1'}\tag{\ref{eq:Abel1}$_\mathrm{B}$}
\end{align}and\begin{align}&\frac{2}{\pi}\int_{0}^{1}\frac{\mathbf K(\sqrt{1-\kappa^2})}{1-\kappa^2}\left[ \frac{(1+a^{2}+\kappa^{2}-a^2 \kappa ^2) \arctan a}{2(1-\kappa^{2})}-\frac{a}{2}-\frac{\kappa  \sqrt{1+(1-\kappa^{2})a^{2}} }{1-\kappa^{2}}\arctan\frac{a \kappa }{\sqrt{1+(1-\kappa^{2})a^{2}}}\right]\D\kappa\notag\\={}&\frac{\pi}{8}\left[ (1+2a^2)\log(a+\sqrt{1+a^2}) -a\sqrt{1+a^{2}}\right],\quad a\geq0.\label{eq:Abel2'}\tag{\ref{eq:Abel2}$_\mathrm{B}$}\end{align}

The output of the type-B  transformation (Eq.~\ref{eq:Beltrami}) always takes on the shape of $ \int_0^1\mathbf K(\sqrt{1-\kappa^2})\rho_e(  \kappa)\D\kappa$ where $ \rho_e(\kappa)$ is an even analytic function of $\kappa$. However, the  type-iB   transformation (Eq.~\ref{eq:iB})\begin{align*}\int_0^1\mathbf K(k)f(k)\D k=\int_0^1\mathbf K\left(\sqrt{1-\kappa^2}\right)\left[ \frac{2}{\pi}\int_0^1\frac{\sqrt{1-\xi^2}f(\xi)\D\xi}{1-\xi^2(1-\kappa^2)} \right]\D\kappa
\end{align*} results in a host of integral formulae for  $ \int_0^1\mathbf K(\sqrt{1-\kappa^2})\rho(  \kappa)\D\kappa$  in a different flavor, where $ \rho(\kappa)$ is not necessarily an  even function in $ \kappa$. For example, we can transform a formerly derived integral representation of Catalan's constant into\begin{align}G=\frac{1}{2}\int_{0}^1\mathbf K(\xi)\D \xi=\frac{1}{2}\int_0^1\frac{\mathbf K(\sqrt{1-\kappa^2})}{1+\kappa}\D\kappa=\frac{1}{2}\int_0^1\frac{\mathbf K(\xi)\xi}{1+\sqrt{1-\xi^{2}}}\frac{\D\xi}{\sqrt{1-\xi^{2}}}=\frac{1}{\pi}\int_0^1\mathbf K\left(\sqrt{1-\kappa^2}\right)\left[ \frac{\kappa\arccos\kappa}{\sqrt{1-\kappa^2}}-\log(2\kappa) \right]\D\kappa,\label{eq:G_odd}\end{align}after two rounds of applications of the type-iB transformation (Eq.~\ref{eq:iB}).

We  gather below some other consequences of  the  integral transform in Eq.~\ref{eq:iB} while leaving out the computational details of their straightforward proofs:
{\allowdisplaybreaks\begin{align}
\frac{\pi^2}{8}={}&\int_0^1\frac{\mathbf K(\xi)}{1+\xi}\D \xi=\frac{1}{\pi}\int_{0}^{\pi/2}\mathbf K(\sin\theta)\frac{\pi  \sin \theta -2 \theta}{ \cos \theta}  \D\theta,\label{eq:a_iB}\tag{\ref{eq:a}$_\mathrm{iB}$}\\G={}&-\frac{1}{4}\int_0^\beta\log\tan\frac{\pi-\alpha}{4}\D\alpha+\frac{1}{2}\int_{0}^1\frac{\mathbf K(\sqrt{1-\kappa^2})}{1-\kappa^{2}}\left( 1-\frac{\kappa}{\sqrt{\smash[b]{\cos^2({\beta}/{2})+\kappa^{2}\sin^2({\beta}/{2})}}} \right)\D\kappa,\quad 0\leq\beta<\pi,
\label{eq:beta'_diamond}\tag{\ref{eq:beta'}$_\mathrm{iB}$}\\\frac{2G}{3}={}&\frac{1}{2}\int_{0}^1\frac{\mathbf K(\sqrt{1-\kappa^2})}{1-\kappa^{2}}\left( 1-\frac{2\kappa}{\sqrt{\smash[b]{1+3\kappa^{2}}}} \right)\D\kappa,\tag{\ref{eq:beta'}$_\mathrm{iB}$, with $ \beta=2\pi/3$}\\\frac{2G}{5}={}&\frac{1}{2}\int_{0}^1\frac{\mathbf K(\sqrt{1-\kappa^2})}{1-\kappa^{2}}\left[ \frac{(\sqrt{5}+1)\kappa}{\sqrt{\smash[b]{1+(5+2\sqrt{5})\kappa^{2}}}} -\frac{(\sqrt{5}-1)\kappa}{\sqrt{\smash[b]{1+(5-2\sqrt{5})\kappa^{2}}}}\right]\D\kappa,\tag{\ref{eq:beta'}$_\mathrm{iB}$, with $ \beta=2\pi/5,4\pi/5$}\\\frac{7\zeta(3)}{4}={}&\frac{1}{2}\int_{0}^1\frac{\mathbf K(\sqrt{1-\kappa^2})[ \pi-2\kappa\mathbf K(\sqrt{1-\kappa^2})]}{1-\kappa^{2}}\D\kappa,\label{eq:zeta3_iB_int}\tag{\ref{eq:beta'}$_\mathrm{iB}$, integrated over $ 0\leq\beta<\pi$}\\\frac{\pi}{2}={}&\int_{0}^1\frac{\mathbf K(\sqrt{1-\kappa^2})}{1-\kappa^{2}}\left( 1-\frac{\kappa\arccos\kappa}{\sqrt{1-\kappa^2}} \right)\D\kappa,\label{eq:pi_iB_int}\tag{\ref{eq:beta'}$_\mathrm{iB}\times\cos\frac{\beta}{2}$, integrated over $ 0\leq\beta<\pi$}\\\frac14\log\tan\frac{\pi-\beta}{4}={}&-\frac{\sin\beta}{8}\int_{0}^1\frac{\kappa\mathbf K(\sqrt{1-\kappa^2})}{[\cos^2({\beta}/{2})+\kappa^{2}\sin^2({\beta}/{2})]^{3/2}}\D\kappa,\quad 0\leq\beta<\pi,
\label{eq:beta_iB_deriv}\tag{\ref{eq:beta'}$_\mathrm{iB}$, derivative in $ \beta$}
\\\dilog (z)-\dilog (-z)={}&-\frac{2}{\pi}\int_0^1\frac{\mathbf  K(\sqrt{1-\kappa ^2}) }{  1-\kappa ^2}\left[ \log\frac{1-z}{1+z} +\right.\notag\\&+\left.\frac{2z\kappa}{\sqrt{(1+z^2)^2 \kappa ^2-(1-z^2)^2}}\tanh^{-1}\frac{\sqrt{(1+z^2)^2 \kappa ^2-(1-z^2)^2}}{(1+z^2) \kappa}\right]\D\kappa,\quad |z|\leq 1\text{ and }z^2\neq1,\label{eq:Li2int_diamond}\tag{\ref{eq:Li2int}$_\mathrm{iB}$}\\\frac{\pi^{2}}{4}={}&-\frac{2}{\pi}\int_0^1\frac{\mathbf  K(\sqrt{1-\kappa ^2}) }{  1-\kappa ^2}\log\kappa\D\kappa,\tag{\ref{eq:Li2int}$_\mathrm{iB}$, with $ z=1$}\\2={}&\frac{4}{\pi}\int_0^1\frac{\mathbf  K(\sqrt{1-\kappa ^2}) }{  1-\kappa ^2}\left( 1-\frac{\kappa\arccos\kappa}{\sqrt{1-\kappa^{2}}} \right)\D\kappa,\label{eq:pi_Li2_deriv_iB}\tag{\ref{eq:Li2int}$_\mathrm{iB}$, derivative at  $ z=0$}\\\frac{\pi}{4z(1-z^2)}\log\frac{1+z}{1-z}={}&-\int_0^1\left[ 1-\frac{(1+z^{2})\kappa}{\sqrt{(1+z^2)^2 \kappa ^2-(1-z^2)^2}} \tanh^{-1}\frac{\sqrt{(1+z^2)^2 \kappa ^2-(1-z^2)^2}}{(1+z^2) \kappa}\right]\frac{\mathbf  K(\sqrt{1-\kappa ^2})\D\kappa }{  (1+z^2)^2 \kappa ^2-(1-z^2)^2},\quad |z|<1,\label{eq:Li2int'_iB}\tag{\ref{eq:Li2int'}$_\mathrm{iB}$}\\\pi\arcsin x={}&\frac4\pi\int_0^1 \left[\tanh^{-1}x-\frac{\kappa}{\sqrt{1-(1-\kappa^2)x^2}}\tanh^{-1}\frac{\kappa x}{\sqrt{1-(1-\kappa^2)x^2}}\right]\frac{\mathbf K( \sqrt{1-\kappa^{2}})\D\kappa}{1-\kappa^{2}},\quad 0\leq x< 1,\label{eq:K_arcsin_diamond}\tag{\ref{eq:K_arcsin}$_\mathrm{iB}$}\\\frac{\pi}{\sqrt{1-x^{2}}}={}&\frac4\pi\int_0^1 \left[1-\frac{\kappa x}{\sqrt{1-(1-\kappa^2)x^2}}\tanh^{-1}\frac{\kappa x}{\sqrt{1-(1-\kappa^2)x^2}}\right]\frac{\mathbf K( \sqrt{1-\kappa^{2}})\D\kappa}{1-(1-\kappa^2)x^2},\quad 0\leq x< 1,\label{eq:K_arcsin_iB_deriv}\tag{\ref{eq:K_arcsin}$_\mathrm{iB}$, derivative in $ x$}\\\pi \log(x+\sqrt{1+x^2})={}&\frac4\pi\int_0^1 \tanh^{-1}\frac{\kappa x}{\sqrt{1+\kappa^{2}x^2}}\frac{\mathbf K( \sqrt{1-\kappa^{2}})\D\kappa}{\kappa\sqrt{1+\kappa^{2}x^2}},\quad 0\leq x<+\infty,\label{eq:K_arsinh_diamond}\tag{\ref{eq:K_arsinh}$ _\mathrm{iB}$}\\\frac{\pi}{\sqrt{1+x^{2}}}={}&\frac4\pi\int_0^1 \left(1-\frac{\kappa x}{\sqrt{1+\kappa^{2}x^2}}\tanh^{-1}\frac{\kappa x}{\sqrt{1+\kappa^{2}x^2}}\right)\frac{\mathbf K( \sqrt{1-\kappa^{2}})\D\kappa}{1+\kappa^{2}x^2{}},\quad 0\leq x<+\infty.\label{eq:K_arsinh_iB_deriv}\tag{\ref{eq:K_arsinh}$_\mathrm{iB}$, derivative in $ x$}\end{align}}As we make the substitution $ x\mapsto x/\sqrt{1+x^2}$ in the antepenultimate formula, we also obtain
the last formula. It is also worth noting that Eqs.~\ref{eq:K_arcsin''} and \ref{eq:K_arsinh_diamond} are in fact identical.
The same can be said for the pair of equations \ref{eq:K_arsinh''} and \ref{eq:K_arcsin_diamond}.

Applying  the  type-LB and type-iLB transformations to Eq.~\ref{eq:a},  we obtain the following identities \begin{align}\frac{\pi^2}{8}={}&\int_0^1\frac{\mathbf K(r)}{1+r}\D r=\int_0^1\frac{2\mathbf K(\sqrt{1-\kappa^2})}{\pi}\left[ \int_{0}^{1}\frac{\D r}{(1+r)^2-4r\kappa^2} \right]\D\kappa=\int_0^1\frac{\mathbf K(\sqrt{1-\kappa^2})}{\pi\kappa\sqrt{1-\kappa^{2}}}\arcsin\kappa\D\kappa,\tag{\ref{eq:a}$_\mathrm{LB}$}\label{eq:a'_L}\\\frac{\pi^2}{8}={}&\int_0^1\frac{\mathbf K(\eta)}{1+\eta}\D \eta=\int_0^1\frac{2\mathbf K(\sqrt{1-\kappa^2})}{\pi}\left[ \int_{0}^{1}\frac{1}{1+\eta}\frac{(1-\eta)\D \eta}{(1-\eta)^2+4\eta\kappa^2} \right]\D\kappa=-\int_0^1\frac{\mathbf K(\sqrt{1-\kappa^2})}{\pi(1-\kappa^{2}){}}\log\kappa\D\kappa.\tag{\ref{eq:a}$_\mathrm{iLB}$}\label{eq:a_diamond_L}\end{align}Neither of them is particularly new, as Eq.~\ref{eq:a'_L} follows from the $ x=1$ specific case in Eq.~\ref{eq:K_arcsin'} or a variable transformation $ k=\sqrt{1-\kappa^2}$ in Eq.~\ref{eq:pi8_arccos}, and the $ z\to1^-$ limit of Eq.~\ref{eq:Li2int_diamond} recovers Eq.~\ref{eq:a_diamond_L}. Doubtlessly, the two equations above also arise from the applications of  type-B and type-iB transformations to the form of Eq.~\ref{eq:a} after Landen's transformation: \begin{align*}\frac{\pi^2}{8}={}&\frac12\int_0^1\mathbf K\left( \sqrt{1-k^2} \right)\D k=\int_0^1\frac{\mathbf K(\sqrt{1-\kappa^2})}{\pi}\left[ \int_{0}^{1}\frac{\D k}{1-(1-k^{2})\kappa^2} \right]\D\kappa=\int_0^1\frac{\mathbf K(\sqrt{1-\kappa^2})}{\pi\kappa\sqrt{1-\kappa^{2}}}\arcsin\kappa\D\kappa,\\\frac{\pi^2}{8}={}&\frac12\int_0^1\mathbf K\left( \sqrt{\smash[b]{1-\xi^2}} \right)\D \xi=\int_0^1\frac{\mathbf K(\sqrt{1-\kappa^2})}{\pi}\left[ \int_{0}^{1}\frac{\xi\D \xi}{1-(1-\xi^{2})(1-\kappa^2)} \right]\D\kappa=-\int_0^1\frac{\mathbf K(\sqrt{1-\kappa^2})}{\pi(1-\kappa^{2}){}}\log\kappa\D\kappa.\end{align*}In the next subsection, we will  perform the  transformations specified by Eqs.~\ref{eq:Beltrami} and \ref{eq:iB}  on  multiple elliptic integrals which have undergone Landen's transformation  beforehand. In these contexts, we will be effectively using the transformations of types ``LB'' and ``iLB''  in the disguise of types ``B'' and ``iB''.

\subsection{Duality Relations\label{subsubsec:duality}}
The next corollary presents some further applications  of the Beltrami transformations (re)introduced in \S\ref{subsubsec:B_simp}.\begin{corollary}[Duality Relations]\label{cor:duality}\begin{enumerate}[label=\emph{(\alph*)}, ref=(\alph*), widest=a]\item We have the following identities for $ r>0$:{\allowdisplaybreaks\begin{align}W(r):={}&\int_0^1\frac{r\mathbf K(\xi)\D\xi}{r^{2}+\xi^2}=\int_0^1\frac{\mathbf K(\sqrt{1-\kappa^2})\D\kappa}{\sqrt{1+r^2}+\kappa r}\label{eq:GB_1}\\={}&\frac{2}{\pi}\int_0^1\left(\frac{ \kappa r  \arccos\kappa }{\sqrt{1-\kappa ^2} }-\sqrt{1+r^{2}}\log \kappa  \right)\frac{\mathbf K(\sqrt{1-\kappa^2})\D\kappa}{1+r^{2}-\kappa^2}-\frac{\log \left(1+\frac{r}{\sqrt{1+r^{2}}}\right)}{\sqrt{1+r^{2}}}\mathbf K\left( \frac{1}{\sqrt{1+r^2}} \right)\label{eq:GB_2}\\={}&\frac{2}{\pi}\int_0^1\left( \frac{ \kappa \sqrt{1+r^2}  \arcsin\kappa }{\sqrt{1-\kappa ^2} }+r\log \sqrt{1-\kappa^{2}}  \right)\frac{\mathbf K(\sqrt{1-\kappa^2})\D\kappa}{r^{2}+\kappa^2}+\frac{\log \left(1+\frac{r}{\sqrt{1+r^{2}}}\right)}{\sqrt{1+r^{2}}}\mathbf K\left( \frac{1}{\sqrt{1+r^2}} \right)\label{eq:GB_3}\\={}&\frac{2}{\pi}\int_0^1r\log \frac{\sqrt{1-\kappa^{2}}}{\kappa}\frac{\mathbf K(\sqrt{1-\kappa^2})\D\kappa}{r^{2}+\kappa^2}+\frac{\log \frac{r}{\sqrt{1+r^{2}}}}{\sqrt{1+r^{2}}}\mathbf K\left( \frac{1}{\sqrt{1+r^2}} \right).\label{eq:GB_4}\end{align}}This implies two more integral representations of Catalan's constant \begin{align}G=\frac{1}{2}\lim_{r\to+\infty}\int_0^1\frac{r^{2}\mathbf K(\xi)\D\xi}{r^{2}+\xi^2}={}&\frac{1}{\pi}\int_0^1\left( \frac{ \kappa\arcsin\kappa }{\sqrt{1-\kappa ^2} }+\log \sqrt{4-4\kappa^{2}}  \right)\mathbf K\left(\sqrt{1-\kappa^2}\right)\D\kappa=\frac{1}{\pi}\int_0^1\log \frac{\sqrt{1-\kappa^{2}}}{\kappa}\mathbf K\left(\sqrt{1-\kappa^2}\right)\D\kappa,\label{eq:GB_G}\end{align} as well as
the functional equation for $ W(r),r>0$\begin{align}&W(r)+iW\left(i\sqrt{{1+r^{2}}}\right)=4i\left(r+\sqrt{1+r^{2}}\right)W\left(i\left(r+\sqrt{1+r^{2}}\right)^2\right)\label{eq:f_eq_W}\\\Longleftrightarrow&\int_0^1\frac{r\mathbf K(\xi)\D\xi}{r^{2} +\xi^2}+\int_0^1\frac{\sqrt{1+r^{2}}\mathbf K(\xi)\D\xi}{1+r^{2} -\xi^2}=\int_0^1\frac{4(r+\sqrt{1+r^{2}})^3\mathbf K(\eta)\D\eta}{(r+\sqrt{1+r^{2}})^4-\eta^{2}}.\label{eq:f_eq_W'}\tag{\ref{eq:f_eq_W}$'$}\end{align}
\item We have the identities\begin{align}
M(r):={}&\int_0^1\frac{r\xi\mathbf K(\xi)\D\xi}{1+r^{2}\xi^2}=-\frac{2}{\pi}\int_0^1r\log \sqrt{1-\kappa^{2}}\frac{\mathbf K(\sqrt{1-\kappa^2})\D\kappa}{r^{2}+\kappa^2}+\frac{\log\sqrt{1+r^{2}}}{\sqrt{1+r^{2}}}\mathbf K\left( \frac{1}{\sqrt{1+r^2}} \right)\label{eq:GB_5}\\={}&-\frac{2}{\pi}\int_0^1\frac{ \kappa r  \arccos\kappa }{\sqrt{1-\kappa ^2} }\frac{\mathbf K(\sqrt{1-\kappa^2})\D\kappa}{1+r^{2}-\kappa^2}+\frac{\log (r+\sqrt{1+r^{2}})}{\sqrt{1+r^{2}}}\mathbf K\left( \frac{1}{\sqrt{1+r^2}} \right),\label{eq:GB_6}
\end{align} the integral formula \begin{align}\int_0^1\left(\frac{r}{r^{2}+\kappa^2}-\frac{\sqrt{1+r^2}}{1+r^{2}-\kappa^2}\right)\mathbf K\left( \sqrt{1-\kappa^2} \right)\log\kappa\D\kappa=\frac{\pi}{2}\frac{\log \frac{r}{\sqrt{1+r^{2}}}}{\sqrt{1+r^{2}}}\mathbf K\left( \frac{1}{\sqrt{1+r^2}} \right),\label{eq:log_combo}\end{align}and the functional relation\begin{align}\label{eq:sum_rule}W(r)+M(r)={}&\int_0^1\frac{r\mathbf K(\xi)\D\xi}{r^{2}+\xi^2}+\int_0^1\frac{r\xi\mathbf K(\xi)\D\xi}{1+r^{2}\xi^2}=\frac{2}{\pi}\int_0^1r\log \frac{r}{\kappa}\frac{\mathbf K(\sqrt{1-\kappa^2})\D\kappa}{r^{2}+\kappa^2}=\frac{2}{\pi}\int_0^1\sqrt{1+r^{2}}\log \frac{\sqrt{1+r^{2}}}{\kappa}\frac{\mathbf K(\sqrt{1-\kappa^2})\D\kappa}{1+r^{2}-\kappa^2}\notag\\={}&\frac{\pi}{2}\frac{1}{\sqrt{1+r^{2}}}\mathbf K\left( \frac{r}{\sqrt{1+r^2}} \right),\end{align}all valid for $ r>0$.
As a consequence, the following integral relations hold for $ x>0$:\begin{align}
\int_0^1\tanh^{-1}\sqrt{\frac{1+x^{2}\kappa^{2}}{1+x^2}}\frac{\mathbf K(\sqrt{1-\kappa^2})\D\kappa}{\sqrt{1+x^{2}\kappa^{2}}}+\int_0^1\tanh^{-1}\sqrt{\frac{x^2+\kappa^2}{1+x^2}}\frac{\mathbf K(\sqrt{1-\kappa^2})\D\kappa}{\sqrt{x^{2}+\kappa^{2}}}={}&\int_0^1\tanh^{-1}\frac{x}{\sqrt{x^2+\kappa^2}}\frac{\mathbf K(\sqrt{1-\kappa^2})\D\kappa}{\sqrt{x^{2}+\kappa^{2}}},
\label{eq:3artanh}\\\int_0^1\tanh^{-1}\frac{1}{\sqrt{1+x^2\kappa^2}}\frac{\mathbf K(\sqrt{1-\kappa^2})\D\kappa}{\sqrt{1+x^{2}\kappa^{2}}}={}&\int_0^1\tanh^{-1}\frac{x}{\sqrt{x^2+\kappa^2}}\frac{\mathbf K(\sqrt{1-\kappa^2})\D\kappa}{\sqrt{x^{2}+\kappa^{2}}},\label{eq:2artanh}
\end{align}and there are several multiple elliptic integral representations for $ \mathbf K(\sqrt{\lambda})\mathbf K(\sqrt{1-\lambda}),0<\lambda<1$:\begin{align}
 \mathbf K(\sqrt{\lambda})\mathbf K(\sqrt{1-\lambda})={}&\int^1_{(1-\lambda)/(1+\lambda)}\frac{\mathbf K(k)\D k}{\sqrt{1-k^{2}}\sqrt{(1+\lambda)^{2}k^2-(1-\lambda)^2}}=\int_0^1\frac{\mathbf K(k)\D k}{\sqrt{4\lambda+(1-\lambda)^2k^2}}+\int_0^1\frac{\mathbf K(k)\D k}{\sqrt{(1-\lambda)^2+4\lambda k^2}}\notag\\={}&\frac{2}{\pi}\int_0^1\tanh^{-1}\frac{2\sqrt{\lambda}}{\sqrt{4\lambda+(1-\lambda)^{2}\kappa^2}}\frac{\mathbf K(\sqrt{1-\kappa^2})\D\kappa}{\sqrt{4\lambda+(1-\lambda)^{2}\kappa^2}}\notag\\={}&\frac{2}{\pi}\int_0^1\tanh^{-1}\frac{1-\lambda}{\sqrt{(1-\lambda)^{2}+4\lambda\kappa^2}}\frac{\mathbf K(\sqrt{1-\kappa^2})\D\kappa}{\sqrt{(1-\lambda)^{2}+4\lambda\kappa^2}}.\label{eq:KK'_various}
\end{align}
\end{enumerate}\end{corollary}\begin{proof}\begin{enumerate}[label=(\alph*),widest=a]\item With the type-iB  transformation (Eq.~\ref{eq:iB}) and the elementary integration\[\frac{2}{\pi}\int_0^1\frac{r}{r^2+\xi^2}\frac{\sqrt{1-\xi^2}\D\xi}{1-\xi^{2}(1-\kappa^2)}=\frac{1}{\sqrt{1+r^2}+\kappa r},\quad r>0\]one may reveal the truth of Eq.~\ref{eq:GB_1}. Applying a second round of type-iB  transformation to the right-hand side of  Eq.~\ref{eq:GB_1}, we obtain\begin{align*}\int_0^1\frac{\mathbf K(\sqrt{1-\kappa^2})\D\kappa}{\sqrt{1+r^2}+\kappa r}=\int_0^1\frac{\mathbf K(\xi)}{\sqrt{1+r^2}+r \sqrt{1-\xi^{2}}}\frac{\xi\D\xi}{\sqrt{1-\xi^{2}}}=\frac{2}{\pi}\int_0^1\left[ \frac{ \kappa r  \arccos\kappa }{\sqrt{1-\kappa ^2} }-\sqrt{1+r^{2}}\log \frac{\kappa(r+\sqrt{1+r^{2}})}{\sqrt{1+r^{2}}}  \right]\frac{\mathbf K(\sqrt{1-\kappa^2})\D\kappa}{1+r^{2}-\kappa^2}.\end{align*}As we recognize that \begin{align*}\frac{2}{\pi}\int_0^1\frac{\mathbf K(\sqrt{1-\kappa^2})\D\kappa}{1+r^{2}-\kappa^2}=\frac{1}{1+r^{2}}\mathbf K\left( \frac{1}{\sqrt{1+r^2}} \right),\end{align*}we may confirm Eq.~\ref{eq:GB_2}. Now, writing $\frac\pi2-\arcsin\kappa$ for $ \arccos\kappa$, and noticing that \begin{align*}\int_0^1\frac{ \kappa r   }{\sqrt{1-\kappa ^2} }\frac{\mathbf K(\sqrt{1-\kappa^2})\D\kappa}{1+r^{2}-\kappa^2}=\int_0^1\frac{r\mathbf K(\xi)\D\xi}{r^{2}+\xi^2}\end{align*}follows from  a simple substitution $ \kappa\mapsto\sqrt{1-\xi^2}$, we see that   Eqs.~\ref{eq:GB_1} and \ref{eq:GB_2} entail the identity\begin{align}\int_0^1\frac{ \kappa r  \arcsin\kappa }{\sqrt{1-\kappa ^2} }\frac{\mathbf K(\sqrt{1-\kappa^2})\D\kappa}{1+r^{2}-\kappa^2}=\int_0^1\sqrt{1+r^{2}}\log \frac{1}{\kappa}\frac{\mathbf K(\sqrt{1-\kappa^2})\D\kappa}{1+r^{2}-\kappa^2}-\frac{\pi}{2}\frac{\log \left(1+\frac{r}{\sqrt{1+r^{2}}}\right)}{\sqrt{1+r^{2}}}\mathbf K\left( \frac{1}{\sqrt{1+r^2}} \right).\label{eq:arcsin_id}\end{align}By analytic continuation in $r$, we may convert the last identity into\begin{align}\int_0^1\frac{ \kappa\sqrt{1+r^{2}}  \arcsin\kappa }{\sqrt{1-\kappa ^2} }\frac{\mathbf K(\sqrt{1-\kappa^2})\D\kappa}{r^{2}+\kappa^2}=\int_0^1r\log \frac{1}{\kappa}\frac{\mathbf K(\sqrt{1-\kappa^2})\D\kappa}{r^{2}+\kappa^2}-\frac{\pi}{2}\frac{\log\left(1+\frac{\sqrt{1+r^2}}{r}\right)}{\sqrt{1+r^{2}}}\mathbf K\left( \frac{1}{\sqrt{1+r^2}} \right).\tag{\ref{eq:arcsin_id}$'$}\label{eq:arcsin_id'}\end{align}

On the other hand, if we apply type-B transformation (Eq.~\ref{eq:Beltrami}) to the right-hand side of Eq.~\ref{eq:GB_1}, we may set up the following equation\begin{align*}\int_0^1\frac{\mathbf K(\sqrt{1-\kappa^2})\D\kappa}{\sqrt{1+r^2}+\kappa r}={}&\int_0^1\frac{\mathbf K(k)}{\sqrt{1+r^2}+\kappa \sqrt{1-k^{2}}}\frac{k\D k}{\sqrt{1-k^{2}}}\notag\\={}&\frac{2}{\pi}\int_0^1\left[ \frac{ \kappa \sqrt{1+r^2}  \arcsin\kappa }{\sqrt{1-\kappa ^2} }+r\log \frac{\sqrt{1-\kappa^{2}}(r+\sqrt{1+r^{2}})}{\sqrt{1+r^{2}}}  \right]\frac{\mathbf K(\sqrt{1-\kappa^2})\D\kappa}{r^{2}+\kappa^2},\end{align*}which implies Eq.~\ref{eq:GB_3} after we exploit the identity  \[\frac{2}{\pi}\int_0^1\frac{r\mathbf K(\sqrt{1-\kappa^2})\D\kappa}{r^{2}+\kappa^2}=\frac{1}{\sqrt{1+r^{2}}}\mathbf K\left( \frac{1}{\sqrt{1+r^2}} \right).\] As we set Eq.~\ref{eq:arcsin_id'} into  Eq.~\ref{eq:GB_3}, we obtain Eq.~\ref{eq:GB_4}.
In passage to the limit $ r\to+\infty$, Eqs.~\ref{eq:GB_1} and \ref{eq:GB_2} recover the identities stated in Eq.~\ref{eq:G_odd}, while Eqs.~\ref{eq:GB_3} and \ref{eq:GB_4} lead to Eq.~\ref{eq:GB_G}.

After applying Landen's transformation on the right-hand side of   Eq.~\ref{eq:GB_1}, we get \begin{align*}W(r)=\int_0^1\frac{r\mathbf K(\xi)\D\xi}{r^{2}+\xi^2}=\int_0^1\frac{2\mathbf K(\eta)\D\eta}{\sqrt{1+r^{2}}+r+(\sqrt{1+r^2}-r)\eta}.\end{align*}Setting $ r=\sinh s$ in the equation above, we obtain\[e^sW(\sinh s)=e^{s}\int_0^1\frac{\sinh s\mathbf K(\xi)\D\xi}{\sinh^{2} s+\xi^2}=\int_0^1\frac{2\mathbf K(\eta)\D\eta}{1+e^{-2s}\eta},\quad s>0.\]By analytic continuation from $ s$ to $ s+\frac{i\pi}{2}$, we have\begin{align*}ie^{s}W(i\cosh s)=e^{s}\int_0^1\frac{\cosh s\mathbf K(\xi)\D\xi}{\cosh^{2} s-\xi^2}=\int_0^1\frac{2\mathbf K(\eta)\D\eta}{1-e^{-2s}\eta},\quad s>0.\end{align*}Adding up, we arrive at the following functional relation \begin{align*}e^sW(\sinh s)+ie^{s}W(i\cosh s)=e^{s}\left[\int_0^1\frac{\sinh s\mathbf K(\xi)\D\xi}{\sinh^{2} s+\xi^2}+\int_0^1\frac{\cosh s\mathbf K(\xi)\D\xi}{\cosh^{2} s-\xi^2}\right]=\int_0^1\frac{4\mathbf K(\eta)\D\eta}{1-e^{-4s}\eta^{2}}=4ie^{2s}W(ie^{2s}),\end{align*}which entails Eqs.~\ref{eq:f_eq_W} and \ref{eq:f_eq_W}$'$ after the back substitution $ s=\sinh^{-1}r$.
 \item By the transformations of types B (Eq.~\ref{eq:Beltrami}) and iB (Eq.~\ref{eq:iB}), we may establish
the integral representations\begin{align*}M(r):={}&\int_0^1\frac{r\xi\mathbf K(\xi)\D\xi}{1+r^{2}\xi^2}=-\frac{2}{\pi}\int_0^1r\log \sqrt{\frac{1-\kappa^{2}}{1+r^{2}}}\frac{\mathbf K(\sqrt{1-\kappa^2})\D\kappa}{r^{2}+\kappa^2}\\={}&-\frac{2}{\pi}\int_0^1\left(\frac{ \kappa r  \arccos\kappa }{\sqrt{1-\kappa ^2} }-\sqrt{1+r^{2}}\sinh^{-1}r\right)\frac{\mathbf K(\sqrt{1-\kappa^2})\D\kappa}{1+r^{2}-\kappa^2},\end{align*}which simplify into Eqs.~\ref{eq:GB_5} and \ref{eq:GB_6}. As we combine Eqs.~\ref{eq:GB_4} and \ref{eq:GB_5}, we obtain\begin{align}W(r)+M(r)=\frac{2}{\pi}\int_0^1r\log \frac{1}{\kappa}\frac{\mathbf K(\sqrt{1-\kappa^2})\D\kappa}{r^{2}+\kappa^2}+\frac{\log r}{\sqrt{1+r^{2}}}\mathbf K\left( \frac{1}{\sqrt{1+r^2}} \right)=\frac{2}{\pi}\int_0^1r\log \frac{r}{\kappa}\frac{\mathbf K(\sqrt{1-\kappa^2})\D\kappa}{r^{2}+\kappa^2},\quad r>0,\label{eq:LM1}\end{align}whereas Eqs.~\ref{eq:GB_2} and \ref{eq:GB_6} together entail the relation
\begin{align}W(r)+M(r)={}&\frac{2}{\pi}\int_0^1\sqrt{1+r^{2}}\log \frac{1}{\kappa}\frac{\mathbf K(\sqrt{1-\kappa^2})\D\kappa}{1+r^{2}-\kappa^2}+\frac{\log \sqrt{1+r^{2}}}{\sqrt{1+r^{2}}}\mathbf K\left( \frac{1}{\sqrt{1+r^2}} \right)\notag\\={}&\frac{2}{\pi}\int_0^1\sqrt{1+r^{2}}\log \frac{\sqrt{1+r^{2}}}{\kappa}\frac{\mathbf K(\sqrt{1-\kappa^2})\D\kappa}{1+r^{2}-\kappa^2},\quad r>0.\label{eq:LM2}\end{align}Equating the  two expressions for $ W(r)+M(r)$ in Eqs.~\ref{eq:LM1} and \ref{eq:LM2}, we have verified Eq.~\ref{eq:log_combo}.

To establish the last equality in Eq.~\ref{eq:sum_rule}, we set $ r=k/\sqrt{1-k^2}$ for $ 0<k<1$, and make use of the rightmost term in Eq.~\ref{eq:GB_1} to put down\begin{align*}W(r)=W\left( \frac{k}{\sqrt{1-k^{2}}} \right)=\sqrt{1-k^{2}}\int_0^1\frac{\mathbf K(\sqrt{1-\kappa^2})\D\kappa}{1+k\kappa},\end{align*}which adds up with \begin{align*}M(r)=M\left( \frac{k}{\sqrt{1-k^{2}}} \right)=\int_0^1\frac{r\xi\mathbf K(\xi)\D\xi}{1+r^{2}\xi^2}\overset{\kappa=\sqrt{1-\xi^2}}{=\!\!=\!\!=\!\!=\!\!=\!\!=\!\!=\!\!=}\int_0^1\frac{r\kappa\mathbf K(\sqrt{1-\kappa^{2}})\D\kappa}{1+r^{2}(1-\kappa^2)}=\sqrt{1-k^{2}}\int_0^1\frac{k\kappa\mathbf K(\sqrt{1-\kappa^2})\D\kappa}{1-k^{2}\kappa^{2}}\end{align*}into a form anticipated from the Beltrami transformation in Eq.~\ref{eq:Beltrami}:\begin{align*}W(r)+M(r)=W\left( \frac{k}{\sqrt{1-k^{2}}} \right)+M\left( \frac{k}{\sqrt{1-k^{2}}} \right)=\sqrt{1-k^{2}}\int_0^1\frac{\mathbf K(\sqrt{1-\kappa^2})\D\kappa}{1-k^{2}\kappa^{2}}=\frac{\pi\sqrt{1-k^2}}{2}\mathbf K(k)=\frac{\pi}{2}\frac{1}{\sqrt{1+r^2}}\mathbf K\left( \frac{r}{\sqrt{1+r^2}} \right).\end{align*}As we may recall that the last term in Eq.~\ref{eq:GB_1} itself is a consequence of the Beltrami transformation of type iB (Eq.~\ref{eq:iB}), it might be appropriate to name the identity expressed by the two extreme ends of  Eq.~\ref{eq:sum_rule}\[\int_0^1\frac{r\mathbf K(\xi)\D\xi}{r^{2}+\xi^2}+\int_0^1\frac{r\xi\mathbf K(\xi)\D\xi}{1+r^{2}\xi^2}=\frac{\pi}{2}\frac{1}{\sqrt{1+r^2}}\mathbf K\left( \frac{r}{\sqrt{1+r^2}} \right)\]as the sesqui-Beltrami transformation.

We perform Abel transform on both sides of the last equation:\begin{align*}\int_x^\infty\frac{1}{\sqrt{r^2-x^2}}\left[\int_0^1\frac{r\mathbf K(\xi)\D\xi}{r^{2}+\xi^2}+\int_0^1\frac{r\xi\mathbf K(\xi)\D\xi}{1+r^{2}\xi^2}\right]\D r={}&\frac{\pi}{2}\int_x^\infty\frac{1}{\sqrt{r^2-x^2}}\left[\int_0^1\frac{\D t}{\sqrt{1-t^2}\sqrt{1+r^2-r^{2}t^2}}\right]\D r,\end{align*}which results in \begin{align}\label{eq:K_tanh_addition}\frac{\pi}{2}\left[\int_0^1\frac{\mathbf K(\xi)\D\xi}{\sqrt{x^{2}+\xi^2}}+\int_0^1\frac{\mathbf K(\xi)\D\xi}{\sqrt{1+x^{2}\xi^2}}\right]={}&\frac{\pi}{2}\int_0^1\frac{1}{x({1-t^2})}\mathbf K\left(\frac{1}{x\sqrt{t^{2}-1}}\right)\D t=x\int_0^1 \left[ \int_0^1\frac{\mathbf K(\sqrt{1-\kappa^2})\D\kappa}{x^{2}(1-t^{2})+\kappa^{2}}\right]\D t\notag\\={}&\int_0^1\tanh^{-1}\frac{x}{\sqrt{x^2+\kappa^2}}\frac{\mathbf K(\sqrt{1-\kappa^2})\D\kappa}{\sqrt{x^{2}+\kappa^{2}}}.\end{align}On the other hand, using the Beltrami transformation in Eq.~\ref{eq:Beltrami}, we may verify \begin{align*}\frac{\pi}{2}\left[\int_0^1\frac{\mathbf K(\xi)\D\xi}{\sqrt{x^{2}+\xi^2}}+\int_0^1\frac{\mathbf K(\xi)\D\xi}{\sqrt{1+x^{2}\xi^2}}\right]=\int_0^1\tanh^{-1}\sqrt{\frac{1+x^{2}\kappa^{2}}{1+x^2}}\frac{\mathbf K(\sqrt{1-\kappa^2})\D\kappa}{\sqrt{1+x^{2}\kappa^{2}}}+\int_0^1\tanh^{-1}\sqrt{\frac{x^2+\kappa^2}{1+x^2}}\frac{\mathbf K(\sqrt{1-\kappa^2})\D\kappa}{\sqrt{x^{2}+\kappa^{2}}},\end{align*} so Eq.~\ref{eq:3artanh} is clearly true. As  we trade $x$ for $1/x$ in    Eq.~\ref{eq:3artanh}, we obtain\begin{align*}x\left[\int_0^1\tanh^{-1}\sqrt{\frac{x^{2}+\kappa^{2}}{1+x^2}}\frac{\mathbf K(\sqrt{1-\kappa^2})\D\kappa}{\sqrt{x^{2}+\kappa^{2}}}+\int_0^1\tanh^{-1}\sqrt{\frac{1+x^2\kappa^2}{1+x^2}}\frac{\mathbf K(\sqrt{1-\kappa^2})\D\kappa}{\sqrt{1+x^{2}\kappa^{2}}}\right]=x\int_0^1\tanh^{-1}\frac{1}{\sqrt{1+x^2\kappa^2}}\frac{\mathbf K(\sqrt{1-\kappa^2})\D\kappa}{\sqrt{1+x^{2}\kappa^{2}}},\end{align*}which gives rise to Eq.~\ref{eq:2artanh}.

In \cite[Ref.][Proposition 3.3(b)]{Zhou2013Pnu}, it was shown that \begin{align*} \mathbf K(\sqrt{\lambda})\mathbf K(\sqrt{1-\lambda})={}&\int^1_{(1-\lambda)/(1+\lambda)}\frac{\mathbf K(k)\D k}{\sqrt{1-k^{2}}\sqrt{(1+\lambda)^{2}k^2-(1-\lambda)^2}}\end{align*}holds for $ 0<\lambda\leq1/2$, and the right-hand side of the equation above is left unchanged after a reflection $ \lambda\mapsto1-\lambda$. As we carry out the following computations for $ b>1$:\begin{align*}\int_{1/b}^1\frac{\mathbf K(\xi)\D \xi}{\sqrt{1-\xi^2}\sqrt{b^2\xi^2-1}}=\int^{1}_{1/b}\left[\frac{2\sqrt{1-\xi^2}}{\pi}\int_0^1\frac{\mathbf K(\sqrt{1-\kappa^2})\D\kappa}{1-\xi^2(1-\kappa^2)}\right]\frac{\D \xi}{\sqrt{1-\xi^2}\sqrt{b^2\xi^2-1}}=\frac{2}{\pi}\int_0^1\tanh^{-1}\sqrt{\frac{b^2-1}{b^2-1+\kappa^2}}\frac{\mathbf K(\sqrt{1-\kappa^2})\D\kappa}{\sqrt{b^{2}-1+\kappa^2}}\end{align*}before invoking Eq.~\ref{eq:K_tanh_addition}, we can prove that \begin{align*}\int_0^1\frac{\mathbf K(k)\D k}{\sqrt{b^2-1+k^2}}+\int_0^1\frac{\mathbf K(k)\D k}{\sqrt{1+(b^2-1)k^2}}={}&\int_{1/b}^1\frac{\mathbf K(k)\D k}{\sqrt{1-k^2}\sqrt{b^2k^2-1}},\quad b>1,\end{align*}which verifies Eq.~\ref{eq:KK'_various}. \qed\end{enumerate}\end{proof}

\section{Moment Expansion and Hypergeometric Summations\label{subsec:comb}}

Some of the integral identities developed in \S\S\ref{subsec:S2}-\ref{sec:Beltrami_revisited} can be recast into infinite series, where each summand is expressed combinatorially. This leads to some hypergeometric summations possibly of independent interest.

For example,  using the moments of $ \mathbf K(\sqrt{1-\kappa^2})$ \cite{BBGW} \[\int_0^1\kappa^n\mathbf K\left(\sqrt{1-\kappa^2}\right)\D\kappa=\frac{\pi}{4}\left[ \frac{\Gamma(\frac{n+1}{2})}{\Gamma(\frac{n+2}{2})} \right]^2,\] we may reformulate some previously derived results as{\allowdisplaybreaks\begin{align}G={}&\frac{1}{2\pi}\int_0^1\frac{\mathbf K(\sqrt{1-\kappa^2})}{\kappa}\log\frac{1+\kappa}{1-\kappa}\D\kappa=\frac{\pi}{4}\sum_{m=0}^\infty\frac{1}{2n+1}\left[ \frac{(2n-1)!!}{2^{n}n!} \right]^2=\frac{\pi}{4}\,{_3F_2}\left( \left.\begin{array}{c}
\frac{1}{2},\frac{1}{2},\frac{1}{2} \\[2pt]
1,\frac{3}{2} \\
\end{array}\right|1\right),\label{eq:b_series_exp}\tag{\ref{eq:b}$_\Sigma$}\\\frac{2\pi G-7\zeta(3)}{2}={}&-\frac{1}{2}\int_0^1\mathbf K \left(\sqrt{1-\kappa^{2}} \right)\log\frac{1+\kappa}{1-\kappa}\D \kappa=-\sum_{n=0}^\infty\frac{1}{2n+1}\left[ \frac{2^{n}n!}{(2n+1)!!} \right]^2=-\,{_4F_3}\left( \left.\begin{array}{c}
\frac{1}{2},1,1,1 \\[2pt]
\frac{3}{2},\frac{3}{2} ,\frac32\\
\end{array}\right|1\right),\label{eq:G_K_mix'}\tag{\ref{eq:G_K_mix}$_\Sigma$}\\\frac{\pi ^2}{8}-\frac{1}{2} \log ^2(\sqrt{2}-1)={}&\int_0^1\frac{\mathbf K(\sqrt{1-\kappa^{2}})}{\sqrt{1+\kappa^{2}}}\kappa\D \kappa=\sum_{n=0}^\infty\frac{(-2)^n}{2n+1}\frac{n!}{(2n+1)!!}={_3F_2}\left( \left.\begin{array}{c}
\frac{1}{2},1,1 \\[2pt]
\frac{3}{2},\frac{3}{2} \\
\end{array}\right|-1\right),\label{eq:K_tan_pi8'}\tag{\ref{eq:K_tan_pi8}$_\Sigma$}\\\frac{\pi ^2}{12}-\frac{1}{6} \log ^2(\sqrt{5}-2)={}&\int_0^1\frac{\mathbf K(\sqrt{1-\kappa^{2}})}{\sqrt{4+\kappa^{2}}}\kappa\D \kappa=\frac{1}{2}\sum_{n=0}^\infty\frac{(-\frac12)^n}{2n+1}\frac{n!}{(2n+1)!!}=\frac{1}{2}\,{_3F_2}\left( \left.\begin{array}{c}
\frac{1}{2},1,1 \\[2pt]
\frac{3}{2},\frac{3}{2} \\
\end{array}\right|-\frac{1}{4}\right),\label{eq:K_5_2'}\tag{\ref{eq:K_5_2}$_\Sigma$}\\\frac{1}{\pi}={}&\int_0^1\frac{\mathbf K(\sqrt{1-\kappa^2})}{\pi^{2}\kappa^{2}}\log\frac{1}{1-\kappa^{2}}\D\kappa=\frac{1}{4}\sum_{n=0}^\infty\frac{1}{n+1}\left[ \frac{(2n-1)!!}{2^{n}n!} \right]^2=\frac{1}{4}\,{_2F_1}\left( \left.\begin{array}{c}
\frac{1}{2},\frac{1}{2} \\
2 \\
\end{array}\right|1\right).\label{eq:Li2_limit_pi'}\tag{\ref{eq:Li2_limit_pi}$_\Sigma$}\end{align}}One may double-check these series representations by evaluating them  directly with the Wilf-Zeilberger  hypergeometric summation, as implemented in commercially available mathematical software packages, such as \textit{Mathematica}. Combining our series for $ 2\pi G-7\zeta(3)$ with the following result of Gosper \cite{Gosper}:
\[ 2\pi G-\frac{7\zeta(3)}{2}=\sum_{n=0}^\infty\frac{1}{n+1}\left[ \frac{2^{n}n!}{(2n+1)!!} \right]^2={_4F_3}\left( \left.\begin{array}{c}
1,1,1,1 \\
\frac{3}{2},\frac{3}{2},2 \\
\end{array}\right|1\right),\]
 we may deduce two more sums:
\begin{align}
 2\pi G=\sum _{n=0}^{\infty } \frac{2 (3n+2)}{(n+1) (2 n+1)}\left[ \frac{2^{n}n!}{(2n+1)!!} \right]^2,\quad \frac{7\zeta(3)}{2}=\sum _{m=0}^{\infty } \frac{4n+3}{(n+1) (2 n+1)}\left[ \frac{2^{n}n!}{(2n+1)!!} \right]^2.\label{eq:G_K_mix''}\tag{\ref{eq:G_K_mix}$^*_\Sigma$}
\end{align}Here, the two series representations of Catalan's constant \begin{align*}G=\frac{\pi}{4}\sum_{m=0}^\infty\frac{1}{2n+1}\left[ \frac{(2n-1)!!}{2^{n}n!} \right]^2=\frac{1}{\pi}\sum _{n=0}^{\infty } \frac{3n+2}{(n+1) (2 n+1)}\left[ \frac{2^{n}n!}{(2n+1)!!} \right]^2\end{align*}provide no numerical advantage over the expression $ G=\sum_{n=0}^\infty(-1)^n(2n+1)^{-2}$, as the summands all have $ O(1/n^2)$ decay rate.

We could improve the  convergence rate by turning other integral representations of $G$ into series.  It is  worth noting that Eq.~\ref{eq:G_Ti_b} translates into Ramanujan's series representation for Catalan's constant $G$ \cite{Ramanujan1915}:
\begin{align}
G= {}&\frac{\pi}{8}\log(2+\sqrt{3})+\frac{3}{8}\int_0^1\frac{\mathbf K(\sqrt{1-t})}{\sqrt{4-t}}\D t=\frac{\pi}{8}\log(2+\sqrt{3})+\frac{3}{8}\sum _{n=0}^{\infty }\frac{\left(\frac{1}{2}\right)^n}{2n+1} \frac{ n! }{(2 n+1)\text{!!}},\label{eq:G_Ti_b'}\tag{\ref{eq:G_Ti_b}$_\Sigma$}
\end{align}
while Eq.~\ref{eq:G_Ti_c} corresponds to Bradley's series acceleration formula \cite{Bradley1999}:\begin{align}G={}&\frac{\pi}{8}\log\frac{10+\sqrt{50-22\sqrt{5}}}{10-\sqrt{50-22\sqrt{5}}}+\frac{5}{8}\int_0^1\mathbf K(\sqrt{1-t})\left( \frac{1}{ \sqrt{\smash[b]{( \sqrt{5}-1)^{2}-t}}}-\frac{1}{\sqrt{\smash[b]{( \sqrt{5}+1)^{2}-t}}} \right)\D t\notag\\={}&\frac{\pi}{8}\log\frac{10+\sqrt{50-22\sqrt{5}}}{10-\sqrt{50-22\sqrt{5}}}+\frac{5}{4}\left[\frac{1}{ \sqrt{5}-1}\sum _{n=0}^{\infty }\left(\frac{\sqrt{2}}{\sqrt{5}-1}\right)^{2n}\frac{1}{2n+1} \frac{ n! }{(2 n+1)\text{!!}}-\frac{1}{ \sqrt{5}+1}\sum _{n=0}^{\infty } \left(\frac{\sqrt{2}}{\sqrt{5}+1}\right)^{2n}\frac{1}{2n+1} \frac{ n! }{(2 n+1)\text{!!}}\right].\label{eq:G_Ti_c'}\tag{\ref{eq:G_Ti_c}$_\Sigma$}\end{align}Both the series of Ramanujan and Bradley converge exponentially fast,  and thus have provided efficient algorithms to evaluate Catalan's constant numerically.

Making  wise use of combinatorial techniques, we may also verify Eq.~\ref{eq:Li2int'} in its most general form.
Suppose that $ |1-z^2|>2|z|$, we may use Taylor expansion to compute\begin{align*}\int_0^1\frac{(1-z^{4})\mathbf K(\sqrt{1-t})}{[(1-z^2)^{2}+4z^{2}t]\sqrt{(1-z^2)^{2}+4z^{2}t}}\D t={}&\frac{2(1-z^4)}{(1-z^2)^3}\sum _{n=0}^{\infty } \left(\frac{2z}{1-z^2}\right)^{2n}\frac{(-2)^{n}n! }{ (2 n+1)\text{!!}}.\end{align*}In the meantime, one can verify that \begin{align*}f_1(z)=\sum_{n=0}^{\infty } z^{2n}\frac{(-2)^{n}n! }{ (2 n+1)\text{!!}}\quad \text{and}\quad f_2(z)=\frac{\sinh^{-1} z}{z\sqrt{1+z^2}}\end{align*}both satisfy the differential equation\[z(1+z^2)\frac{\D f(z)}{\D z}+(1+2z^{2})f(z)=1\]with the same initial conditions $ f_1(0)=f_2(0)=1$.
Thus we have \begin{align*}\int_0^1\frac{(1-z^{4})\mathbf K(\sqrt{1-t})}{[(1-z^2)^{2}+4z^{2}t]\sqrt{(1-z^2)^{2}+4z^{2}t}}\D t=\frac{2(1-z^4)}{(1-z^2)^3}\frac{\sinh^{-1} \frac{2z}{1-z^2}}{\frac{2z}{1-z^2}\sqrt{1+(\frac{2z}{1-z^2})^2}}=\frac{1}{z}\log\frac{1+z}{1-z},\quad \text{for }|1-z^2|>2|z|,\end{align*}which extends to the open unit disk $ |z|<1$ by analytic continuation.
One may compare  this  mechanical derivation to   the geometric motivations that  led to  Eq.~\ref{eq:Li2int}.

Using the moment expansion, we may express the function $ M(r)$ (Eq.~\ref{eq:GB_5}) in terms of generalized hypergeometric series:\begin{align}M(r)={}&\int_0^1\frac{r\xi\mathbf K(\xi)\D\xi}{1+r^{2}\xi^2}=\int_0^1\frac{r\kappa\mathbf K(\sqrt{1-\kappa^{2}})\D\kappa}{1+r^{2}-r^{2}\kappa^2}=\frac{r}{1+r^{2}}\sum_{n=0 }^\infty\int_0^1\kappa^{2n+1}\left( \frac{r^2}{1+r^2} \right)^n\mathbf K\left(\sqrt{1-\kappa^2}\right)\D\kappa\notag\\={}&\frac{r}{1+r^{2}}\sum_{n=0 }^\infty\left( \frac{r^2}{1+r^2} \right)^n\frac{\pi}{4}\left[ \frac{\Gamma(n+1)}{\Gamma(n+\frac{3}{2})} \right]^2=\frac{r}{{1+r^{2}}}\,{_3F_2}\left( \left.\begin{array}{c}
1,1,1 \\
\frac{3}{2},\frac{3}{2} \\
\end{array}\right|\frac{r^2}{1+r^2}\right),\label{eq:M_3F2}\tag{\ref{eq:GB_5}$ _\Sigma$}\end{align}which is also a special case of    2.16.5.1 in \cite{PBMVol3}.  By the relation  \[W(r)+M(r)=\frac{\pi}{2}\frac{1}{\sqrt{1+r^{2}}}\mathbf K\left( \frac{r}{\sqrt{1+r^2}} \right),\]we  have \begin{align}W(r)={}&\int_0^1\frac{r\mathbf K(\xi)\D\xi}{r^{2}+\xi^2}=\frac{\pi}{2}\frac{1}{\sqrt{1+r^{2}}}\mathbf K\left( \frac{r}{\sqrt{1+r^2}} \right)-\frac{r}{{1+r^{2}}}\,{_3F_2}\left( \left.\begin{array}{c}
1,1,1 \\
\frac{3}{2},\frac{3}{2} \\
\end{array}\right|\frac{r^2}{1+r^2}\right),\label{eq:W_3F2}\tag{\ref{eq:GB_1}$_\Sigma $}\end{align}which may be compared to a particular case of  2.16.5.10 in \cite{PBMVol3}.
 As we analytically continue the functional equation for $ W(r)$ (Eq.~\ref{eq:f_eq_W}) to $ r=-i\sin\theta,0<\theta<\pi/2$, we obtain a hypergeometric identity:\begin{align}&\frac{i\sin\theta}{\cos^2\theta}\,{_3F_2}\left( \left.\begin{array}{c}
1,1,1 \\
\frac{3}{2},\frac{3}{2} \\
\end{array}\right|-\tan^2\theta\right)+\frac{\cos\theta}{\sin^2\theta}\,{_3F_2}\left( \left.\begin{array}{c}
1,1,1 \\
\frac{3}{2},\frac{3}{2} \\
\end{array}\right|-\cot^2\theta\right)+\left(\frac{i}{{\sin\theta}}+\frac{1}{\cos\theta}\right)\,{_3F_2}\left( \left.\begin{array}{c}
1,1,1 \\
\frac{3}{2},\frac{3}{2} \\
\end{array}\right|\frac{1}{1-e^{4i\theta}}\right)\notag\\={}&-\frac{\pi}{2}\left[\mathbf K(\sin\theta)+i\mathbf K(\cos\theta)-4ie^{-i\theta}\mathbf K(e^{-2i\theta})\right]=\frac{\pi[\mathbf K(\sin\theta)+i\mathbf K(\cos\theta)]}{2},\quad 0<\theta<\frac{\pi}{2}.\label{eq:3F2_sum_rule}\tag{\ref{eq:f_eq_W}$ _\Sigma$}\end{align}Here, in the last step, we have used Landen's transformation and the inverse modulus transformation to deduce\begin{align*}ie^{-i\theta}\mathbf K(e^{-2i\theta})=ie^{-i\theta}\mathbf K\left( \frac{1-i\tan\theta}{1+i\tan\theta} \right)=\frac{i\mathbf K(\sec\theta)}{2\cos\theta}=\frac{i}{2}[\mathbf K(\cos\theta)-i\mathbf K(\sin\theta)].\end{align*}Setting $ \theta=\pi/4$ in  Eq.~\ref{eq:3F2_sum_rule}, and employing the special value \begin{align*}\mathbf K\left( \frac{1}{\sqrt{2}} \right)=\frac{[\Gamma(\frac{1}{4})]^{2}}{4\sqrt{\pi}},\end{align*} we obtain the relation\begin{align}{_3F_2}\left( \left.\begin{array}{c}
1,1,1 \\
\frac{3}{2},\frac{3}{2} \\
\end{array}\right|-1\right)+{_3F_2}\left( \left.\begin{array}{c}
1,1,1 \\
\frac{3}{2},\frac{3}{2} \\
\end{array}\right|\frac{1}{2}\right)=\frac{\pi}{2\sqrt{2}}\mathbf K\left( \frac{1}{\sqrt{2}} \right)=\frac{\sqrt{\pi}[\Gamma(\frac{1}{4})]^{2}}{8\sqrt{2}}.\label{eq:3F2_sum_rule_pi4}\tag{\ref{eq:3F2_sum_rule}, $ \theta=\frac{\pi}{4}$}\end{align}Thanks to the Chowla-Selberg theory \cite{ChowlaSelberg,SelbergChowla,Zucker1977,JoyceZucker1991} which leads to the special values\[\mathbf K\left( \sin\frac{\pi}{12}\right)=\frac{1}{\sqrt{3}}\mathbf K\left( \cos\frac{\pi}{12}\right)=\frac{3^{1/4}}{2^{7/3}}\frac{[\Gamma(\frac{1}{3})]^{3}}{\pi},\] the case of $ \theta=\pi/12$  for Eq.~\ref{eq:3F2_sum_rule} also results in an exact evaluation\begin{align}&(5+3\sqrt{3})\,{_3F_2}\left( \left.\begin{array}{c}
1,1,1 \\
\frac{3}{2},\frac{3}{2} \\
\end{array}\right|-(2+\sqrt{3})^{2}\right)-i(5-3\sqrt{3})\,{_3F_2}\left( \left.\begin{array}{c}
1,1,1 \\
\frac{3}{2},\frac{3}{2} \\
\end{array}\right|-(2-\sqrt{3})^{2}\right)+(1+i)(\sqrt{3}+i)\,{_3F_2}\left(\vphantom{\frac12} \smash{\left.\begin{array}{c}
1,1,1 \\
\frac{3}{2},\frac{3}{2} \\
\end{array}\right|\frac{1}{2}+\frac{i\sqrt{3}}{2}}\right)\notag\\={}&\frac{\pi e^{i\pi/3}}{\sqrt{2}}\mathbf K\left( \sin\frac{\pi}{12}\right)=e^{i\pi/3}\frac{3^{1/4}[\Gamma(\frac{1}{3})]^{3}}{2^{17/6}}.\label{eq:3F2_sum_rule_pi12}\tag{\ref{eq:3F2_sum_rule}, $ \theta=\frac{\pi}{12}$}\end{align}Similarly, setting $ \theta=5\pi/12$ would bring us\begin{align}&-i(5+3\sqrt{3})\,{_3F_2}\left( \left.\begin{array}{c}
1,1,1 \\
\frac{3}{2},\frac{3}{2}
\end{array}\right|-(2+\sqrt{3})^{2}\right)+(5-3\sqrt{3})\,{_3F_2}\left( \left.\begin{array}{c}
1,1,1 \\
\frac{3}{2},\frac{3}{2} \\
\end{array}\right|-(2-\sqrt{3})^{2}\right)-(1+i)(\sqrt{3}-i)\,{_3F_2}\left(\vphantom{\frac12} \smash{\left.\begin{array}{c}
1,1,1 \\
\frac{3}{2},\frac{3}{2} \\
\end{array}\right|\frac{1}{2}-\frac{i\sqrt{3}}{2}}\right)\notag\\={}&-\frac{\pi e^{i\pi/6}}{\sqrt{2}}\mathbf K\left( \sin\frac{\pi}{12}\right)=-e^{i\pi/6}\frac{3^{1/4}[\Gamma(\frac{1}{3})]^{3}}{2^{17/6}}.\label{eq:3F2_sum_rule_5pi12}\tag{\ref{eq:3F2_sum_rule}, $ \theta=\frac{5\pi}{12}$}\end{align}

It would be unfair to go without mentioning the following summation formula in Ramanujan's notebooks \cite[Ref.][p.~153]{RN3}:\begin{align*}\sum_{n=0}^\infty\frac{1}{(2n+1)^2\cosh\frac{(2n+1)\pi z}{2i}}=\frac{\pi\sqrt{\lambda(z)}}{4\mathbf K(\sqrt{\lambda(z)})}{_3F_2}\left( \left.\begin{array}{c}
1,1,1 \\
\frac{3}{2},\frac{3}{2} \\
\end{array}\right|\lambda(z)\right),\quad -1<\R z\leq 1,\left\vert z+\frac{1}{2} \right\vert\geq\frac{1}{2},\left\vert z-\frac{1}{2} \right\vert>\frac{1}{2},\I z>0\end{align*}where the prescribed domain of validity  ensures that $z\mathbf K(\sqrt{\lambda(z)})=i\mathbf K(\sqrt{1-\lambda(z)}) $, with $ \lambda(z)$ being the elliptic modular lambda function. Using the imaginary modulus transformation for complete elliptic integrals, one can verify that  $ \lambda(z)$ satisfies the following transformation laws\begin{align}\begin{array}{r@{\;=\;}lr@{\;=\;}lr@{\;=\;}l}
\lambda(z+1)&\dfrac{\lambda(z)}{\lambda(z)-1},\quad & \lambda\left(-\dfrac{1}{z}\right)&1-\lambda(z),\quad &\lambda(z+2)&\lambda(z),\\[8pt]\lambda\left( -\dfrac{1}{z+1} \right)&\dfrac{1}{1-\lambda(z)},& \lambda\left( \dfrac{z\mathstrut}{z+1} \right)&\dfrac{1}{\lambda(z)},& \lambda\left( 1-\dfrac{1}{z} \right)&1-\dfrac{1}{\lambda(z)},
\end{array}\label{eq:lambda_transf}
\end{align}in the upper half-plane where $ \I z>0$. Meanwhile,  Landen's transformation for complete elliptic integrals
entails the following duplication formula for the elliptic modular lambda function (cf.~\cite[Ref.][\S135]{WeberVol3}):\begin{align}\label{eq:double_half_lambda}\lambda(2z)=\left[ \frac{1-\sqrt{1-\lambda(z)}}{1+\sqrt{1-\lambda(z)}} \right]^2,\quad -1<\R z< 1,\left\vert z+\frac{1}{2} \right\vert>\frac{1}{2},\left\vert z-\frac{1}{2} \right\vert>\frac{1}{2},\I z>0.\end{align}

 With the correspondence between trigonometric factors and elliptic modular lambda functions (cf.~Eqs.~\ref{eq:lambda_transf} and \ref{eq:double_half_lambda})\begin{align*}& \sin^2\theta=\lambda(z),\quad \cos^2\theta =\lambda \left( -\frac{1}{z} \right),\quad -\tan^2\theta =\frac{\lambda(z)}{\lambda(z)-1}=\lambda(z+1),\quad -\cot^2\theta=1-\frac{1}{\lambda(z)}=\lambda\left( 1-\frac{1}{z} \right),\notag\\& e^{-4i\theta}=\left(\frac{1-i\tan\theta}{1+i\tan\theta}\right)^{2}=\left[\frac{1-\sqrt{\lambda(z+1)}}{1+\sqrt{\lambda(z+1)}}\right]^{2}=\lambda\left( -\frac{2}{1+z} \right),\quad \frac{1}{1-e^{4i\theta}}=\frac{\lambda\left( -\frac{2}{1+z} \right)}{\lambda\left( -\frac{2}{1+z} \right)-1}=\lambda\left( \frac{z-1}{z+1} \right),\quad\text{where }z/i>0,\end{align*}and the related modular transformations for complete elliptic integrals of the first kind \begin{align*}&\frac{i\sqrt{\lambda(z)}}{1-{\lambda(z)}}\frac{\mathbf K(\sqrt{\lambda(z+1)})}{\sqrt{\lambda(z+1)}}=\frac{i\sqrt{\lambda(z)}}{1-{\lambda(z)}}\frac{1}{\sqrt{\lambda(z+1)}}\mathbf K\left( \sqrt{\frac{\lambda(z)}{\lambda(z)-1}} \right)=\frac{i\sqrt{\lambda(z)}}{1-{\lambda(z)}}\frac{1}{\sqrt{\lambda(z+1)}}\sqrt{1-\lambda(z)}\mathbf K(\sqrt{\lambda(z)})=\mathbf K(\sqrt{\lambda(z)}),\notag\\&\frac{\sqrt{1-\lambda(z)}}{\lambda(z)} \frac{\mathbf K\left( \sqrt{\lambda\left( 1-\frac{1}{z} \right)} \right)}{ \sqrt{\lambda\left( 1-\frac{1}{z} \right)}}=-\frac{\sqrt{1-\lambda(z)}}{\lambda(z)}\frac{iz\sqrt{\lambda(z)}}{\sqrt{\lambda\left( 1-\frac{1}{z} \right)}}\mathbf K(\sqrt{\lambda(z)})=-z \mathbf K(\sqrt{\lambda(z)}),\notag\\&\left(\frac{i}{{\sqrt{\lambda(z)}}}+\frac{1}{\sqrt{1-\lambda(z)}}\right)\frac{\mathbf K\left( \sqrt{\lambda\left( \frac{z-1}{z+1} \right)} \right)}{ \sqrt{\lambda\left( \frac{z-1}{z+1} \right)}}=2(1+z) \mathbf K(\sqrt{\lambda(z)}),\quad\text{where } z/i>0,\end{align*} we can rewrite Eq.~\ref{eq:3F2_sum_rule}
using Ramanujan's summation formula: \begin{align}&\sum_{n=0}^\infty\frac{i(-1)^{n}}{(2n+1)^2\sinh\frac{(2n+1)\pi y}{2}}+y\sum_{n=0}^\infty\frac{(-1)^{n}}{(2n+1)^2\sinh\frac{(2n+1)\pi }{2y}}+2(1+iy)\sum_{n=0}^\infty\frac{1}{(2n+1)^2\cosh\frac{(2n+1)\pi(iy-1) }{2i(iy+1)}}=\frac{\pi^{2}(1+iy)}{8},\quad \forall y>0.\label{eq:3F2_sum_rule_R}\tag{\ref{eq:f_eq_W}$ _{\Sigma_R}$}\end{align}The special cases of Eq.~\ref{eq:3F2_sum_rule} with  $ \theta=\pi/4,\pi/12,5\pi/12$ can be converted into some hyperbolic series of Ramanujan type:\begin{align}
\sum_{n=0}^\infty\frac{1}{(2n+1)^2}\left[ \frac{(-1)^n}{\sinh\frac{(2n+1)\pi}{2}} +\frac{2}{\cosh\frac{(2n+1)\pi}{2}}\right]={}&\frac{\pi^2}{8},\label{eq:3F2_sum_rule_R_1}\tag{\ref{eq:3F2_sum_rule_R}, $ y=1$}\\
\sum_{n=0}^\infty\frac{1}{(2n+1)^2}\left[ \frac{(-1)^n\sqrt{3}}{\sinh\frac{\vphantom{\sqrt1}(2n+1)\pi}{2\sqrt{3}}} +\frac{(-1)^n{i}}{\sinh\frac{(2n+1)\sqrt{3}\pi}{2}}+\frac{2(1+i\sqrt{3})}{\cos\frac{(2n+1)(1+i\sqrt{3})\pi}{4}}\right]={}&\frac{\pi^2}{8}(1+i\sqrt{3}),\label{eq:3F2_sum_rule_R_2}\tag{\ref{eq:3F2_sum_rule_R}, $ y=\sqrt{3}$}\\
\sum_{n=0}^\infty\frac{1}{(2n+1)^2}\left[ \frac{(-1)^n}{\sqrt{3}\sinh\frac{\vphantom{\sqrt1}(2n+1)\sqrt{3}\pi}{2}} +\frac{(-1)^n{i}}{\sinh\frac{\vphantom{\sqrt1}(2n+1)\pi}{2\sqrt{3}}}+\frac{2(1+\frac{i}{\sqrt{3}})}{\cos\frac{(2n+1)(1-i\sqrt{3})\pi}{4}}\right]={}&\frac{\pi^2}{8}\left( 1+\frac{i}{\sqrt{3}} \right),\label{eq:3F2_sum_rule_R_3}\tag{\ref{eq:3F2_sum_rule_R}, $ y=1/\sqrt{3}$}
\end{align}although these formulae did not seem to show up in  Ramanujan's  works. By analytic continuation and complex conjugation, one can further generalize Eq.~\ref{eq:3F2_sum_rule_R} into \begin{align}
\sum_{n=0}^\infty\frac{1}{(2n+1)^2(1+z)}\left[ \frac{1}{\cosh\frac{(2n+1)\pi(z+1)}{2i}} -\frac{z}{\cosh\frac{(2n+1)\pi(z-1)}{2iz}}\right]+2\sum_{n=0}^\infty\frac{1}{(2n+1)^2\cosh\frac{(2n+1)\pi(z-1)}{2i(z+1)}}
=\frac{\pi^2}{8},\quad \forall z\in\mathbb C\smallsetminus\mathbb R.\label{eq:3F2_sum_rule_cplx}\tag{\ref{eq:f_eq_W}$ ^*_{\Sigma_R}$}\end{align}The left-hand side of  Eq.~\ref{eq:3F2_sum_rule_cplx}also has  well-defined limits when $ z$ tends to certain rational numbers. For example, the case where $ z\to 1$  recovers the familiar series  $ \sum_{n=0}^\infty1/(2n+1)^2=\pi^2/8$.

The moment expansions involving the product of two elliptic integrals also paraphrase some of our integral formulae into hypergeometric identities. For instance, we may rewrite Eq.~\ref{eq:K_sqr_minus_pi4} as\begin{align}\frac{\pi^2}{2}\log2-\frac{7\zeta(3)}{4}={}&\int_0^1\left\{ \left[\mathbf K\left(\sqrt{1-\kappa^2}\right)\right]^2-\frac{\pi^2}{4} \right\}\frac{\kappa\D\kappa}{1-\kappa^2}=\sum_{n=0}^{\infty}\int_0^1\left\{ \left[\mathbf K\left(\sqrt{1-\kappa^2}\right)\right]^2-\frac{\pi^2}{4} \right\}\kappa^{2n+1}\D\kappa\notag\\={}&\sum_{n=0}^{\infty}\left\{ \frac{2^{4n+1}(n!)^4(n+1)}{[(2n+1)\text{!!}]^4}\,{_7F_6}\left( \left.\begin{array}{c}
\frac{1}{2},\frac{1}{2},\frac{1}{2},\frac12,n+1,n+1,\frac{n+3}{2} \\[4pt]
1,\frac{n+1}{2},n+\frac{3}{2},n+\frac{3}{2} ,n+\frac{3}{2},n+\frac{3}{2}\\
\end{array}\right|1\right)-\frac{\pi^2}{8(n+1)} \right\},\tag{\ref{eq:K_sqr_minus_pi4}$_\Sigma$}\label{eq:K_sqr_minus_pi4'}\end{align}after invoking the expression for $ \int_0^1[\mathbf K(\sqrt{1-\kappa^2})]^2\kappa^{2n+1}\D\kappa$ from  Eq.~18 in \cite{Wan2012}. Similarly, according to the formula for  $ \int_0^1\mathbf K(x)\mathbf K(\sqrt{1-x^2})x^{2n+1}\D x$ from  Eq.~18 in \cite{Wan2012}, our   Eq.~\ref{eq:KKlogt_zeta3} can be rephrased as\begin{align}\frac{\pi^2}{2}\log2-\frac{7\zeta(3)}{4}={}&\frac1\pi\int_0^1\mathbf K(x)\mathbf K\left(\sqrt{1-x^2}\right)x\log\frac{1}{1-x^{2}}\D x=\frac{1}{\pi}\sum_{n=1}^\infty\int_0^1\mathbf K(x)\mathbf K\left(\sqrt{1-x^2}\right)\frac{x^{2n+1}}{n}\D x\notag\\={}&\sum_{n=1}^\infty\frac{2^{2n-1} (n!)^2}{[(2n+1)\text{!!]}^2n}\,{_4F_3}\left( \left.\begin{array}{c}
\frac{1}{2},\frac{1}{2},n+1,n+1 \\[2pt]
1,n+\frac{3}{2},n+\frac{3}{2} \\
\end{array}\right|1\right).\tag{\ref{eq:KKlogt_zeta3}$_\Sigma$}\label{eq:KKlogt_zeta3'}\end{align}Interestingly, every summand in the last line is actually a rational multiple of $ \pi^2$ \cite[cf.~Ref.][Lemma~1]{Wan2012}.  This statement can be proved by using the third-order ordinary differential equation\footnote{As one may recognize, this third-order differential equation just paraphrases \cite[Ref.][Eq.~16]{Zhou2013Pnu} in the case where $ \nu=-1/2$. The use of such a Legendre-Appell differential equation  dispenses the combinatorial argument  in the proof of Lemma~1 in~\cite{Wan2012}.}\begin{align*}\frac{\D }{\D t}\left\{t (1-t) \frac{\D}{\D t}\left[t (1-t) \frac{\D f(t)}{\D t}\right]- t (1-t) f(t)\right\}=\frac{2t-1}{2}  f(t),\quad 0<t<1\end{align*}satisfied by $ f(t)=\mathbf K(\sqrt{t})\mathbf K(\sqrt{1-t})$, which leads to the recurrence relations\begin{align*}&2\int_0^1\mathbf K(x)\mathbf K\left(\sqrt{1-x^2}\right)x^{2n+1}\D x=\int_{0}^1\mathbf K(\sqrt{t})\mathbf K(\sqrt{1-t})t^n\D t\notag\\={}&2\int_{0}^1\mathbf K(\sqrt{t})\mathbf K(\sqrt{1-t})t^{n+1}\D t-2\int_{0}^1t^{n}\frac{\D }{\D t}\left\{t (1-t) \frac{\D}{\D t}\left[t (1-t) \frac{\D }{\D t}\right]- t (1-t) \right\}[\mathbf K(\sqrt{t})\mathbf K(\sqrt{1-t})]\D t\notag\\={}&2\int_{0}^1\mathbf K(\sqrt{t})\mathbf K(\sqrt{1-t})[n^3t^{n-1}-n(2 n^2+3 n+2)t^{n}+(n+1)^3 t^{n+1}]\D t,\quad n\geq1;\\&\int_{0}^1\mathbf K(\sqrt{t})\mathbf K(\sqrt{1-t})\D t=2\int_{0}^1\mathbf K(\sqrt{t})\mathbf K(\sqrt{1-t})t\D t,\end{align*}after integration by parts. Here, we have $ \int_{0}^1\mathbf K(\sqrt{t})\mathbf K(\sqrt{1-t})\D t=\pi^{3}/8$ according to Eq.~\ref{eq:K_sqr_star} and the recursion
relations outlined above reveal the rationality of  $ \pi^{-3}\int_{0}^1\mathbf K(\sqrt{t})\mathbf K(\sqrt{1-t})t^n\D t$ for all positive integers $n$.

\newpage
\appendix
\section{Index of Multiple Elliptic Integrals and Related Series\label{app:ind_formulae}} \subsection{Brief Lists of Identities for  $G$, $ \zeta(3)$ and $\pi$}
We compile here all the formulae leading to Catalan's constant $G$, Ap\'ery's constant $ \zeta(3)$, and the circle constant $ \pi$, which had been derived somewhere in the main text.
We keep track of the original equation numbers, followed by contextual annotations if necessary. Some closely related representations are (re)grouped together, so the enumeration does not strictly follow   the original order of    appearance.

The versatility of integral and series representations for Catalan's constant has been well recognized, as early as in Catalan's own treatises \cite{Catalan1865,Catalan1883}. As we have biased towards results that  have geometric interpretations and/or have connections to elliptic integrals, the list below is by no means an exhaustive catalog of identities for Catalan's constant $G$. In fact, we will forgo many definite integrals of elementary functions giving rise to $G$, which  can be looked up in well-indexed references such as \cite{GradshteynRyzhik}. Furthermore, a few formulae gathered below have been known long before the current work, and our ``geometric proofs'' for them do not necessarily represent the most straightforward approach. \vspace{1em}\noindent\hrule\vspace{.2em}\hrule {\allowdisplaybreaks\begin{align}
G={}&\int_{S^2}\frac{\D \sigma_1}{4\pi}\int_{S^2}\frac{\D \sigma_{2}}{4\pi}\frac{1}{|\bm n-\bm n_1|}\frac{1}{|\bm n_{1}-\bm n_2|}\frac{1}{|\bm n_{2}+\bm n|}=\int_0^1\D u\int_0^1\D v\int_0^{1}\D w\frac{1}{8\pi\sqrt{\mathstrut  u(1-u)(1-u+uw)}\sqrt{\vphantom(  1-vw}\sqrt{\mathstrut w(1-v)}}\tag{\ref{eq:b}}\\={}&\frac{1}{2\pi}\int_0^{1}\frac{\mathbf{K}(\sqrt{1-\xi^2})}{\xi}\log\frac{1+\xi}{1-\xi}\D \xi=\frac{1}{\pi}\int_0^{1}\frac{\mathbf{K}(k)}{1-k}\log\frac{1}{k}\D k\tag{\ref{eq:b}, before and after Landen's transformation}\\={}&\int_{0}^1\frac{\mathbf K(\sqrt{1-\kappa^2})}{3 \pi ^2   (1-\kappa ^2)}\left[\pi^{2}-3 (1+\kappa ) \text{Li}_2(\kappa )-3 (1-\kappa) \text{Li}_2(-\kappa )\right]\D\kappa\tag{\ref{eq:b_B_R}}\\={}&\frac{1}{2\pi}\left\{\int_0^{1}\frac{[\mathbf K(k)]^{2}\D k}{\sqrt{1-k^{2}}}-2\int_{0}^1\mathbf K(k)\mathbf K\left(\sqrt{1-k^2}\right)\D k\right\}+\frac{1}{2\pi}\int_0^{1}\frac{\mathbf K(\sqrt{1-\xi^{2}})}{\xi}\log\frac{1+\xi}{1-\xi}\D \xi\tag{\ref{eq:b_star}}\\={}&\frac{\pi}{4}\sum_{m=0}^\infty\frac{1}{2m+1}\left[ \frac{(2m-1)!!}{2^{m}m!} \right]^2=\frac{\pi}{4}\,{_3F_2}\left( \left.\begin{array}{c}
\frac{1}{2},\frac{1}{2},\frac{1}{2} \\[2pt]
1,\frac{3}{2} \\
\end{array}\right|1\right)\tag{\ref{eq:b_series_exp}}\\={}&\int_{S^2}\frac{\D \sigma_1}{4\pi}\int_{S^2}\frac{\D \sigma_{2}}{4\pi}\frac{1}{|\bm n-\bm n_1|}\frac{\theta_{H}(\cos\beta-\bm n_1\cdot\bm n_2)}{\sqrt{2(\cos\beta-\bm n_1\cdot\bm n_2)}}\frac{1}{|\bm n_{2}+\bm n|}-\frac{1}{4}\int_0^\beta\log\tan\frac{\pi-\alpha}{4}\D\alpha,\quad 0\leq\beta<\pi\tag{\ref{eq:beta'}}\\={}&\int_0^1\D q\int_0^1\D r\int_{0}^1\D s\frac{1}{8\pi\sqrt{\smash[b]{(1-r)(1-s)rs(1-q)}\vphantom{s^2}}{\sqrt{\smash[b]{\sec^2(\beta/2)-s[qr+\tan^2(\beta/2)]}}}}-\frac{1}{4}\int_0^\beta\log\tan\frac{\pi-\alpha}{4}\D\alpha,\quad 0\leq\beta<\pi\tag{\ref{eq:beta'}}\notag\\
={}&\frac{\cos\frac{\beta}{2}}{2}\int_0^1\mathbf K\left( \sqrt{k^2\cos^2\frac{\beta}{2}+\sin^2\frac{\beta}{2} }\right)\D k-\frac{1}{4}\int_0^\beta\log\tan\frac{\pi-\alpha}{4}\D\alpha,\quad 0\leq\beta<\pi\tag{\ref{eq:beta'}}\\={}&\frac{1}{\pi}\int_{0}^1\frac{\mathbf K(\sqrt{1-\kappa^2})}{\kappa\sqrt{\smash[b]{1-\kappa^{2}\sin^2(\beta/2)}}}\tanh^{-1}\frac{\kappa\cos(\beta/2)}{\sqrt{\smash[b]{1-\kappa^{2}\sin^2(\beta/2)}}}\D\kappa-\frac{1}{4}\int_0^\beta\log\tan\frac{\pi-\alpha}{4}\D\alpha,\quad 0\leq\beta<\pi\tag{\ref{eq:beta_B}}\\={}&-\left(\cos\frac{\beta}{2}\right)\int_0^1\frac{(1+t^2)\log t}{1+2t^{2}\cos\beta+t^4}\D t-\frac{1}{4}\int_0^\beta\log\tan\frac{\pi-\alpha}{4}\D\alpha,\quad 0\leq\beta<\pi\tag{\ref{eq:beta'}}\\={}&\int_0^1\D q\int_0^1\D r\int_{0}^1\D s\frac{1}{8\pi\sqrt{(1-r)(1-s)rs(1-q)}{\sqrt{1-qrs}}}=\frac{1}{2}\int_0^1\mathbf K(k)\D k\tag{\ref{eq:beta'}, with $ \beta=0$}\\={}&\frac{3}{8}\int_0^1\mathbf K\left( \frac{\sqrt{k^2+3} }{2}\right)\D k=\frac{3}{4}\int_{0}^1\frac{\mathbf K(\sqrt{1-\kappa^2})}{1-\kappa^{2}}\left( 1-\frac{2\kappa}{\sqrt{\smash[b]{1+3\kappa^{2}}}} \right)\D\kappa\tag{\ref{eq:beta'} and \ref{eq:beta'_diamond}, with $ \beta=\frac{2\pi}3$ }\\={}&\frac{5}{16}\left[ (\sqrt{5}+1)\int_0^1\mathbf K\left( \frac{\sqrt{5}+1}{4}\sqrt{k^2+5-2\sqrt{5}} \right)\D k-(\sqrt{5}-1)\int_0^1\mathbf K\left( \frac{\sqrt{5}-1}{4}\sqrt{k^2+5+2\sqrt{5}} \right)\D k\right]\tag{\ref{eq:beta'}, with $ \beta=\frac{2\pi}5,\frac{4\pi}{5}$}\\={}&\frac{5}{4}\int_{0}^1\frac{\mathbf K(\sqrt{1-\kappa^2})}{1-\kappa^{2}}\left[ \frac{(\sqrt{5}+1)\kappa}{\sqrt{\smash[b]{1+(5+2\sqrt{5})\kappa^{2}}}} -\frac{(\sqrt{5}-1)\kappa}{\sqrt{\smash[b]{1+(5-2\sqrt{5})\kappa^{2}}}}\right]\D\kappa\tag{\ref{eq:beta'_diamond}, with $ \beta=\frac{2\pi}5,\frac{4\pi}{5}$}\\={}&\frac12\int_0^{\pi/2}\frac{\theta}{\sin\theta}\D\theta=\frac{\cos(\beta/2)}{2}\int_0^{\pi/2}\frac{ \theta\D \theta}{\sin\theta\sqrt{\smash[b]{1-\cos^2\theta\sin^2(\beta/2)}}}-\frac{1}{4}\int_0^\beta\log\tan\frac{\pi-\alpha}{4}\D\alpha,\quad 0\leq \beta<\pi\tag{\ref{eq:G_elem}$ ^\circ$, \ref{eq:G_elem}}\\={}&\frac{7\zeta(3)}{4\pi}+\frac{\pi}{8}-\int_0^1\frac{\mathbf K(k)}{2k}\left[ 1-\frac{2}{\pi} \mathbf E(k)\right]\D k\tag{\ref{eq:KE_pi_G_zeta3}}\\={}&\int_{S^3}\frac{\D \sigma_1}{2\pi^{2}}\int_{S^3}\frac{\D \sigma_2}{2\pi^{2}}\frac{1}{|\bm n-\bm n_1|^2}\frac{1}{|\bm n_{1}-\bm n_2|^{2}}\frac{1}{|\bm n_{2}-\bm n_\bot|^2}=\frac{1}{4\pi}\int_{0}^\pi(\pi-\psi)\log\frac{1+\sin\psi}{1-\sin\psi}\D\psi=\frac{1}{4 }\int_{0}^{\pi/2}\log\frac{1+\sin\psi}{1-\sin\psi}\D\psi\tag{\ref{eq:G_S3}}\\={}&\frac{\pi}{4}-\frac{1}{2}\log2-\frac{1}{4}\int_0^1 \mathbf K( \sqrt{1-t} )\log \frac{\sqrt{1-t}(1-\sqrt{1-t})}{ t}\D t=\frac{\pi}{4}-\frac{1}{2}\log2-\frac{1}{2\pi}\int_0^1 \frac{\mathbf K( \sqrt{1-\kappa^{2}})}{\kappa^2} \log (1-\kappa)\log(1+\kappa)\D \kappa\tag{\ref{eq:G_Klog}, \ref{eq:G_Klog'}}\\={}&\frac{1}{\pi}\int_0^{\pi/2}\D\theta\int_0^{\pi/2}\D\phi\frac{\sin\theta(1-\cos^4\theta)\sin\phi\cos\phi\mathbf K(\sin\theta)\mathbf K(\sin\phi)}{(\sin^4\theta+4\cos^2\theta\cos^2\phi)^{3/2}}\tag{\ref{eq:G_KK}}\\={}&\frac{1}{2}\int_0^{\pi/2}\D\theta\int_0^{\pi/2}\D\phi\frac{\cos\theta(1-\cos^4\theta)\sin\phi\cos\phi\mathbf K(\sin\phi)}{(\sin^4\theta+4\cos^2\theta\cos^2\phi)^{3/2}}\tag{\ref{eq:G_K}}\\={}&\frac{7\zeta(3)}{2\pi}-\frac{1}{\pi}\int_0^1\frac{\mathbf K (\eta)}{1+\eta}\log\frac{1}{\eta}\D \eta=\frac{1}{\pi}\sum _{m=0}^{\infty } \frac{3 m+2}{(m+1) (2 m+1)}\left[ \frac{2^{m}m!}{(2m+1)!!} \right]^2=\frac{7\zeta(3)}{4\pi}-\frac{1}{\pi}\int_0^{\pi/2}\theta\log\tan\frac{\theta}{2}\D\theta\tag{\ref{eq:G_K_mix}, \ref{eq:G_K_mix''}, \ref{eq:logtan_G_zeta3}}\\={}&\frac{\pi}{6}\log2+\frac{1}{3}\int_0^1\mathbf K(\sqrt{1-t})\left( \frac{2}{\sqrt{25-16 t}}+\frac{3}{4 \sqrt{25-9 t}}+\frac{3}{\sqrt{625-576 t}} \right)\D t\tag{\ref{eq:G_Ti_a}}\\= {}&\frac{\pi}{8}\log(2+\sqrt{3})+\frac{3}{8}\int_0^1\frac{\mathbf K(\sqrt{1-t})}{\sqrt{4-t}}\D t=\frac{\pi}{8}\log(2+\sqrt{3})+\frac{3}{8}\sum _{n=0}^{\infty }\frac{\left(\frac{1}{2}\right)^n}{2n+1} \frac{ n! }{(2 n+1)\text{!!}}\tag{\ref{eq:G_Ti_b}, \ref{eq:G_Ti_b'}}\\={}&\frac{\pi}{8}\log\frac{10+\sqrt{50-22\sqrt{5}}}{10-\sqrt{50-22\sqrt{5}}}+\frac{5}{8}\int_0^1\mathbf K(\sqrt{1-t})\left( \frac{1}{ \sqrt{6-2 \sqrt{5}-t}}-\frac{1}{\sqrt{6+2 \sqrt{5}-t}} \right)\D t\tag{\ref{eq:G_Ti_c}}\\={}&\frac{\pi}{8}\log\frac{10+\sqrt{50-22\sqrt{5}}}{10-\sqrt{50-22\sqrt{5}}}+\frac{5}{4}\left[\frac{1}{ \sqrt{5}-1}\sum _{n=0}^{\infty }\left(\frac{\sqrt{2}}{\sqrt{5}-1}\right)^{2n}\frac{1}{2n+1} \frac{ n! }{(2 n+1)\text{!!}}-\frac{1}{ \sqrt{5}+1}\sum _{n=0}^{\infty } \left(\frac{\sqrt{2}}{\sqrt{5}+1}\right)^{2n}\frac{1}{2n+1} \frac{ n! }{(2 n+1)\text{!!}}\right]\tag{\ref{eq:G_Ti_c'}}\\={}&\frac{\pi}{2}\log2+\frac{1}{\pi}\int_0^1\mathbf K\left(\sqrt{1-\kappa^2}\right)\log\sqrt{1-\kappa^2}\D\kappa=\frac{\pi}{2}\log2-\int_0^1\left[ \mathbf K(k)-\frac{\pi}{2}\right]\frac{\D k}{2k}\tag{\ref{eq:G_log2_log}, \ref{eq:low_deg'}}\\={}&\frac{7\zeta(3)}{2\pi}-\frac{1}{2\pi^2}\int_0^{\pi/2}\mathbf K(\sin\theta)\left[ \frac{\dilog(e^{2i\theta})-\dilog(e^{-2i\theta})}{i\cos\theta} \right]\D\theta\tag{\ref{eq:7Apery_Li2}}\\={}&\frac{1}{\pi}\int_0^1\frac{k\tanh^{-1}\sqrt{1-k^2}}{1-k^{2}}\mathbf K(k)\D k=\frac12\int_0^{1}\frac{y\D y}{\cosh y}=\frac{1}2\int_0^\infty\frac{\log(r+\sqrt{1+r^2})\D r}{1+r^{2}}\tag{\ref{eq:G_tanh}, \ref{eq:G_cosh}}\\={}&\frac{1}{2}\int_0^1\frac{\mathbf K(\sqrt{1-\kappa^2})}{1+\kappa}\D\kappa=\frac{1}{\pi}\int_0^1\mathbf K\left(\sqrt{1-\kappa^2}\right)\left[ \frac{\kappa\arccos\kappa}{\sqrt{1-\kappa^2}}-\log(2\kappa) \right]\D\kappa\tag{\ref{eq:G_odd}}\\={}&\frac{1}{\pi}\int_0^1\left( \frac{ \kappa\arcsin\kappa }{\sqrt{1-\kappa ^2} }+\log \sqrt{4-4\kappa^{2}}  \right)\mathbf K\left(\sqrt{1-\kappa^2}\right)\D\kappa=\frac{1}{\pi}\int_0^1\log \frac{\sqrt{1-\kappa^{2}}}{\kappa}\mathbf K\left(\sqrt{1-\kappa^2}\right)\D\kappa.\tag{\ref{eq:GB_G}}\end{align}}\noindent\hrule\vspace{.2em}\hrule\vspace{1em}

Continuing the list above, we bring together some integral  and series representations of Ap\'ery's constant $ \zeta(3)$, which do not explicitly involve Catalan's constant $G$.

\vspace{1em}

\noindent\hrule\vspace{.2em}\hrule\vspace{.2em}
{\allowdisplaybreaks\begin{align}\zeta(3)={}&\frac{4}{7}\int_{S^2}\frac{\D \sigma_1}{4\pi}\int_{S^2}\frac{\D \sigma_{2}}{4\pi}\frac{1}{|\bm n-\bm n_1|}\mathbf K\left( \sqrt{\frac{1-\bm n_1\cdot\bm n_2}{2}} \right)\frac{1}{|\bm n_{2}+\bm n|}\tag{\ref{eq:L_beta_int}}\\={}&\frac{2}{7\pi}\iint_{0\leq\theta_1\leq\pi,0\leq\theta_2\leq\pi,0\leq\theta_1+\theta_2\leq\pi}\mathbf K\left( \sin\frac{\theta_{1}}{2} \right)\mathbf K\left( \sin\frac{\theta_{2}}{2} \right)\sin\frac{\theta_1+\theta_2}{2}\D\theta_1\D\theta_2\tag{\ref{eq:L_beta_int'}}\\
={}&\int_0^1\D q\int_0^1\D r\int_{0}^1\D s\frac{1}{7\pi s\sqrt{\smash[b]{(1-r)(1-s)r(1-q)(1-qr)}}}\arctan\sqrt{\frac{(1-qr)s}{1-s}}\tag{\ref{eq:L_beta_int}}\\
={}&\frac{4}{7}\int_0^1\log\frac{1-t}{1+t}\log t\frac{\D t}{t}=\frac{4}{7}\int_0^{\pi/2}\frac{ \theta}{\sin\theta}\frac{\frac{\pi}{2}-\theta}{\cos\theta}\D \theta\tag{\ref{eq:L_beta_int}}\\={}&\frac{2\pi^{2}}{7}\log2-\frac{4}{7}\int_{0}^{1}\left\{[\mathbf K(k)]^{2}-\frac{\pi^2}{4}\right\}\frac{ \D k}{k}\tag{\ref{eq:K_sqr_minus_pi4}}\\={}&\frac{2\pi^{2}}{7}\log2-\frac{4}{7}\sum_{n=0}^{\infty}\left\{ \frac{2^{4n+1}(n!)^4(n+1)}{[(2n+1)\text{!!}]^4}\,{_7F_6}\left( \left.\begin{array}{c}
\frac{1}{2},\frac{1}{2},\frac{1}{2},\frac12,n+1,n+1,\frac{n+3}{2} \\[4pt]
1,\frac{n+1}{2},n+\frac{3}{2},n+\frac{3}{2} ,n+\frac{3}{2},n+\frac{3}{2}\\
\end{array}\right|1\right)-\frac{\pi^2}{8(n+1)} \right\}\tag{\ref{eq:K_sqr_minus_pi4'}}\notag\\={}&\frac{2}{7}\int_0^1[\mathbf K(\sqrt{1-t})]^2\D t=\frac{2}{7}\int_0^1[\mathbf K(\sqrt{t})]^2\D t\tag{\ref{eq:Kt2_7Zeta3}}\\={}&\frac{2\pi^{2}}{7}\log2+\frac{4}{7\pi}\int_0^1\mathbf K(\sqrt{t})\mathbf K(\sqrt{1-t})\log t\D t=\frac{2\pi^{2}}{7}\log2-\frac{4}{7}\sum_{n=1}^\infty\frac{2^{2n-1} (n!)^2}{[(2n+1)\text{!!}]^2n}\,{_4F_3}\left( \left.\begin{array}{c}
\frac{1}{2},\frac{1}{2},n+1,n+1 \\[2pt]
1,n+\frac{3}{2},n+\frac{3}{2} \\
\end{array}\right|1\right)\tag{\ref{eq:KKlogt_zeta3}, \ref{eq:KKlogt_zeta3'}}\\={}&\frac{\pi^2}{7}\left[ 2\log2-\sum_{n=1}^\infty\frac{a_n}{n} \right],\quad \text{where }\begin{cases}a_{0}=\dfrac14,\;a_1=\dfrac{1}{8} &  \\[6pt]
a_{n}=2[n^3a_{n-1}-n(2 n^2+3 n+2)a_{n}+(n+1)^3 a_{n+1}], & \forall n\geq1\tag{ \ref{eq:KKlogt_zeta3'}} \\
\end{cases}\\={}&\frac27\sum _{m=0}^{\infty } \frac{4n+3}{(n+1) (2 n+1)}\left[ \frac{2^{n}n!}{(2n+1)!!} \right]^2\tag{\ref{eq:G_K_mix''}}\\={}&\frac27\int_{0}^1\frac{\mathbf K(\sqrt{1-\kappa^2})[ \pi-2\kappa\mathbf K(\sqrt{1-\kappa^2})]}{1-\kappa^{2}}\D\kappa.\tag{\ref{eq:zeta3_iB_int}}\end{align}}
\vspace{.2em}
\noindent\hrule\vspace{.2em}\hrule\vspace{1em}

The last part of this subsection contains a collection of formulae that evaluate to  $ \pi^2/8$, up to  an integer  power of $\pi$  times a rational number.\vspace{1em} \noindent\hrule\vspace{.2em}\hrule\vspace{.5em}
{\allowdisplaybreaks\begin{align}\frac{\pi^2}{8}={}&\int_{S^2}\frac{\D \sigma_1}{4\pi}\int_{S^2}\frac{\D \sigma_{2}}{4\pi}\frac{1}{|\bm n-\bm n_1|}\frac{1}{|\bm n_{1}-\bm n_2|}\frac{1}{|\bm n_{2}-\bm n|}=\int_0^1\D u\int_0^1\D v\int_0^{1}\D w\frac{1}{8\pi\sqrt{\mathstrut u(1-u)(1-u+uw)}\sqrt{\vphantom(  vw(1-v)}}\tag{\ref{eq:a}}\\={}&\frac{1}{4}\int_{0}^{\pi/2}\frac{\log\frac{1+\cos\theta}{1-\cos\theta}\D \theta}{ \cos\theta}=-\int_{0}^{\pi/2}\frac{\log\sin\theta\D \theta}{ \cos\theta}=-\int_{0}^{\pi/2}\frac{\log\tan\frac{\theta}{2}\D \theta}{ 2\cos\theta}=\int_{0}^{1}\frac{\log t\D t}{ t^{2}-1}\tag{\ref{eq:a_star}}\\={}&\int_0^{1}\frac{\mathbf{K}(k)}{1+k}\D k=\int_{0}^1\frac{\mathbf K(\sqrt{1-\kappa^2}) }{\pi  (\kappa ^2-1)}\left( \log \frac{1-\kappa ^2}{4}+\kappa  \log\frac{1+\kappa}{1-\kappa} \right)\D\kappa=\frac{1}{\pi}\int_{0}^{\pi/2}\mathbf K(\sin\theta)\frac{\pi  \sin \theta -2 \theta}{ \cos \theta}  \D\theta\tag{\ref{eq:a}, \ref{eq:a_B}, \ref{eq:a_iB}}\notag\\={}&\frac{1}{2}\int_0^1\mathbf K\left( \sqrt{1-\kappa^2} \right)\D\kappa=\int_0^1\frac{\mathbf K(\sqrt{1-\kappa^2})}{\pi\kappa\sqrt{1-\kappa^{2}}}\arcsin\kappa\D\kappa=-\int_0^1\frac{\mathbf K(\sqrt{1-\kappa^2})}{\pi(1-\kappa^{2}){}}\log\kappa\D\kappa=\frac{1}{\pi}\int_0^1\frac{\arccos k}{1-k^2}\mathbf K(k)\D k\tag{\ref{eq:a}, \ref{eq:a'_L}, \ref{eq:a_diamond_L}, \ref{eq:pi8_arccos} }\\={}&\int_0^\pi\frac{\pi-\beta}{4}\D\beta=\frac{2}{\pi}\int_{S^2}\frac{\D \sigma_1}{4\pi}\int_{S^2}\frac{\D \sigma_{2}}{4\pi}\frac{1}{|\bm n-\bm n_1|}\mathbf K\left( \sqrt{\frac{1-\bm n_1\cdot\bm n_2}{2}} \right)\frac{1}{|\bm n_{2}-\bm n|}=\frac{1 }{4}\int_{0}^{1}\log\frac{1+\sqrt{\smash[b]{q}}}{1-\sqrt{\smash[b]{q}}}\frac{\D q}{q}\tag{\ref{eq:pipibeta_int}}\\={}&\frac{1}{\pi^2}\iint_{0\leq\theta_1\leq\pi,0\leq\theta_2\leq\pi,0\leq\theta_1+\theta_2\leq\pi}\mathbf K\left( \sin\frac{\theta_{1}}{2} \right)\mathbf K\left( \sin\frac{\theta_{2}}{2} \right)\cos\frac{\theta_1-\theta_2}{2}\D\theta_1\D\theta_2\tag{\ref{eq:pipibeta_int'}}\\={}&\int_0^\pi\frac{\pi^{2}(\pi-\beta)}{32}\cos\frac{\beta}{2}\D\beta=\frac{\pi^{2}}{8}\int_{S^2}\frac{\D \sigma_1}{4\pi}\int_{S^2}\frac{\D \sigma_{2}}{4\pi}\frac{1}{|\bm n-\bm n_1|} \frac{1}{|\bm n_{2}-\bm n|}=\frac{\pi^{2} }{16}\int_{0}^{1}[\mathbf K(\sqrt{\smash[b]{q}})-\mathbf E(\sqrt{\smash[b]{q}})]\frac{\D q}{q^{3/2}}\tag{\ref{eq:pipibeta_int_cos_half_beta}}\\={}&\frac{\pi^2}8\int_{S^2}\frac{\D \sigma_1}{4\pi}\int_{S^2}\frac{\D \sigma_{2}}{4\pi}\frac{1}{|\bm n-\bm n_1|}\frac{1}{|\bm n_{2}+\bm n|}=\int_0^1\D q\int_0^1\D r\int_{0}^1\D s\frac{(1-qrs)\mathbf E\left( \sqrt{\frac{(1-qr)s}{1-qrs}} \right)-(1-s)\mathbf K\left( \sqrt{\frac{(1-qr)s}{1-qrs}} \right)}{16 (1-qr)s\sqrt{\smash[b]{(1-r)(1-s)rs(1-q)(1-qrs)}\vphantom{s^2}}}\tag{\ref{eq:L_beta_int_cos_half_beta}}\\
={}&-\frac{\pi^{2}}{8}\int_0^1\log t\D t=\frac{\pi}{4}\int_0^{\pi/2}\left[ \frac{\mathbf E(\cos\theta)}{\sin\theta} -\sin\theta\mathbf K(\cos\theta)\right]\frac{ \theta\D \theta}{\cos^2\theta}\tag{\ref{eq:L_beta_int_cos_half_beta}}\\={}&\frac{\pi}{8}\int_0^\pi\left(\sin\frac{\alpha }{2}\right)\log\cot\frac{\pi-\alpha}{4}\D\alpha=\frac{\pi^{2}}{8}\int_0^1\mathbf K(k)k\D k=\int_0^1\frac{\mathbf K(k)}{k}[\mathbf K(k)-\mathbf E(k)]\D k\tag{\ref{eq:unity_int}, \ref{eq:KEK_pi8}}\\={}&\frac{\log^2(\sqrt{2}-1)}{2}+\int_0^1\frac{\mathbf K(\sqrt{1-t})}{2\sqrt{1+t}}\D t=\frac{\log^2(\sqrt{2}-1)}{2}+\sum_{n=0}^\infty\frac{(-2)^n}{2n+1}\frac{n!}{(2n+1)!!}=\frac{\log^2(\sqrt{2}-1)}{2}+{_3F_2}\left( \left.\begin{array}{c}
\frac{1}{2},1,1 \\[2pt]
\frac{3}{2},\frac{3}{2} \\
\end{array}\right|-1\right)\tag{\ref{eq:K_tan_pi8}, \ref{eq:K_tan_pi8'}}\\={}&\frac{9}{8} \log ^2\left(\frac{\sqrt{5}-1}{2}\right)+\frac{3}{4}\int_0^1\frac{\mathbf K(\sqrt{1-t})}{\sqrt{1+4t}}\D t\tag{\ref{eq:K_gr}}\\={}&\frac{\log ^2(\sqrt{5}-2)}{4} +\frac{3}{4}\int_0^1\frac{\mathbf K(\sqrt{1-t})}{\sqrt{4+t}}\D t=\frac{\log ^2(\sqrt{5}-2)}{4} +\frac{3}{4}\sum_{n=0}^\infty\frac{(-\frac12)^n}{2n+1}\frac{n!}{(2n+1)!!}=\frac{\log ^2(\sqrt{5}-2)}{4} +\frac{3}{4}\,{_3F_2}\left( \left.\begin{array}{c}
\frac{1}{2},1,1 \\[2pt]
\frac{3}{2},\frac{3}{2} \\
\end{array}\right|-\frac{1}{4}\right)\tag{\ref{eq:K_5_2}, \ref{eq:K_5_2'}}\\={}&\frac{1}{2}\int_0^{\pi}\mathbf K(\sin\beta)\sin\beta\cos\frac{\beta}{2}\D\beta=-\frac{2^{(n-5)/2}(n-1)^2}{n}\left[ \frac{(\frac{n+1}{4})!}{(\frac{n-1}{2})!!} \right]^2\int_0^{\pi}\mathbf K(\sin\beta)\sin\beta\cos\frac{n\beta}{2}\D\beta,\quad n>0,\;n\equiv3\bmod4\tag{\ref{eq:MD_proj}}\\={}&\frac12\int_0^{1}\frac{y\D y}{\sinh y}=\frac{1}2\int_0^\infty\frac{\log(r+\sqrt{1+r^2})\D r}{r\sqrt{1+r^{2}}}=\frac{1}{\pi}\int_0^\infty\frac{\arctan\frac{1}{x}}{\sqrt{1+x^2}}\mathbf K\left(\frac{x}{\sqrt{1+x^2}}\right )\D x\tag{\ref{eq:pi8_y_sinhy}, \ref{eq:pi8_arctan_inf_int}}\\={}&\frac\pi8\int_0^1\frac{\mathbf K(\sqrt{1-\kappa^2})}{\kappa^2}\log\frac{1}{1-\kappa^2}\D\kappa=\frac{\pi^{3}}{32}\sum_{n=0}^\infty\frac{1}{n+1}\left[ \frac{(2n-1)!!}{2^{n}n!} \right]^2=\frac{\pi^{3}}{32}\,{_2F_1}\left( \left.\begin{array}{c}
\frac{1}{2},\frac{1}{2} \\
2 \\
\end{array}\right|1\right)\tag{\ref{eq:Li2_limit_pi}, \ref{eq:Li2_limit_pi'}}\\={}&\frac{\pi}{4}\int_{0}^1\frac{\mathbf K(\sqrt{1-\kappa^2})}{1-\kappa^{2}}\left( 1-\frac{\kappa\arccos\kappa}{\sqrt{1-\kappa^2}} \right)\D\kappa\tag{\ref{eq:pi_iB_int}; \textit{also} \ref{eq:pi_Li2_deriv_iB}}\\={}&\sum_{n=0}^\infty\frac{(-1)^{n}}{(2n+1)^2(1+iy)}\left[\frac{i}{\sinh\frac{(2n+1)\pi y}{2}}+\frac{y}{\sinh\frac{(2n+1)\pi }{2y}}\right]+2\sum_{n=0}^\infty\frac{1}{(2n+1)^2\cosh\frac{(2n+1)\pi(iy-1) }{2i(iy+1)}},\quad \forall y>0\tag{\ref{eq:3F2_sum_rule_R}}\\={}&\sum_{n=0}^\infty\frac{1}{(2n+1)^2}\left[ \frac{(-1)^n}{\sinh\frac{(2n+1)\pi}{2}} +\frac{2}{\cosh\frac{(2n+1)\pi}{2}}\right]\tag{\ref{eq:3F2_sum_rule_R_1}}\\
={}&\frac{1}{1+i\sqrt{3}}\sum_{n=0}^\infty\frac{1}{(2n+1)^2}\left[ \frac{(-1)^n\sqrt{3}}{\sinh\frac{\vphantom{\sqrt1}(2n+1)\pi}{2\sqrt{3}}} +\frac{(-1)^n{i}}{\sinh\frac{(2n+1)\sqrt{3}\pi}{2}}+\frac{2(1+i\sqrt{3})}{\cos\frac{(2n+1)(1+i\sqrt{3})\pi}{4}}\right]\tag{\ref{eq:3F2_sum_rule_R_2}}\\={}&\frac{1}{1+\frac{i}{\sqrt{3}}}
\sum_{n=0}^\infty\frac{1}{(2n+1)^2}\left[ \frac{(-1)^n}{\sqrt{3}\sinh\frac{\vphantom{\sqrt1}(2n+1)\sqrt{3}\pi}{2}} +\frac{(-1)^n{i}}{\sinh\frac{\vphantom{\sqrt1}(2n+1)\pi}{2\sqrt{3}}}+\frac{2(1+\frac{i}{\sqrt{3}})}{\cos\frac{(2n+1)(1-i\sqrt{3})\pi}{4}}\right]\tag{\ref{eq:3F2_sum_rule_R_3}}
\\={}&\sum_{n=0}^\infty\frac{1}{(2n+1)^2(1+z)}\left[ \frac{1}{\cosh\frac{(2n+1)\pi(z+1)}{2i}} -\frac{z}{\cosh\frac{(2n+1)\pi(z-1)}{2iz}}\right]+2\sum_{n=0}^\infty\frac{1}{(2n+1)^2\cosh\frac{(2n+1)\pi(z-1)}{2i(z+1)}}
,\quad \forall z\in\mathbb C\smallsetminus\mathbb R.\tag{\ref{eq:3F2_sum_rule_cplx}}\end{align}}\noindent\hrule\vspace{.2em}\hrule\vspace{.5em}
\subsection{Miscellaneous Integrals and Series Concerning  Elliptic Integrals}In this subsection, the results will be divided into two categories: constants and functions. The arrangement of equations within each category roughly follows their original order of appearance.

We first present a set of identities that invoke special values of surds, logarithms, elliptic integrals, and  (generalized) hypergeometric series.

\vspace{1em}

\noindent\noindent\hrule\vspace{.2em}\hrule\vspace{.5em}\noindent\begin{align}(2-\sqrt{2})\pi={}&\int_{S^2}\frac{\D \sigma_1}{4\pi}\int_{S^2}\frac{\D \sigma_{2}}{4\pi}\frac{1}{|\bm n-\bm n_1|} \frac{\mathbf K\left( \sqrt{\frac{2\sqrt{\smash[b]{(1-\bm n_1\cdot\bm n_2)/2}}}{1+\sqrt{\smash[b]{(1-\bm n_1\cdot\bm n_2)/2}}}} \right)}{\sqrt{\smash[b]{1+\sqrt{\smash[b]{(1-\bm n_1\cdot\bm n_2)/2}}}}}\frac{1}{|\bm n_{2}-\bm n|}=\frac{\sqrt{2}\pi}{3}\int_0^1{_3F_2\left( \left.\begin{array}{c}
\frac{1}{2},1 ,\frac{3}{2}  \\[4pt]
\frac{5}{4},\frac{7}{4}
\end{array}\right| t^{2}\right)}{\D t}\tag{\ref{eq:pipibeta_int_cos_quarter_beta}}\\={}&\frac{1}{2\pi}\iint_{0\leq\theta_1\leq\pi,0\leq\theta_2\leq\pi,\theta_1+\theta_2\leq\pi} \frac{\mathbf K\left( \sqrt{\frac{2\sin(\theta_{1}/2)}{1+\sin(\theta_{1}/2)}} \right)\mathbf K\left( \sqrt{\frac{2\sin(\theta_{2}/2)}{1+\sin(\theta_{2}/2)}} \right)}{\sqrt{\smash[b]{1+\sin(\theta_{1}/2)}}\sqrt{\smash[b]{1+\sin(\theta_{2}/2)}}}\cos\frac{\theta_1-\theta_2}{2}\D\theta_1\D\theta_2.
\tag{\ref{eq:pipibeta_int_cos_quarter_beta'}}
\end{align}\hrule
{\allowdisplaybreaks
\begin{align}2\sqrt{2}\log2={}&\int_{S^2}\frac{\D \sigma_1}{4\pi}\int_{S^2}\frac{\D \sigma_{2}}{4\pi}\frac{1}{|\bm n-\bm n_1|} \frac{\mathbf K\left( \sqrt{\frac{2\sqrt{\smash[b]{(1-\bm n_1\cdot\bm n_2)/2}}}{1+\sqrt{\smash[b]{(1-\bm n_1\cdot\bm n_2)/2}}}} \right)}{\sqrt{\smash[b]{1+\sqrt{\smash[b]{(1-\bm n_1\cdot\bm n_2)/2}}}}}\frac{1}{|\bm n_{2}+\bm n|}\tag{\ref{eq:L_beta_int_cos_quarter_beta}}\\={}&\frac{\sqrt2}{6}\int_0^1\D q\int_0^1\D r\int_{0}^1\D s\frac{1}{\pi(1- s)\sqrt{\smash[b]{(1-r)rs(1-q)}}}\, {_3F_2\left( \left.\begin{array}{c}
\frac{1}{2},1 ,\frac{3}{2}  \\[4pt]
\frac{5}{4},\frac{7}{4}
\end{array}\right| -\frac{(1-qr)s}{1-s}\right)} \tag{\ref{eq:L_beta_int_cos_quarter_beta}}\\={}&e^{i\pi/4}\int_0^1\left[ (t+i) \tan ^{-1}\frac{(1+i) \sqrt{t}}{t-i}-(t-i) \tanh ^{-1}\frac{(1+i) \sqrt{t}}{t+i} \right]\frac{\log t\D t}{2 t^{3/2}}\tag{\ref{eq:L_beta_int_cos_quarter_beta}}\\={}&\frac{1}{2\pi}\iint_{0\leq\theta_1\leq\pi,0\leq\theta_2\leq\pi,\theta_1+\theta_2\leq\pi} \frac{\mathbf K\left( \sqrt{\frac{2\sin(\theta_{1}/2)}{1+\sin(\theta_{1}/2)}} \right)\mathbf K\left( \sqrt{\frac{2\sin(\theta_{2}/2)}{1+\sin(\theta_{2}/2)}} \right)}{\sqrt{\smash[b]{1+\sin(\theta_{1}/2)}}\sqrt{\smash[b]{1+\sin(\theta_{2}/2)}}}\sin\frac{\theta_1+\theta_2}{2}\D\theta_1\D\theta_2.\tag{\ref{eq:L_beta_int_cos_quarter_beta'}}\end{align}}\hrule
\begin{align}
\frac{\pi^{3}}{4}\,{_4F_3\left( \left.\begin{array}{c}
\frac{1}{2},\frac{1}{2} ,\frac{1}{2},\frac{1}{2}  \\[2pt]
1,1,1
\end{array}\right| 1\right)}={}&\int_0^1\frac{[\mathbf K(k)]^2\D k}{\sqrt{1-k^2}}=2\int_0^{1}\mathbf K(k)\mathbf K\left( \sqrt{1-k^{2}} \right)\D k.\tag{\ref{eq:KKKK}}\end{align}\noindent\hrule\begin{align}&\frac{\pi^{4}}{32}\,{_7F_6\left( \left.\begin{array}{c}
\frac{1}{2},\frac{1}{2} ,\frac{1}{2},\frac{1}{2}  ,\frac{1}{2},\frac{1}{2},\frac54\\[4pt]\frac14,1,1,
1,1,1
\end{array}\right| 1\right)}=\frac{1}{\pi}\int_0^1\frac{[\mathbf K(k)]^2}{\sqrt{1-k^2}}\log\frac{k}{\sqrt{1-k^2}}\D k=\int_0^1[\mathbf K(k)]^2\D k=\frac{1}{2}\int_0^1\left[\mathbf K\left(\sqrt{1-\kappa^2}\right)\right]^2\D \kappa\tag{\ref{eq:K_sqr_log}}\\={}&\frac{2}{\pi}\int_0^1\mathbf K(\xi)\mathbf K\left( \sqrt{1-\xi^2} \right)\log\frac{1+\xi}{1-\xi}\D\xi=\frac{2}{\pi}\int_0^1\mathbf K(k)\mathbf K\left( \sqrt{1-k^2} \right)\log\frac{1}{k}\D k=\frac{1}{\pi}\int_0^1\frac{\mathbf K(k)\mathbf K(\sqrt{1-k^{2}})}{\sqrt{1-k^{2}}}\arccos(1-2k^2)\D k.\tag{\ref{eq:K_sqr_log_chain}}
\end{align}\hrule\begin{align}\frac{\sqrt{\pi}[\Gamma(\frac{1}{4})]^{2}}{8\sqrt{2}}=\frac{\pi}{2\sqrt{2}}\mathbf K\left( \frac{1}{\sqrt{2}} \right)={_3F_2}\left( \left.\begin{array}{c}
1,1,1 \\
\frac{3}{2},\frac{3}{2} \\
\end{array}\right|-1\right)+{_3F_2}\left( \left.\begin{array}{c}
1,1,1 \\
\frac{3}{2},\frac{3}{2} \\
\end{array}\right|\frac{1}{2}\right).\tag{\ref{eq:3F2_sum_rule_pi4}}\end{align}\noindent\hrule{\allowdisplaybreaks\begin{align}\frac{3^{1/4}[\Gamma(\frac{1}{3})]^{3}}{2^{17/6}}=\frac{\pi }{\sqrt{2}}\mathbf K\left( \sin\frac{\pi}{12}\right)={}&e^{-i\pi/3}(5+3\sqrt{3})\,{_3F_2}\left( \left.\begin{array}{c}
1,1,1 \\
\frac{3}{2},\frac{3}{2} \\
\end{array}\right|-(2+\sqrt{3})^{2}\right)-ie^{-i\pi/3}(5-3\sqrt{3})\,{_3F_2}\left( \left.\begin{array}{c}
1,1,1 \\
\frac{3}{2},\frac{3}{2} \\
\end{array}\right|-(2-\sqrt{3})^{2}\right)\notag\\&+(1+i)(\sqrt{3}-i)\,{_3F_2}\left(\vphantom{\frac12} \smash{\left.\begin{array}{c}
1,1,1 \\
\frac{3}{2},\frac{3}{2} \\
\end{array}\right|\frac{1}{2}+\frac{i\sqrt{3}}{2}}\right)\tag{\ref{eq:3F2_sum_rule_pi12}}\notag\\={}&ie^{-i\pi/6}(5+3\sqrt{3})\,{_3F_2}\left( \left.\begin{array}{c}
1,1,1 \\
\frac{3}{2},\frac{3}{2} \\
\end{array}\right|-(2+\sqrt{3})^{2}\right)-e^{-i\pi/6}(5-3\sqrt{3})\,{_3F_2}\left( \left.\begin{array}{c}
1,1,1 \\
\frac{3}{2},\frac{3}{2} \\
\end{array}\right|-(2-\sqrt{3})^{2}\right)\notag\\&+(1+i)(1-i\sqrt{3})\,{_3F_2}\left(\vphantom{\frac12} \smash{\left.\begin{array}{c}
1,1,1 \\
\frac{3}{2},\frac{3}{2} \\
\end{array}\right|\frac{1}{2}-\frac{i\sqrt{3}}{2}}\right)\tag{\ref{eq:3F2_sum_rule_5pi12}}\\={}&\frac{5+3\sqrt{3}}{2}\,{_3F_2}\left( \left.\begin{array}{c}
1,1,1 \\
\frac{3}{2},\frac{3}{2} \\
\end{array}\right|-(2+\sqrt{3})^{2}\right)-\frac{\sqrt{3}(5-3\sqrt{3})}{2}\,{_3F_2}\left( \left.\begin{array}{c}
1,1,1 \\
\frac{3}{2},\frac{3}{2} \\
\end{array}\right|-(2-\sqrt{3})^{2}\right)\notag\\&+\R\left[ (1+i)(\sqrt{3}-i)\,{_3F_2}\left(\vphantom{\frac12} \smash{\left.\begin{array}{c}
1,1,1 \\
\frac{3}{2},\frac{3}{2} \\
\end{array}\right|\frac{1}{2}+\frac{i\sqrt{3}}{2}}\right)\right ].\tag{average of the last two formulae}\end{align}}\noindent\hrule\begin{align}&\frac{\sqrt{3}(5+3\sqrt{3})}{2}\,{_3F_2}\left( \left.\begin{array}{c}
1,1,1 \\
\frac{3}{2},\frac{3}{2} \\
\end{array}\right|-(2+\sqrt{3})^{2}\right)+\frac{5-3\sqrt{3}}{2}\,{_3F_2}\left( \left.\begin{array}{c}
1,1,1 \\
\frac{3}{2},\frac{3}{2} \\
\end{array}\right|-(2-\sqrt{3})^{2}\right)\notag\\={}&\I\left[ (1+i)(\sqrt{3}-i)\,{_3F_2}\left(\vphantom{\frac12} \smash{\left.\begin{array}{c}
1,1,1 \\
\frac{3}{2},\frac{3}{2} \\
\end{array}\right|\frac{1}{2}+\frac{i\sqrt{3}}{2}}\right)\right ].\tag{difference between (\ref{eq:3F2_sum_rule_pi12}) and (\ref{eq:3F2_sum_rule_5pi12})}\end{align}\noindent\hrule\vspace{.2em}\hrule\vspace{1em}
Lastly, we provide a collation of integrals and series that evaluate to elementary and special functions. \vspace{1em}\noindent\hrule\vspace{.2em}\hrule\vspace{.5em} \begin{align}\frac14\log\tan\frac{\pi-\beta}{4}={}&-\frac{\sin(\beta/2)}{2\pi}\int_{0}^1\frac{\mathbf K(\sqrt{1-\kappa^2})}{1-\kappa^{2}\sin^2(\beta/2)}\left[1-\frac{\kappa\cos(\beta/2)}{\sqrt{\smash[b]{1-\kappa^{2}\sin^2(\beta/2)}}}\tanh^{-1}\frac{\kappa\cos(\beta/2)}{\sqrt{\smash[b]{1-\kappa^{2}\sin^2(\beta/2)}}}\right]\D\kappa\tag{\ref{eq:beta'_B}}\\={}&-\frac{\sin(\beta/2)}{4}\mathbf K\left( \sin\frac{\beta}{2} \right)+\frac{\sin(\beta/2)}{2\pi}\int_{0}^1\frac{\kappa\cos(\beta/2)\mathbf K(\sqrt{1-\kappa^2})}{[1-\kappa^{2}\sin^2(\beta/2)]^{3/2}}\tanh^{-1}\frac{\kappa\cos(\beta/2)}{\sqrt{\smash[b]{1-\kappa^{2}\sin^2(\beta/2)}}}\D\kappa\tag{\ref{eq:beta'_B}}
\\={}&-\frac{\sin\beta}{8}\int_{0}^1\frac{\kappa\mathbf K(\sqrt{1-\kappa^2})}{[\cos^2(\beta/2)+\kappa^{2}\sin^2(\beta/2)]^{3/2}}\D\kappa,\quad 0\leq\beta<\pi.
\tag{\ref{eq:beta_iB_deriv}}\end{align}\noindent\hrule\begin{align}\dilog (z)-\dilog (-z)={}&\int_0^1\frac{ z\mathbf K( \sqrt{1-t} )}{\sqrt{(1-z^{2})^{2}+4z^{2}t}}\D t\tag{\ref{eq:Li2int}}\\={}&\int_0^1\frac{4z\mathbf  K(\sqrt{1-\kappa ^2}) }{\pi  \kappa  \sqrt{(1+z ^2)^2 \kappa ^2-4 z ^2}}\left[ \log\frac{1}{\sqrt{1-\kappa^{2}}} +\log\frac{\sqrt{(1+z^{2})^2 \kappa ^2-4 z ^2}+(1-z ^2) \kappa}{\sqrt{(1+z^{2})^2 \kappa ^2-4 z ^2}+(1+z ^2) \kappa}\right]\D\kappa,\quad |z|\leq1;\tag{\ref{eq:Li2int_B}}\\\dilog (z)-\dilog (-z)={}&-\frac{2}{\pi}\int_0^1\frac{\mathbf  K(\sqrt{1-\kappa ^2}) }{  1-\kappa ^2}\left[ \log\frac{1-z}{1+z}+\frac{2z\kappa}{\sqrt{(1+z^2)^2 \kappa ^2-(1-z^2)^2}}\tanh^{-1}\frac{\sqrt{(1+z^2)^2 \kappa ^2-(1-z^2)^2}}{(1+z^2) \kappa}\right]\D\kappa,\notag\\&\text{for } |z|\leq 1\text{ and }z^2\neq1.\tag{\ref{eq:Li2int_diamond}}\end{align}\noindent\hrule{\allowdisplaybreaks\begin{align}&\frac{1}{z}\log\frac{1+z}{1-z}=\int_0^1\frac{ (1-z^{4} )\mathbf K( \sqrt{1-t} )}{[(1-z^{2})^{2}+4z^{2}t]\sqrt{(1-z^{2})^{2}+4z^{2}t}}\D t\tag{\ref{eq:Li2int'}}\\={}&\frac{4(1+z^{2})}{\pi}\int_{0}^1\frac{\mathbf K(\sqrt{1-\kappa^2})}{(1+z^{2})^{2}-4z^{2}\kappa^{2}}\left[1-\frac{\kappa(1-z^{2})}{\sqrt{(1+z^{2})^{2}-4z^{2}\kappa^{2}}}\tanh^{-1}\frac{\kappa(1-z^{2})}{\sqrt{(1+z^{2})^{2}-4z^{2}\kappa^{2}}}\right]\D\kappa\tag{\ref{eq:beta'_star_B}}\\={}&\frac{2}{1+z^{2}}\mathbf K\left( \frac{2z}{1+z^2} \right)-\frac{4(1+z^{2})}{\pi}\int_{0}^1\frac{\kappa(1-z^{2})\mathbf K(\sqrt{1-\kappa^2})}{[(1+z^{2})^{2}-4z^{2}\kappa^{2}]\sqrt{(1+z^{2})^{2}-4z^{2}\kappa^{2}}}\tanh^{-1}\frac{\kappa(1-z^{2})}{\sqrt{(1+z^{2})^{2}-4z^{2}\kappa^{2}}}\D\kappa\tag{\ref{eq:beta'_star_B}}\\={}&-\frac{4(1-z^2)}{\pi}\int_0^1\frac{\mathbf  K(\sqrt{1-\kappa ^2}) }{  (1+z^2)^2 \kappa ^2-(1-z^2)^2}\left[ 1-\frac{(1+z^{2})\kappa}{\sqrt{(1+z^2)^2 \kappa ^2-(1-z^2)^2}} \tanh^{-1}\frac{\sqrt{(1+z^2)^2 \kappa ^2-(1-z^2)^2}}{(1+z^2) \kappa}\right]\D\kappa,\quad |z|<1;\tag{\ref{eq:Li2int'_iB}}\\&\frac{1}{z}\log\frac{1+z}{1-z}\notag\\={}&\frac{2}{(1-z^2)^2}\mathbf K\left( \frac{1+z^2}{1-z^2} \right)+\frac{4}{\pi}\int_0^1\frac{(1-z^{4})\kappa\mathbf  K(\sqrt{1-\kappa ^2}) }{  [(1+z^2)^2 \kappa ^2-(1-z^2)^2]\sqrt{(1+z^2)^2 \kappa ^2-(1-z^2)^2}}\tanh^{-1}\frac{\sqrt{(1+z^2)^2 \kappa ^2-(1-z^2)^2}}{(1+z^2) \kappa}\D\kappa\tag{\ref{eq:Li2int'_iB}}\\={}&\int_0^1\frac{4\kappa (1-z ^4 ) \mathbf  K(\sqrt{1-\kappa ^2}) }{\pi[(1+z^{2})^2 \kappa ^2-4 z ^2]\sqrt{(1+z^{2})^2 \kappa ^2-4 z ^2}  }\left[ \log\frac{1}{\sqrt{1-\kappa^{2}}} +\log\frac{\sqrt{(1+z^{2})^2 \kappa ^2-4 z ^2}+(1-z ^2) \kappa}{\sqrt{(1+z^{2})^2 \kappa ^2-4 z ^2}+(1+z ^2) \kappa}\right]\D\kappa\notag\\&+\int_0^1\frac{8z^{2} \mathbf  K(\sqrt{1-\kappa ^2}) }{\pi[(1+z^{2})^2 \kappa ^2-4 z ^2]  }\D\kappa,\quad |z|<1,\I z\neq0.\tag{\ref{eq:Li2int'_B}}\end{align}}\noindent\hrule\begin{align}&z \dilog(z)-z \dilog(-z)-(1+z) \log (1+z)-(1-z) \log (1-z)\notag\\={}&\frac{1}{2}\int_0^1 \mathbf K( \sqrt{1-t} )\log \frac{\sqrt{(1-z^{2})^{2}+4z^{2}t}+2 t+z^2-1}{2 t}\D t\tag{\ref{eq:Li2int_S}}\\={}&\int_0^1\frac{\mathbf  K(\sqrt{1-\kappa ^2}) }{4 \pi  \kappa ^2}\left\{\log ^2(1-\kappa ^2)-4\left[ \log\frac{1}{\sqrt{1-\kappa^{2}}} +\log\frac{\sqrt{(1+z^{2})^2 \kappa ^2-4 z ^2}+(1-z ^2) \kappa}{\sqrt{(1+z^{2})^2 \kappa ^2-4 z ^2}+(1+z ^2) \kappa} \right]^{2}\right\}\D\kappa,\quad|z|\leq1.\tag{\ref{eq:Li2int_S_B}}\end{align}\noindent\hrule\begin{align}\frac{\pi^{2}}{4}+\log z\log\frac{1+z}{1-z}={}&\int_0^1
\frac{ (1-z^{2} )\mathbf K( \sqrt{1-t} )}{2\sqrt{4z^{2}+t(1-z^{2})^{2}}}\D t+\int_0^1\frac{ z\mathbf K( \sqrt{1-t} )}{\sqrt{(1-z^{2})^{2}+4z^{2}t}}\D t,\quad |z|<1.\tag{\ref{eq:chi2_1}}\\\frac{\pi^2}{2}+2i\theta\log\left( i \tan\frac{\theta}{2}\right)={}&\int_0^1
\frac{ \mathbf K( \sqrt{1-t} )}{\sqrt{t-\sin^2\theta}}\D t+i\int_0^1\frac{ \mathbf K( \sqrt{1-t} )\sin\theta}{\sqrt{1-t\sin^2\theta}}\D t,\quad 0\leq\theta\leq\frac{\pi}{2}.\tag{\ref{eq:chi2_2}}\\2\theta\log\tan\frac{\theta}{2}={}&-\int_0^{\sin^2\theta}
\frac{ \mathbf K( \sqrt{1-t} )}{\sqrt{\sin^2\theta-t}}\D t+\int_0^1\frac{ \mathbf K( \sqrt{1-t} )\sin\theta}{\sqrt{1-t\sin^2\theta}}\D t,\quad 0\leq\theta\leq\frac{\pi}{2}.\tag{\ref{eq:Im_ext}}\end{align}\noindent\hrule\begin{align}\pi\arcsin x={}&\int_0^{x^{2}}
\frac{ \mathbf K( \sqrt{t} )}{\sqrt{x^{2}-t}}\D t=\frac4\pi\int_0^1 \frac{\mathbf K( \sqrt{1-\kappa^{2}})\arcsin(\kappa x)}{\kappa  \sqrt{1- \kappa ^2x^2}}\D\kappa,\quad0\leq x\leq1;\tag{\ref{eq:K_arcsin}, \ref{eq:K_arcsin'}}\\\pi\arcsin x={}&\frac4\pi\int_0^1 \frac{\mathbf K( \sqrt{1-\kappa^{2}})}{1-\kappa^{2}}\left[\sinh^{-1}\left(\frac{x}{\sqrt{1-x^2}}\right)-\frac{\kappa}{\sqrt{1-(1-\kappa ^2 )x^2}} \sinh^{-1}\left(\frac{\kappa  x}{\sqrt{1-x^2}}\right  )\right]\D\kappa,\quad0\leq x<1;\tag{\ref{eq:K_arsinh''}$ \equiv$\ref{eq:K_arcsin_diamond}}\\\frac{\pi}{\sqrt{1-x^{2}}}={}&\frac4\pi\int_0^1 \left[1-\frac{\kappa x}{\sqrt{1-(1-\kappa^2)x^2}}\tanh^{-1}\frac{\kappa x}{\sqrt{1-(1-\kappa^2)x^2}}\right]\frac{\mathbf K( \sqrt{1-\kappa^{2}})\D\kappa}{1-(1-\kappa^2)x^2}\tag{\ref{eq:K_arcsin_iB_deriv}}\\={}&\frac{2
\mathbf K(x)}{\sqrt{1-x^2}}-\frac4\pi\int_0^1\frac{\kappa x\mathbf K( \sqrt{1-\kappa^{2}})}{[1-(1-\kappa^2)x^2]^{3/2}}\tanh^{-1}\frac{\kappa x}{\sqrt{1-(1-\kappa^2)x^2}}\D\kappa,\quad 0\leq x< 1.\tag{\ref{eq:K_arcsin_iB_deriv}}\end{align}\noindent\hrule\begin{align}\pi \sinh^{-1}x=\pi \log(x+\sqrt{1+x^2})={}&\int_0^{x^{2}}
\frac{ 1}{\sqrt{ 1+t}\sqrt{x^{2}-t}}\mathbf K\left( \sqrt{\frac{t}{1+t}}\right )\D t\tag{\ref{eq:K_arsinh}}\\={}&\frac4\pi\int_0^1 \frac{\mathbf K( \sqrt{1-\kappa^{2}})}{1-\kappa^{2}}\left[\arcsin\left(\frac{x}{\sqrt{1+x^2}}\right)-\frac{\kappa}{\sqrt{1+(1-\kappa ^2 )x^2}} \arcsin\left(\frac{\kappa  x}{\sqrt{1+x^2}}\right  )\right]\D\kappa\tag{\ref{eq:K_arsinh'}}\\={}&\frac4\pi\int_0^1 \frac{\mathbf K( \sqrt{1-\kappa^{2}})\log( \kappa x+\sqrt{1+\kappa ^{2}x^2})}{\kappa  \sqrt{1+ \kappa ^2x^2}}\D\kappa,\quad \forall x\geq0;\tag{\ref{eq:K_arcsin''}$ \equiv$\ref{eq:K_arsinh_diamond}}\\\frac{\pi}{\sqrt{1+x^{2}}}={}&\frac4\pi\int_0^1 \left(1-\frac{\kappa x}{\sqrt{1+\kappa^{2}x^2}}\tanh^{-1}\frac{\kappa x}{\sqrt{1+\kappa^{2}x^2}}\right)\frac{\mathbf K( \sqrt{1-\kappa^{2}})\D\kappa}{1+\kappa^{2}x^2{}}\tag{\ref{eq:K_arsinh_iB_deriv}}\\={}&\frac{2\mathbf K(x/\sqrt{1+x^2})}{\sqrt{1+x^{2}}}-\frac4\pi\int_0^1\frac{\kappa x\mathbf K( \sqrt{1-\kappa^{2}})}{(1+\kappa^{2}x^2)^{3/2}}\tanh^{-1}\frac{\kappa x}{\sqrt{1+\kappa^{2}x^2}}\D\kappa,\quad \forall x\geq0.\tag{\ref{eq:K_arsinh_iB_deriv}}
\end{align}\noindent\hrule\begin{align}
\frac{\pi\mathbf K(|r|)}{8}-\frac{\pi^2}{16(1+r)}={}&\frac{1}{1+r}\int_{0}^{1}\frac{\mathbf K(k)rk\sqrt{1-k^2}}{1-2r(1-2k^2)+r^2}\D k\tag{\ref{eq:K_reprod}}\\={}&\int_{0}^{1}\frac{2\mathbf K(\sqrt{1-\kappa^2})[ \kappa(1+r)  \sqrt{r}  \tanh ^{-1}\sqrt{r}-r \sqrt{1-\kappa ^2}  \arcsin \kappa ]}{\pi  (1+r) \kappa  [4 r+(1-r)^2 \kappa ^2]}\D\kappa,\quad -1<r<1;\tag{\ref{eq:K_reprod'}}\\\frac{\pi\mathbf K(r)}{8}-\frac{\pi^2}{16(1+r)}={}&\frac{\tanh^{-1}\sqrt{r}}{2(1+r)}\mathbf K\left( \frac{1-r}{1+r} \right)-\frac{2r}{\pi(1+r)}\int_{0}^{1}\frac{ \sqrt{1-\kappa ^2}  \arcsin \kappa\mathbf K(\sqrt{1-\kappa^2})}{ \kappa  [4 r+(1-r)^2 \kappa ^2]}\D\kappa,\quad 0<r<1.\tag{\ref{eq:K_reprod'}}
\end{align}\noindent\hrule

\begin{align}&\int_0^1\frac{\mathbf K(k)}{1-k^2}\left[ 1-\frac{k}{\sqrt{k^{2}+a^2(1-k^2)}} \right]\D k=\begin{cases}\dfrac{\arcsin a}{2}\left( \pi-\arcsin a \right), & 0\leq a\leq1 \\[8pt]\dfrac{\pi ^2}{8}+
\dfrac{\log ^2(\sqrt{a^2-1}+a)}{2} , & a>1 \\
\end{cases}\tag{\ref{eq:Abel1}}\\={}&\frac{2}{\pi}\int_{0}^{1}\frac{\mathbf K(\sqrt{1-\kappa^2})}{1-\kappa^2}\left[ \log(1+a)-\kappa\tanh^{-1}\kappa-\frac{\kappa}{\sqrt{1-(1-\kappa^{2})a^{2}}}\log\frac{\sqrt{1-\kappa^2}[a \kappa +\sqrt{1-(1-\kappa^{2})a^{2}}]}{ \kappa +\sqrt{1-(1-\kappa^{2})a^{2}}}\right]\D\kappa.\tag{\ref{eq:Abel1'}}\end{align}\noindent\hrule\begin{align}&\frac{\pi}{8}\left[ (1+2a^2)\log(a+\sqrt{1+a^2}) -a\sqrt{1+a^{2}}\right]=\int_0^{\frac{a}{\sqrt{1+a^2}}}\frac{k\mathbf K(k)}{(1-k^2)^{2}}\sqrt{a^{2}-(1+a^2)k^2}\D k\tag{\ref{eq:Abel2}}\\={}&\frac{2}{\pi}\int_{0}^{1}\frac{\mathbf K(\sqrt{1-\kappa^2})}{1-\kappa^2}\left[ \frac{(1+a^{2}+\kappa^{2}-a^2 \kappa ^2) \arctan a}{2(1-\kappa^{2})}-\frac{a}{2}-\frac{\kappa  \sqrt{1+(1-\kappa^{2})a^{2}} }{1-\kappa^{2}}\arctan\frac{a \kappa }{\sqrt{1+(1-\kappa^{2})a^{2}}}\right]\D\kappa,\quad \forall a\geq0.\tag{\ref{eq:Abel2'}}\end{align}\noindent\hrule\begin{align}\int_0^{\frac{a}{\sqrt{1+a^2}}}\frac{k\mathbf K(k)}{(1-k^2)^{3}}[{a^{2}-(1+a^2)k^2}]^{3/2}\D k={}&\frac{3\pi}{128}   \left[(8 a^4+8 a^2+3) \log (a+\sqrt{1+a^2})-3 a(1+2a^2)\sqrt{1+a^{2}}\right],\quad \forall a\geq0.\tag{\ref{eq:Abel3}}\end{align}\noindent\hrule\begin{align}&\int_0^{\frac{a}{\sqrt{1+a^2}}}\frac{k\mathbf K(k)}{(1-k^2)^{2}}\left[a  \sqrt{{a^{2}-(1+a^2)k^2}}-\frac{k^2}{\sqrt{1-k^2}} \log \frac{\sqrt{{a^{2}-(1+a^2)k^2}}+a \sqrt{1-k^2}}{k}\right]\D k\notag\\={}&\pi \left[\frac{2}{9}-\frac{1}{36} \sqrt{1+a^2} (8+5 a^2)+\frac{a (3+2 a^2)}{12}  \log (a+\sqrt{1+a^2})\right],\quad \forall a\geq0.\tag{\ref{eq:Abel4}}\end{align}\noindent\hrule

\begin{align}&\frac1\pi\int_{0}^1\frac{k\tanh^{-1}\sqrt{\frac{1-k^2}{1-k^2\sin^2\theta}}}{(1-k^{2})\sqrt{k^{2}+(1-k^2)\sec^2\theta}}\mathbf K(k)\D k=\frac{\theta}{2}\left(\frac{i\pi}{2}-\log\tan\frac{\pi-2\theta}{4}\right)-\frac{i[\dilog(ie^{-i\theta})-\dilog(-ie^{-i\theta})]}{2},\quad 0\leq\theta<\frac{\pi}{2}.\tag{\ref{eq:Abel_Li2}}\\&
\frac{1}{\pi}\int_0^1\left[ \frac{\tanh ^{-1}\sqrt{1-k^2}}{k}-\frac{ \log \cos \theta}{k\sqrt{1-k^{2}}}-\frac{1}{k\sqrt{1-k^2 \sin ^2\theta }} \tanh ^{-1}\sqrt{\frac{1-k^2}{1-k^2 \sin ^2\theta }} \right]\mathbf K(k)\D k=\frac{\theta^2}{4},\quad 0\leq\theta<\frac{\pi}{2}.\tag{\ref{eq:Abel_Li2_int}}
\end{align}\noindent\hrule

\begin{align}\mathbf K(k)={}&\frac{2}{\pi}\int_0^1\frac{\mathbf K(\sqrt{1-\kappa^2})}{1-k^2\kappa^2}\D\kappa=\frac{2}{\pi}\int_0^1\frac{\sqrt{1-k^2}\mathbf K(\sqrt{1-\kappa^2})}{1-k^2(1-\kappa^2)}\D\kappa\tag{\ref{eq:Beltrami}, \ref{eq:iB}}\\={}&\frac{2}{\pi}\int_0^1\frac{(1+k)\mathbf K(\sqrt{1-\kappa^2})}{(1+k)^{2}-4k\kappa^2}\D \kappa=\frac{2}{\pi}\int_0^1\frac{(1-k)\mathbf K(\sqrt{1-\kappa^2})}{(1-k)^{2}+4k\kappa^2}\D \kappa,\quad 0<k<1.\tag{\ref{eq:LB}, \ref{eq:iLB}}\end{align}\noindent\hrule\begin{align}
&\frac{\pi}{2}\frac{1}{\sqrt{1+r^{2}}}\mathbf K\left( \frac{r}{\sqrt{1+r^2}} \right)-\frac{r}{{1+r^{2}}}\,{_3F_2}\left( \left.\begin{array}{c}
1,1,1 \\
\frac{3}{2},\frac{3}{2} \\
\end{array}\right|\frac{r^2}{1+r^2}\right)=\int_0^1\frac{r\mathbf K(\xi)\D\xi}{r^{2}+\xi^2}=\int_0^1\frac{\mathbf K(\sqrt{1-\kappa^2})\D\kappa}{\sqrt{1+r^2}+\kappa r}\tag{\ref{eq:W_3F2}, \ref{eq:GB_1}}\\={}&\frac{2}{\pi}\int_0^1\left(\frac{ \kappa r  \arccos\kappa }{\sqrt{1-\kappa ^2} }-\sqrt{1+r^{2}}\log \kappa  \right)\frac{\mathbf K(\sqrt{1-\kappa^2})\D\kappa}{1+r^{2}-\kappa^2}-\frac{\log \left(1+\frac{r}{\sqrt{1+r^{2}}}\right)}{\sqrt{1+r^{2}}}\mathbf K\left( \frac{1}{\sqrt{1+r^2}} \right)\tag{\ref{eq:GB_2}}\\={}&\frac{2}{\pi}\int_0^1\left( \frac{ \kappa \sqrt{1+r^2}  \arcsin\kappa }{\sqrt{1-\kappa ^2} }+r\log \sqrt{1-\kappa^{2}}  \right)\frac{\mathbf K(\sqrt{1-\kappa^2})\D\kappa}{r^{2}+\kappa^2}+\frac{\log \left(1+\frac{r}{\sqrt{1+r^{2}}}\right)}{\sqrt{1+r^{2}}}\mathbf K\left( \frac{1}{\sqrt{1+r^2}} \right)\tag{\ref{eq:GB_3}}\\={}&\frac{2}{\pi}\int_0^1r\log \frac{\sqrt{1-\kappa^{2}}}{\kappa}\frac{\mathbf K(\sqrt{1-\kappa^2})\D\kappa}{r^{2}+\kappa^2}+\frac{\log \frac{r}{\sqrt{1+r^{2}}}}{\sqrt{1+r^{2}}}\mathbf K\left( \frac{1}{\sqrt{1+r^2}} \right),\quad \forall r>0.\tag{\ref{eq:GB_4}}
\end{align}\noindent\hrule
\begin{align}&\frac{i\sin\theta}{\cos^2\theta}\,{_3F_2}\left( \left.\begin{array}{c}
1,1,1 \\
\frac{3}{2},\frac{3}{2} \\
\end{array}\right|-\tan^2\theta\right)+\frac{\cos\theta}{\sin^2\theta}\,{_3F_2}\left( \left.\begin{array}{c}
1,1,1 \\
\frac{3}{2},\frac{3}{2} \\
\end{array}\right|-\cot^2\theta\right)+\left(\frac{i}{{\sin\theta}}+\frac{1}{\cos\theta}\right)\,{_3F_2}\left( \left.\begin{array}{c}
1,1,1 \\
\frac{3}{2},\frac{3}{2} \\
\end{array}\right|\frac{1}{1-e^{4i\theta}}\right)\notag\\={}&\frac{\pi[\mathbf K(\sin\theta)+i\mathbf K(\cos\theta)]}{2},\quad 0<\theta<\frac{\pi}{2}.\tag{\ref{eq:3F2_sum_rule}}\end{align}
\noindent\hrule\begin{align}\frac{r}{{1+r^{2}}}\,{_3F_2}\left( \left.\begin{array}{c}
1,1,1 \\
\frac{3}{2},\frac{3}{2} \\
\end{array}\right|\frac{r^2}{1+r^2}\right)=
{}&\int_0^1\frac{r\xi\mathbf K(\xi)\D\xi}{1+r^{2}\xi^2}=-\frac{2}{\pi}\int_0^1r\log \sqrt{1-\kappa^{2}}\frac{\mathbf K(\sqrt{1-\kappa^2})\D\kappa}{r^{2}+\kappa^2}+\frac{\log\sqrt{1+r^{2}}}{\sqrt{1+r^{2}}}\mathbf K\left( \frac{1}{\sqrt{1+r^2}} \right)\tag{\ref{eq:GB_5}, \ref{eq:M_3F2}}\\={}&-\frac{2}{\pi}\int_0^1\frac{ \kappa r  \arccos\kappa }{\sqrt{1-\kappa ^2} }\frac{\mathbf K(\sqrt{1-\kappa^2})\D\kappa}{1+r^{2}-\kappa^2}+\frac{\log (r+\sqrt{1+r^{2}})}{\sqrt{1+r^{2}}}\mathbf K\left( \frac{1}{\sqrt{1+r^2}} \right),\quad \forall r>0.\tag{\ref{eq:GB_6}}
\end{align} \noindent\hrule\begin{align}&\frac{\pi}{2}\frac{1}{\sqrt{1+r^{2}}}\mathbf K\left( \frac{1}{\sqrt{1+r^2}} \right)=\int_0^1\left(\frac{r}{r^{2}+\kappa^2}-\frac{\sqrt{1+r^2}}{1+r^{2}-\kappa^2}\right)\mathbf K\left( \sqrt{1-\kappa^2} \right)\frac{\log\kappa}{\log \frac{r}{\sqrt{1+r^{2}}}}\D\kappa\tag{\ref{eq:log_combo}}\\={}&\int_0^1\frac{r\mathbf K(\xi)\D\xi}{r^{2}+\xi^2}+\int_0^1\frac{r\xi\mathbf K(\xi)\D\xi}{1+r^{2}\xi^2}=\frac{2}{\pi}\int_0^1r\log \frac{r}{\kappa}\frac{\mathbf K(\sqrt{1-\kappa^2})\D\kappa}{r^{2}+\kappa^2}=\frac{2}{\pi}\int_0^1\sqrt{1+r^{2}}\log \frac{\sqrt{1+r^{2}}}{\kappa}\frac{\mathbf K(\sqrt{1-\kappa^2})\D\kappa}{1+r^{2}-\kappa^2},\quad \forall r>0.\tag{\ref{eq:sum_rule}}\end{align}
\noindent\hrule\begin{align}
\mathbf K(\sqrt{\lambda})\mathbf K(\sqrt{1-\lambda})={}&\int^1_{(1-\lambda)/(1+\lambda)}\frac{\mathbf K(k)\D k}{\sqrt{1-k^{2}}\sqrt{(1+\lambda)^{2}k^2-(1-\lambda)^2}}=\int_0^1\frac{\mathbf K(k)\D k}{\sqrt{4\lambda+(1-\lambda)^2k^2}}+\int_0^1\frac{\mathbf K(k)\D k}{\sqrt{(1-\lambda)^2+4\lambda k^2}}\tag{\ref{eq:KK'_various}}\\={}&\frac{2}{\pi}\int_0^1\tanh^{-1}\frac{2\sqrt{\lambda}}{\sqrt{4\lambda+(1-\lambda)^{2}\kappa^2}}\frac{\mathbf K(\sqrt{1-\kappa^2})\D\kappa}{\sqrt{4\lambda+(1-\lambda)^{2}\kappa^2}}\tag{\ref{eq:KK'_various}}\\
={}&\frac{2}{\pi}\int_0^1\tanh^{-1}\frac{1-\lambda}{\sqrt{(1-\lambda)^{2}+4\lambda\kappa^2}}\frac{\mathbf K(\sqrt{1-\kappa^2})\D\kappa}{\sqrt{(1-\lambda)^{2}+4\lambda\kappa^2}},\quad 0<\lambda<1.
\tag{\ref{eq:KK'_various}}
\end{align}\noindent\hrule\vspace{.2em}\hrule

 \newpage
\bibliography{Spheres}
\bibliographystyle{unsrt}

\end{document}